\pdfoutput=1

\documentclass[11pt]{article}
\usepackage[letter, center, frame]{crop}
\usepackage{amssymb, amsmath, amsthm, latexsym, fullpage}
\usepackage{eucal, fullpage, setspace}
\usepackage{ytableau}
\usepackage{undertilde}
\usepackage{cancel}
\usepackage{caption}
\usepackage[pdftex]{graphicx}
\usepackage{changepage}
\usepackage[top=.95in,bottom=1in,left=1in,right=1in]{geometry}

\newtheorem{thm}{Theorem}[section]
\newtheorem{cor}[thm]{Corollary}
\newtheorem{lem}[thm]{Lemma}
\newtheorem{prop}[thm]{Proposition}

\newtheorem{fact}[thm]{Fact}
\newtheorem{prob}[thm]{Problem}
\theoremstyle{definition}
\newtheorem{defn}[thm]{Definition}
\theoremstyle{remark}

\numberwithin{equation}{section}

\begin{document}

\title{Parabolic Catalan numbers count flagged Schur functions; Convexity of tableau sets for Demazure characters}

\author{Robert A. Proctor \\ University of North Carolina \\ Chapel Hill, NC 27599 U.S.A. \\ rap@email.unc.edu \and Matthew J. Willis \\ Wesleyan University \\ Middletown, CT 06457 U.S.A. \\ mjwillis@wesleyan.edu}

\maketitle

\vspace{-1.5pc}

\begingroup
    \fontsize{10pt}{12pt}\selectfont
    \begin{spacing}{1.1}

\centerline{\textbf{Abstract}}

\begin{adjustwidth*}{.5in}{.5in}

\noindent Shuffles of cards are $n$-multipermutations with suit multiplicities specified by some subset $R$ of $\{1, ... ,$ $n-1 \}$.  Their ``inverses'' are ordered partitions of $\{1, ... , n \}$ whose block sizes derive from $R$; these are essentially our ``$R$-permutations''.  The $R$-permutations depict the minimum length coset representatives in $W^J$ for the quotient of $S_n$ by the parabolic subgroup $W_J$, where $J$ is the complement of $R$.  We refer to those that blockwise avoid the pattern 312 as ``312-avoiding $R$-permutations'' and define the ``parabolic $R$-Catalan number'' to be the number of them.  When $R = \{1, ... , n-1 \}$ this is the usual Catalan number, which counts 312-avoiding permutations.  Let $\lambda$ be a partition of $N$ with at most $n$ parts whose set of shape column lengths less than $n$ is $R$.  We show that the number of distinct flagged Schur functions formed on the shape of $\lambda$ is the parabolic $R$-Catalan number, and list over a dozen other kinds of $n$-tuples and phenomena concerning flagged Schur functions and Demazure characters that are also enumerated by this quantity.  (Godbole, Goyt, Herdan, and Pudwell had independently just introduced such a notion of pattern avoiding for ordered partitions and had launched the study of their enumeration.)  Let $\pi$ be an $R$-permutation.  We view the Demazure character (key polynomial) indexed by $(\lambda, \pi)$ as the sum of the content weight monomials for our ``$\pi$-Demazure'' semistandard Young tableaux of shape $\lambda$ with entries from $\{ 1, ... , n \}$.  We show that the set of these tableaux is convex in $\mathbb{Z}^N$ if and only if $\pi$ is a 312-avoiding $R$-permutation.  As usual, a flagged Schur function is defined to be the sum of the content weight monomials for the semistandard tableaux of shape $\lambda$ whose entries are row-wise bounded by a given weakly increasing $n$-tuple.  We consider more general ``row bound sums'' for which the row bounds may form any $n$-tuple.  Reiner and Shimozono and then Postnikov and Stanley obtained results concerning coincidences between flagged Schur functions and Demazure characters:  when $\lambda$ is strict, the flagged Schur functions exactly coincide with the 312-avoiding Demazure characters.  For general $\lambda$, we introduce more precise indexing sets of $n$-tuple bounds for the row bound sums.  These indexing schemes and the convexity results are used to sharpen their coincidence results and to extend them to general row bound sums.  Now their coincidences are precisely indexed and are shown to hold at the deeper level of coinciding underlying tableau sets.  The most efficient indexing $n$-tuples for the row bound sums that can arise as flagged Schur functions are the new ``gapless $R$-tuples'';  these bijectively arise from the 312-avoiding $R$-permutations via the application of an ``$R$-ranking'' map.

\end{adjustwidth*}

\end{spacing}

\endgroup

\vspace{.5pc}\noindent\textbf{Keywords.}  Catalan number, Flagged Schur function, Demazure character, Key polynomial, Pattern avoiding permutation, Symmetric group parabolic quotient

\vspace{1pc}\noindent\textbf{MSC Codes.}  {05A05, 05E05, 05E10, 17B10, 14M15}

\begin{figure}[h!]
\begin{center}
\def\arraystretch{1.2}
\begin{tabular}{clc}
\underline{\S\S} & \underline{Content} & \underline{page}\\
2; 3 & Definitions: general;  $R$-tuples & 8 \\
4; 5 & $R$-tuple preliminaries for row bounds & 13 \\
6; 7 & $R$-tuple preliminaries for $R$-permutations & 18 \\
8; 9 & Tableau definitions and preliminary results & 24 \\
10; 11 & Convexity for Demazure tableaux & 28 \\
12 & Row bound tableau sets & 33 \\
13; 14 & Coincidences of tableau sets and polynomials & 35 \\
15; 16 & Earlier work; Polynomial distinctness & 40 \\
17; 18 & Gessel-Viennot determinant; Enumeration & 43
\end{tabular}
\end{center}
\end{figure}

\vspace{-3pc}

\begin{spacing}{1.35}

\section{Introduction}

This paper can be read by anyone interested in tableaux.  Except for a few references to its tableau precursors \cite{Wi2} and \cite{PW1} (and a few motivational remarks), it is self-contained.  Fix $n \geq 1$ and set $[n-1] := \{1, 2, ... , n-1\}$.  Choose a subset $R \subseteq [n-1]$ and set $r := |R|$.  The section on $R$-parabolic Catalan numbers, the last section, has been written so that much of it can be understood independently when read in conjunction with this introduction.  The ``rightmost clump deleting'' chains of sets defined early in Section 6 (when $R = [n-1])$ are the recent addition Exercise 2.202 to Stanley's list \cite{St3} of interpretations of the Catalan numbers $C_n$.  Experimental combinatorialists may be interested in Problem \ref{prob14.5}.  Algebraic geometers may be interested in Problem \ref{prob16.1}.

Fix a partition $\lambda$ of $N \geq 1$ into no more than $n$ parts such that the lengths of the columns in its shape $\lambda$ that are less than $n$ form the set $R$.  Let $\mathcal{T}_\lambda$ be the set of semistandard tableaux on the shape $\lambda$ whose values come from $[n]$.  Flagged Schur functions (flag Schur polynomials) have been defined to be sums of the content weight monomial over certain subsets of $\mathcal{T}_\lambda$, and Demazure characters for $GL(n)$ (key polynomials) can also be viewed in this way.  Beginning in 2011, our original motivation for this project was to better understand results obtained by Reiner and Shimozono \cite{RS} and then by Postnikov and Stanley \cite{PS} concerning coincidences between these two families of polynomials in $x_1, x_2, ... , x_n$.  Demazure characters arose in 1974 when Demazure introduced certain $B$-modules while studying singularities of Schubert varieties in flag manifolds $G/P$.  Flag Schur polynomials arose in 1982 when Lascoux and Sch\"{u}tzenberger were studying Schubert polynomials for the flag manifold $GL(n)/B$.

The subset $R$ can be used to specify $r+1$ suit multiplicities for a deck of $n$ cards that is shuffled.  Given a shuffle, its ``inverse'' is the ordered partition of $[n]$ into $r+1$ blocks which list the positions occupied by the cards in the corresponding suits.  These ordered partitions are essentially our ``$R$-permutations''.  Setting $J := [n-1] \backslash R$, these objects depict the minimum length coset representatives in $W^J$ for the quotient $S_n / W_J$ of the symmetric group by the parabolic subgroup $W_J$.  In 2012 we generalized the notion of 312-pattern avoidance for permutations to that of ``$R$-312-avoidance'' for $R$-permutations.  More recently we defined the $R$-parabolic Catalan number $C_n^R$ to be the number of $R$-312-avoiding $R$-permutations.  We then learned that while this project had been underway, Godbole, Goyt, Herdan, and Pudwell had independently introduced \cite{GGHP} a more general notion of pattern avoidance for such ordered partitions, and that Chen, Dai, and Zhou had obtained further enumerative results \cite{CDZ} concerning them.  Giving what could be the first appearance of this count ``in nature'', we show that the number of flag Schur polynomials that can be formed on the shape $\lambda$ is $C_n^R$.  When $R = [n-1]$, the $R$-permutations are merely permutations and hence $C_n^{[n-1]} = C_n$.  For a shape $\lambda$ to be compatible with $R = [n-1]$, it must be strict (i.e. not have any repeated row lengths).

The content weight monomial $x^{\Theta(T)}$ of a tableau $T \in \mathcal{T}_\lambda$ is formed from the census $\Theta(T)$ of the values $1, 2, ... , n$ appearing in $T$.  Given a flag $1 \leq \varphi_1 \leq \varphi_2 \leq ... \leq \varphi_n \leq n$, the flag Schur polynomial $s_\lambda(\varphi;x)$ has been defined to be the sum of $x^{\Theta(T)}$ over the $T \in \mathcal{T}_\lambda$ whose values in its $i^{th}$ row do not exceed $\varphi_i$.  Since we also require the ``upper'' condition $\varphi \geq i$ to ensure nonvanishing, the number of indexing sequences is $C_n$.  We denote this set of tableau $\mathcal{S}_\lambda(\varphi)$.  As $R$ varies over subsets of $[n-1]$, the Demazure characters $d_\lambda(\pi;x)$ for $GL(n)$ (Demazure polynomials) are indexed by pairs $(\lambda, \pi)$ such that $\pi$ is an $R$-permutation and $\lambda$ is ``compatible'' with $R$.  These polynomials can be recursively specified with the divided difference formula cited in \cite{PW1}.  Taking advantage of the improvements made in \cite{Wi2} and \cite{PW1} upon a description of Lascoux and Sch\"{u}tzenberger, here we define $d_\lambda(\pi;x)$ to be a sum of $x^{\Theta(T)}$ over a certain subset $\mathcal{D}_\lambda(\pi) \subseteq \mathcal{T}_\lambda$.  Our terminology wording choices of `flag Schur polynomial' and `Demazure polynomial' are explained in Section 14 when these polynomials are defined.

To count flag Schur polynomials, it must be decided when to regard two of them as being ``the same''.  If by this it is meant that they are equal as polynomials (our first notion of ``sameness''), then the $C_n$ counting assertion in \cite{PS} on p. 158 may not seem to be correct:  For $n = 3$ and $\lambda = (1, 1, 0)$, note that $s_\lambda( (3,3,3);x) = x_1x_2 + x_1x_3 + x_2x_3 = s_\lambda( (2, 3, 3);x)$.  So the desired count of $C_3 = 5$ is unattainable.  Hence it must have been the case that they regarded these two flag Schur polynomials to be distinct since they are indexed by distinct $n$-tuples of row bounds.  In other words two such polynomials were to be regarded as being the same only when their indexing $n$-tuples were the same $n$-tuple;  this is a second notion of ``sameness''.

Most of the results in this paper are much more straightforward (or even trivial) when $R = [n-1]$, which corresponds to requiring $\lambda$ to be strict.  Then the distinct (by the first notion) flag Schur polynomials are indeed counted by $C_n$.  Permutations $\pi$ index the Schubert varieties $X(\pi)$ of the full flag manifold $GL(n)/B$, where $B$ is the Borel subgroup.  Most of the phenomena in which we are interested arise only when $R \subset [n-1]$, for which a compatible shape $\lambda$ must have at least one repeated row length.  Here the $R$-permutations $\pi$ index the Schubert varieties $X(\pi)$ of the partial flag manifold $GL(n)/P$, where $P$ is the parabolic subgroup specified by $R$ via $W_J \subseteq S_n$.

We consider a third notion of ``sameness'' for polynomials.  Suppose a family of polynomials is defined to consist of the sums of $x^{\Theta(T)}$ over various sets of semistandard tableau of various constant shapes.  If the polynomials $p(x)$ and $q(x)$ arise in this manner from sets $\mathcal{P}$ and $\mathcal{Q}$ of tableau of constant shape, then we say that $p(x)$ and $q(x)$ are ``identical as generating functions'' (and write $p(x) \equiv q(x)$) exactly when it can be shown that $\mathcal{P} = \mathcal{Q}$.  In contrast to the first notion of sameness, here the shape for the set $\mathcal{P}$ must \`{a} priori be the shape for the set $\mathcal{Q}$.

To generalize flag Schur polynomials, we introduce two sets $U_R(n) \supseteq UGC_R(n)$ of $n$-tuples that both contain the set $UF_R(n)$ of upper flags $\varphi$ described above.  The subscript `$R$' indicates that the locations of the ``dividers'' in these $n$-tuples are to be ``carried along'';  hence the elements of these three sets are certain kinds of ``$R$-tuples''.  In addition to the tableau sets $\mathcal{S}_\lambda(\varphi)$ for $\varphi \in UF_R(n)$ we also consider the tableau sets $\mathcal{S}_\lambda(\beta)$ for $\beta \in U_R(n)$ and $\mathcal{S}_\lambda(\eta)$ for $\eta \in UGC_R(n)$ that consist of the tableaux satisfying the row bounds $\beta$ or $\eta$.  Again summing $x^{\Theta(T)}$, the corresponding ``row bound sums'' and ``gapless core Schur polynomials'' are denoted $s_\lambda(\beta;x)$ and $s_\lambda(\eta;x)$.  The $s_\lambda(\beta;x)$ are quite general, since $U_R(n)$ is defined by imposing on $R$-tuples of row bounds only the upper requirement $\beta_i \geq i$ that is needed to ensure nonvanishing.  We develop precise indexing schemes for these three classes of row bound sums.  For two of these classes, these indexes enable us to give the count of $C_n^\lambda$ for flag (and gapless core) Schur polynomials that are distinct according to both polynomial equality and generating function identicality.

Reiner and Shimozono and then Postnikov and Stanley described polynomial coincidences of the form $s_\lambda(\varphi;x) = d_\lambda(\pi;x)$ for $\varphi \in UF_R(n)$ and 312-avoiding permutations $\pi$.  We extend their results by also considering the $s_\lambda(\beta;x)$ and the $s_\lambda(\eta;x)$ introduced above.  We sharpen their results by precisely specifying the $s_\lambda(\varphi;x)$, the $s_\lambda(\eta;x)$, and the $d_\lambda(\pi;x)$ that participate in these coincidences.  We deepen their results by showing that a coincidence such as $s_\lambda(\varphi;x) = d_\lambda(\pi;x)$ is actually manifested at the tableau level by $\mathcal{S}_\lambda(\varphi) = \mathcal{D}_\lambda(\pi)$, in other words $s_\lambda(\varphi;x) \equiv d_\lambda(\pi;x)$.  Two of our four main results, Theorems \ref{theorem721} and \ref{theorem737.2}, present our statements concerning coincidences.  The row bound sets $\mathcal{S}_\lambda(\beta)$ and sums $s_\lambda(\beta;x)$ that participate in such coincidences are those for which $\beta$ is a gapless core $R$-tuple, that is when $\beta := \eta \in UGC_R(n) \supseteq UF_R(n)$.  The Demazure sets $\mathcal{D}_\lambda(\pi)$ and polynomials $d_\lambda(\pi;x)$ that participate are those that are indexed by the $R$-312-avoiding $R$-permutations $\pi$.

Our two other main results, Theorems \ref{theorem520} and \ref{theorem420}, are perhaps our deepest results.  The set $\mathcal{T}_\lambda$ of tableaux is partially ordered by value-wise comparison, and it can be viewed as a subset of $\mathbb{Z}^N$.  Theorem \ref{theorem520} states that if the set $\mathcal{D}_\lambda(\pi)$ of Demazure tableaux is a principal ideal in $\mathcal{T}_\lambda$, or more generally if $\mathcal{D}_\lambda(\pi)$ is a convex polytope in $\mathbb{Z}^N$, then $\pi$ must be $R$-312-avoiding.  Theorem \ref{theorem420} states that if $\pi$ is $R$-312-avoiding, then the set $\mathcal{D}_\lambda(\pi)$ is a principal ideal in $\mathcal{T}_\lambda$ and hence is conversely a convex polytope in $\mathbb{Z}^N$.  These two theorems play central roles in proving Theorems \ref{theorem721} and \ref{theorem737.2}.  None of these theorems could be proved without being able to get one's hands on Demazure tableaux.  The scanning method developed in this second author's thesis \cite{Wi1} \cite{Wi2} for computing the right key of Lascoux and Sch\"{u}tzenberger is used in the proofs of Theorems \ref{theorem520} and \ref{theorem420};  the proof of the latter result also uses the more direct description of $\mathcal{D}_\lambda(\pi)$ we developed in \cite{PW1}.  Postnikov and Stanley noted on p. 162 of \cite{PS} that the Gelfand pattern conversions of the tableaux in $\mathcal{S}_\lambda(\varphi) = \mathcal{D}_\lambda(\pi)$ form a convex polytope in $\mathbb{Z}^M$, where $M$ is the length of the 312-avoiding permutation $\pi$.

Let us return to considering the Schubert varieties $X(\pi) \subseteq GL(n)/B$ and $X(\pi) \subseteq GL(n)/P$, where $P$ is the parabolic subgroup specified by $R$.  For each $\lambda$ that is compatible with $R$, the tableaux in $\mathcal{D}_\lambda(\pi)$ for a given $R$-permutation $\pi$ index a basis for a vector space that describes a projective embedding of $X(\pi)$.  The fact that the bases for all such $\lambda$ enjoy the tableau convexity property when $\pi$ is $R$-312-avoiding hints that the Schubert variety $X(\pi)$ might enjoy some nice geometric properties.  In fact, Postnikov and Stanley noted on p. 134 of \cite{PS} that the 312-avoiding permutations could be seen to be the Kempf elements of $S_n$ considered by Lakshmibai.  In \cite{GL} and earlier papers she showed that the varieties $X(\pi) \subseteq GL(n)/B$ indexed by the Kempf elements $\pi$ did possess special geometric properties.

When we first considered the row bound sums $s_\lambda(\beta;x)$, it seemed needlessly restrictive to require $\beta_1 \leq \beta_2 \leq ... \leq \beta_n$ for the row bound sequence.  But as we proceeded we found it difficult to say much about the $s_\lambda(\beta;x)$ when we required only $\beta_i \geq i$ and $\beta_i \leq n$ for the $R$-tuples forming $U_R(n)$.  This led us to define the third, intermediate, set $U_R(n) \supseteq UGC_R(n) \supseteq UF_R(n)$ mentioned above.  For $\beta, \beta' \in U_R(n)$, we define $\beta \approx_\lambda \beta'$ when $\mathcal{S}_\lambda(\beta) = \mathcal{S}_\lambda(\beta')$.  After we describe this equivalence and its equivalence classes in Sections 12 and 5, in Proposition \ref{prop623.4} we precisely index the tableau sets that underly our three kinds of row bound sums.  We refer to the $R$-tuple indexes we have chosen for the gapless core bound tableau sets $\mathcal{S}_\lambda(\eta)$ as ``gapless $R$-tuples'' and we gather them into a set denoted $UG_R(n)$.  These are the minimal row bounds that can be used to describe both the $\mathcal{S}_\lambda(\eta)$ and the flag bound tableau sets $\mathcal{S}_\lambda(\varphi)$; these $R$-tuples appear to have fundamental importance.

Our characterization of $\approx_\lambda$ in Proposition \ref{prop623.2} and Lemma \ref{lemma608.2}(i) describes when one can expect $s_\lambda(\beta;x)$ and $s_\lambda(\beta';x)$ to ``obviously'' be equal because their underlying tableau sets are the same.  If $s_\lambda(\beta;x) = s_\lambda(\beta';x)$ while $\mathcal{S}_\lambda(\beta) \neq \mathcal{S}_\lambda(\beta')$, we say that these two row bound sums are ``accidentally'' equal.  Corollary \ref{newcor737} rules out an accidental equality between a gapless core Schur polynomial $s_\lambda(\eta;x)$ and any row bound sum $s_\lambda(\beta;x)$ for $\beta \in U_\lambda(n)$.  The proof uses our characterization of coincidences between row bound sums and Demazure polynomials (which depends upon Theorem \ref{theorem520}) to refer to the known distinctness of the Demazure polynomials.  Table 16.1 summarizes our results that say when the polynomials we are studying coincide in either sense and when (using precise indexing) they are distinct;  counts for the equivalence classes of these polynomials are also given.  Problem \ref{prob14.5} asks if there exist accidental equalities among the row bound sums that are not gapless core Schur polynomials.  Extending from flags to gapless core $R$-tuples allows us to use $R$-tuples with smaller entries to serve as row bounds for the same set of tableaux.  However, for every $\eta \in UGC_R(n)$ there exists an equivalent $\varphi \in UF_R(n)$.  Therefore every gapless core Schur polynomial has already arisen as a flag Schur polynomial.  But knowledge of $UGC_R(n)$ and $UG_R(n)$ provides a clearer picture and more efficient row bounds.

Counting polynomials has revealed a new coincidence.  We use ${n \choose R}$ to denote the multinomial coefficient that counts $R$-permutations.  The total number of Demazure polynomials based upon the shape $\lambda$ is ${n \choose R}$.  The number of these that arise as flag Schur polynomials is the parabolic Catalan number $C_n^R$.  The so-counted $R$-312-avoiding Demazure polynomials match up with the flag Schur polynomials.  Hence Theorems \ref{theorem721} and \ref{theorem737.2} provide a complete explanation for this first counting coincidence, between the nicest Demazure polynomials and the nicest row bound sums.  However, up to generating function identicality, the number of our most general class of row bound sums $s_\lambda(\beta)$ also happens to be ${n \choose R}$.  Theorem \ref{theorem737.2} says that those that are not flag Schur polynomials cannot arise as Demazure polynomials.  For a fixed compatible shape $\lambda$, in Problem \ref{prob16.1} we ask why the number of row bound sums that are not Demazure polynomials and the number of Demazure polynomials that are not flag Schur polynomials should both be ${n \choose R} - C_n^R$.

When Stanley expressed some row bound sums $s_\lambda(\beta;x)$ with a Gessel-Viennot determinant in Theorem 2.7.1 of \cite{St1} and Theorem 7.5.1 of \cite{St2}, he noted that taking $\beta$ to be a flag would satisfy a requirement that had been stipulated by Gessel and Viennot for employing their method.  This implicitly raised the problem of characterizing all $\beta \in U_R(n)$ for which the Gessel-Viennot method can be applied to produce a determinant expression for a row bound sum $s_\lambda(\beta;x)$.  Using concepts that had already been developed for this paper, in \cite{PW2} we characterize such $\beta$.  This is previewed in Section 17.  Half of this characterization consists of the requirement $\beta \in UGC_R(n)$.  In fact, using the smallest equivalent bound sequences from $UG_R(n)$ (rather than those from $UF_R(n)$) produces determinants whose evaluations use the fewest possible number of monomials.

Relating our Theorems \ref{theorem721} and \ref{theorem737.2} to Theorems 23 and 25 of \cite{RS} and to Theorem 14.1 of \cite{PS} takes significant effort in Section 15.  Fortunately the tools we develop in earlier sections suffice:  In Section 7 we use the maps of $R$-tuples that were defined and developed in Sections 3, 4, and 6 for other purposes to describe the relationship of the notion of $R$-312-avoiding $R$-permutation to that of 312-avoiding permutation.  Propositions \ref{prop824.2} and \ref{prop824.4} are then used in Section 15 to prove the equivalence of Theorem 14.1 of \cite{PS} with a weaker form of part of our Theorem \ref{theorem737.1}.  We are able to repair one direction of Theorem 25 of \cite{RS} and to extend the other direction to handle more cases.

Other tools are developed in Section 4-6, 9, and 12.  At times we re-express arbitrary $R$-permutations as $R$-chains of subsets of $[n]$ and as key tableaux of a compatible shape $\lambda$.  More specifically, it is useful to re-encode the information contained in an $R$-312-avoiding $R$-permutation into other forms.  Parts (i), (ii), (iii), and (v) of Theorem \ref{theorem18.1} list nine other sets of simple combinatorial structures that are also enumerated by $C_n^R$:  Given an $R$-312-avoiding $R$-permutation $\pi$, the unique corresponding upper gapless $R$-tuple $\gamma \in UG_R(n)$ provides the most efficient description of the set $\mathcal{D}_\lambda(\pi)$ in the form $\mathcal{S}_\lambda(\gamma)$.  Two particular kinds of flags $\varphi$, the floors and the ceilings, can be used to provide alternate precise labelling indexes for the flag bound tableau sets $\mathcal{S}_\lambda(\varphi) = \mathcal{S}_\lambda(\gamma)$ that arise here.  The notions of $R$-rightmost clump deleting chain for general $R$ and of gapless $\lambda$-key give two more ways to encode the information in such a $\pi$.  Propositions \ref{prop608.10} and \ref{prop320.2} and Theorem \ref{theorem340} describe bijections among the sets of these objects.  These bijections include the $R$-core map $\Delta_R$, which can more generally be applied to upper $R$-tuples, and the rank $R$-tuple map $\Psi_R$, which can more generally be applied to $R$-permutations.  These two maps play central roles throughout this paper.  It is striking that the gapless $R$-tuples arise in two independent fashions:  Not only are they the images of flags $\varphi$ under the $R$-core map $\Delta_R$, they are also the images of the $R$-312-avoiding $R$-permutations $\pi$ under the rank $R$-tuple map $\Psi_R$.  The equality $\Delta_R(\varphi) = \gamma = \Psi_R(\pi)$ with $\gamma \in UG_R(n)$ is the central aspect of the connection between flag Schur polynomials and $R$-312-avoiding Demazure polynomials.  Given $\beta \in U_R(n)$, a maximization process in \cite{RS} produced a tableau that we denote $Q_\lambda(\beta)$.  We introduce another maximization process in Section 9 to produce a tableau denoted $M_\lambda(\beta)$.  Proposition \ref{prop623.8} relates $M_\lambda(\beta)$ to $Q_\lambda(\beta)$.  For a gapless $R$-tuple $\gamma \in UG_R(n)$, Theorem \ref{theorem340} says that $M_\lambda(\gamma)$ is a $\lambda$-key $Y_\lambda(\pi)$ for an $R$-312-avoiding $R$-permutation $\pi$.  These two results provide the foundation for the bridge from flag bound tableau sets to the $R$-312-avoiding Demazure tableau sets.

Please be aware of the two notation conventions noted at the end of this paragraph!  The primary independent variable for each section is either a subset $R \subseteq [n-1]$ or a partition $\lambda$ with at most $n$ parts.  Many sections are accordingly said to be in the ``$R$-world'' or in the ``$\lambda$-world''.   The $R$-world is concerned with $R$-tuples and the $\lambda$-world is concerned with tableaux of shape $\lambda$.  If the independent variable is $\lambda$, then we soon find the set $R_\lambda$ of column lengths in its shape that are less than $n$ and take $R := R_\lambda$ when referring to $R$-world concepts and results.  At the end of Section 3 we say that the `$R$' subscripts and prefixes will be omitted when $R = [n-1]$.  Near the end of Section 8, we say that we will usually replace `$R_\lambda$' in subscripts and in prefixes with `$\lambda$'.

In addition to the $R$ versus $\lambda$ dichotomy, another overarching dichotomy in this paper is between the ``left hand side'' entities and results concerned with flag Schur polynomials and their row bound sum generalizations and the ``right hand side'' entities and results concerned with the Demazure polynomials.  Visualize a river that flows from north to south.  The northern portion of each bank lies in the preliminary $R$-world and the southern portion lies in the $\lambda$-world.  After presenting the definitions in Section 3 on an island in the $R$-world, in the northern portion of the left bank in Sections 4 and 5 we prepare to later index row bound sums.  We jump to the northern portion of the right bank and prepare to later index Demazure polynomials in Sections 6 and 7.  On an island in the $\lambda$-world, Section 8 presents the definitions concerning shapes and tableaux.  Back on the right bank, Section 9 transitions from the $R$-world down the river to the $\lambda$-world.  This prepares us to obtain in Section 10 and 11 our results on the convexity of the Demazure tableau sets.  After we jump back to the left bank and land in the $\lambda$-world, Section 12 prepares to build a bridge to the right bank.  The bridge primarily consists of Section 13 and 14, which contain our results on coincidences among, and distinctness for, the row bound sums and Demazure polynomials.  Also on the bridge, Section 15 compares our results to those of \cite{RS} and \cite{PS} and Section 16 summarizes our distinctness results.  Section 17 previews our further results in \cite{PW2} and Section 18 contains enumeration remarks.

To summarize:  How ``special'' are flag Schur polynomials compared to general row bound sums?  When naming the members of the collections $\{s_\lambda(\beta;x)\}_{\beta \in U_\lambda(n)}$ and $\{s_\lambda(\eta;x)\}_{\eta \in UGC_\lambda(n)}$ of newly defined polynomials that extend the collection $\{s_\lambda(\varphi;x)\}_{\varphi \in UF_\lambda(n)}$ of flag Schur polynomials, we decided to not honor the general row bound sums $s_\lambda(\beta;x)$ with the adjective `Schur'.  Recall that each flag Schur polynomial is a gapless core Schur polynomial.  We show that each gapless core Schur polynomial $s_\lambda(\eta;x)$ arises as a Demazure polynomial, we rule out accidental equalities between gapless core Schur polynomials, we can count gapless core Schur polynomials up to polynomial equality, and we show that each gapless core Schur polynomial can be expressed with a determinant.  We cannot show any of these things for the general row bound sums $s_\lambda(\beta;x)$.  Since by Proposition \ref{prop623.1} every gapless core Schur polynomial $s_\lambda(\eta;x)$ arises as a flag Schur polynomial $s_\lambda(\varphi;x)$ for some upper flag $\varphi$, after that proposition is obtained one could think of the gapless core Schur polynomials as being more-flexibly indexed versions of the flag Schur polynomials.  We view the larger indexing set $UGC_\lambda(n)$ of gapless core $\lambda$-tuples as being the most appropriate indexing set; in particular, as is noted in Corollary \ref{cor17.3} the gapless $\lambda$-tuples are the most efficient inputs for the determinant expression.

\section{General definitions}

In posets we use interval notation to denote principal ideals and convex sets.  For example, in $\mathbb{Z}$ one has $(i, k] = \{i+1, i+2, ... , k\}$.  Given an element $x$ of a poset $P$, we denote the principal ideal $\{ y \in P : y \leq x \}$ by $[x]$.  When $P = \{1 < 2 < 3 < ... \}$, we write $[1,k]$ as $[k]$.  If $Q$ is a set of integers with $q$ elements, for $d \in [q]$ let $rank^d(Q)$ be the $d^{th}$ largest element of $Q$.  We write $\max(Q) := rank^1(Q)$ and $\min(Q) := rank^q(Q)$.  A set $\mathcal{D} \subseteq \mathbb{Z}^N$ for some $N \geq 1$ is a \textit{convex polytope} if it is the solution set for a finite system of linear inequalities.

Fix $n \geq 1$ throughout the paper.  Except for $\zeta$, lower case Greek letters indicate $n$-tuples of non-negative integers; their entries are denoted with the same letter.  An $nn$-\textit{tuple} $\nu$ consists of $n$ \emph{entries} $\nu_i \in [n]$ that are indexed by \emph{indices} $i \in [1,n]$, which together form $n$ \emph{pairs} $(i, \nu_i)$.  Let $P(n)$ denote the poset of $nn$-tuples ordered by entrywise comparison.  It is a distributive lattice with meet and join given by entrywise min and max.  Fix an $nn$-tuple $\nu$.  A \emph{subsequence} of $\nu$ is a sequence of the form $(\nu_i, \nu_{i+1}, ... , \nu_j)$ for some $i, j \in [n]$.  The \emph{support} of this subsequence of $\nu$ is the interval $[i,j]$.  The \emph{cohort} of this subsequence of $\nu$ is the multiset $\{ \nu_k : k \in [i,j] \}$.  A \emph{staircase of $\nu$ within a subinterval $[i,j]$} for some $i, j \in [n]$ is a maximal subsequence of $(\nu_i, \nu_{i+1}, ... , \nu_j)$ whose entries increase by 1.  A \emph{plateau} in $\nu$ is a maximal constant nonempty subsequence of $\nu$;  it is \emph{trivial} if it has length 1.

An $nn$-tuple $\phi$ is a \textit{flag} if $\phi_1 \leq \ldots \leq \phi_n$.  The set of flags is a sublattice of $P(n)$;  it is essentially the lattice denoted $L(n,n)$ by Stanley.  An \emph{upper tuple} is an $nn$-tuple $\upsilon$ such that $\upsilon_i \geq i$ for $i \in [n]$.  The upper flags are the sequences of the $y$-coordinates for the above-diagonal Catalan lattice paths from $(0, 0)$ to $(n, n)$.  A \emph{permutation} is an $nn$-tuple that has distinct entries.  Let $S_n$ denote the set of permutations.  A permutation $\pi$ is $312$-\textit{avoiding} if there do not exist indices $1 \leq a < b < c \leq n$ such that $\pi_a > \pi_b < \pi_c$ and $\pi_a > \pi_c$.  Let $S_n^{312}$ denote the set of 312-avoiding permutations.  By Exercises 6.19(h) and 6.19(ff) of \cite{St2} (or Exercises 116 and 24 of \cite{St3}), these permutations and the upper flags are counted by the Catalan number $C_n := \frac{1}{n+1}{2n \choose n}$.

Tableau and shape definitions are in Section 8; polynomials definitions are in Section 14.

\section{Carrels, cohorts, $\mathbf{\emph{R}}$-tuples, maps of $\mathbf{\emph{R}}$-tuples}

Fix $R \subseteq [n-1]$ through the end of Section 7.  Denote the elements of $R$ by $q_1 < \ldots < q_r$ for some $r \geq 0$.  Set $q_0 := 0$ and $q_{r+1} := n$.  We use the $q_h$ for $h \in [r+1]$ to specify the locations of $r+1$ ``dividers'' within $nn$-tuples:  Let $\nu$ be an $nn$-tuple.  On the graph of $\nu$ in the first quadrant draw vertical lines at $x = q_h + \epsilon$ for $h \in [r+1]$ and some small $\epsilon > 0$.  These $r+1$ lines indicate the right ends of the $r+1$ \emph{carrels} $(q_{h-1}, q_h]$ \emph{of $\nu$} for $h \in [r+1]$.  An \emph{$R$-tuple} is an $nn$-tuple that has been equipped with these $r+1$ dividers.  Fix an $R$-tuple $\nu$;  we portray it by $(\nu_1, ... , \nu_{q_1} ; \nu_{q_1+1}, ... , \nu_{q_2}; ... ; \nu_{q_r+1}, ... , \nu_n)$.  Let $U_R(n)$ denote the sublattice of $P(n)$ consisting of upper $R$-tuples.  Let $UF_R(n)$ denote the sublattice of $U_R(n)$ consisting of upper flags.  Fix $h \in [r+1]$.  The $h^{th}$ carrel has $p_h := q_h - q_{h-1}$ indices.  The $h^{th}$ \emph{cohort} of $\nu$ is the multiset of entries of $\nu$ on the $h^{th}$ carrel.

An \emph{$R$-increasing tuple} is an $R$-tuple $\alpha$ such that  $\alpha_{q_{h-1}+1} < ... < \alpha_{q_h}$ for $h \in [r+1]$.  Let $UI_R(n)$ denote the sublattice of $U_R(n)$ consisting of $R$-increasing upper tuples.  It can be seen that $|UI_R(n)| = \prod_{h=1}^{r+1} {{n-q_{h-1}} \choose {p_h}} = n! / \prod_{h=1}^{r+1} p_h! =: {n \choose {p_1 \hspace{.2pc} \dots \hspace{.2pc} p_{r+1}}} =: {n \choose R}$.  An $R$-\textit{permutation} is a permutation that is $R$-increasing when viewed as an $R$-tuple.  Let $S_n^R$ denote the set of $R$-permutations.  Note that $| S_n^R| = {n \choose R}$.  We refer to the cases $R = \emptyset$ and $R = [n-1]$ as the \emph{trivial} and \emph{full cases} respectively.  Here $| S_n^\emptyset | = 1$ and $| S_n^{[n-1]} | = n!$ respectively.  Given a permutation $\sigma \in S_n$, its \emph{$R$-projection} $\bar{\sigma} \in S^R_n$ is the $R$-increasing tuple obtained by sorting its entries in each cohort into increasing order within their carrel.  An $R$-permutation $\pi$ is $R$-$312$-\textit{containing} if there exists $h \in [r-1]$ and indices $1 \leq a \leq q_h < b \leq q_{h+1} < c \leq n$ such that $\pi_a > \pi_b < \pi_c$ and $\pi_a > \pi_c$.  An $R$-permutation is $R$-$312$-\textit{avoiding} if it is not $R$-$312$-containing.  Let $S_n^{R\text{-}312}$ denote the set of $R$-312-avoiding permutations.  We define the \emph{$R$-parabolic Catalan number} $C_n^R$ by $C_n^R := |S_n^{R\text{-}312}|$.  Consult Table 3.1 for examples of, and counterexamples for, our various kinds of $R$-tuples.  Boldface entries indicate failures.

\begin{figure}[h!]
\centering
\begin{tabular}{lccc}
\underline{Type of $R$-tuple} & \underline{Set} & \underline{Example} & \underline{Counterexample} \\ \\
Upper $R$-increasing tuple & $\alpha \in UI_R(n)$ & $(2,6,7;4,5,7,8,9;9)$ & $(3,5,\textbf{5};6,\textbf{4},7,8,9;9)$ \\ \\
$R$-312-avoiding permutation & $\pi \in S_n^{R\text{-}312}$  & $(2,3,6;1,4,5,8,9;7)$ & $(2,4,\textbf{6};1,\textbf{3},7,8,9;\textbf{5})$ \\ \\
Gapless $R$-tuple & $\gamma \in UG_R(n)$ & $(2,4,6;4,5,6,7,9;9)$ & $(2,4,6;\textbf{4},\textbf{6},7,8,9;9)$ \\ \\
$R$-floor flag & $\tau \in UFlr_R(n)$  & $(2,4,5;5,5,6,8,9;9)$ & $(2,4,5;5,5,\textbf{8},\textbf{8},9;9)$ \\ \\
$R$-ceiling flag & $\xi \in UCeil_R(n)$ & $(1,4,4;5,5,9,9,9;9)$ & $(1,4,4;5,5,\textbf{7},\textbf{8},9;9)$ \\ \\
Gapless core $R$-tuple & $\eta \in UGC_R(n)$ & $(4,5,5; 4,8,7,8,8;9)$ & $(4,5,5;\textbf{4},8,7,8,\textbf{9};9)$ \\ \\
\end{tabular}\caption*{Table 3.1.  (Counter-)Examples of R-tuples for $n = 9$ and $R = \{3,8\}$.}
\end{figure}

A \emph{gapless $R$-tuple} is an $R$-increasing upper tuple $\gamma$ such that whenever there exists $h \in [r]$ with $\gamma_{q_h} > \gamma_{q_h+1}$, then $\gamma_{q_h} - \gamma_{q_h+1} + 1 =: s \leq p_{h+1}$ and the first $s$ entries of the $(h+1)^{st}$ carrel $(q_h, q_{h+1} ]$ are $\gamma_{q_h}-s+1, \gamma_{q_h}-s+2, ... , \gamma_{q_h}$.  Let $UG_R(n) \subseteq UI_R(n)$ denote the set of gapless $R$-tuples.  Note that a gapless $\gamma$ has $\gamma_{q_1} \leq \gamma_{q_2} \leq ... \leq \gamma_{q_r} \leq \gamma_{q_{r+1}}$.  So in the full $R = [n-1]$ case, each gapless $R$-tuple is a flag.  Hence $UG_{[n-1]}(n) = UF_{[n-1]}(n)$.

An $R$-\textit{chain} $B$ is a sequence of sets $\emptyset =: B_0 \subset B_1 \subset \ldots \subset B_r \subset B_{r+1} := [n]$ such that $|B_h| = q_h$ for $h \in [r]$.  A bijection from $R$-permutations $\pi$ to $R$-chains $B$ is given by $B_h := \{\pi_1, \pi_2, \ldots, \pi_{q_h}\}$ for $h \in [r]$.  We indicate it by $\pi \leftrightarrow B$.  Fix an $R$-permutation $\pi$ and let $B$ be the corresponding $R$-chain.  For $h \in [r+1]$, the set $B_h$ is the union of the first $h$ cohorts of $\pi$.  Note that $R$-chains $B$ (and hence $R$-permutations $\pi$) are equivalent to the ${n \choose R}$ objects that could be called ``ordered $R$-partitions of $[n]$''; these arise as the sequences $(B_1 \backslash B_0, B_2\backslash B_1, \ldots, B_{r+1}\backslash B_r)$ of $r+1$ disjoint nonempty subsets of sizes $p_1, p_2, \ldots, p_{r+1}$.

Now create an $R$-tuple $\Psi_R(\pi) =: \psi$ as follows:  For $h \in [r+1]$ specify the entries in its $h^{th}$ carrel by $\psi_i := \text{rank}^{q_h-i+1}(B_h)$ for $i \in (q_{h-1},q_h]$.  As well as being $R$-increasing, it can be seen that $\psi$ is upper:  So $\psi \in UI_R(n)$.  We call $\psi$ the \emph{rank $R$-tuple of $\pi$}.  See Table 3.2.  For a model, imagine there are $n$ discus throwers grouped into $r+1$ heats of $p_h$ throwers for $h \in [r+1]$.  Each thrower gets one throw, the throw distances are elements of $[n]$, and there are no ties.  After the $h^{th}$ heat has been completed, the $p_h$ longest throws overall are announced in ascending order.

\begin{figure}[h!]
\centering
\begin{tabular}{lccc}
\underline{Name} & \underline{From/To} & \underline{Input} & \underline{Image} \\ \\
Rank $R$-tuple & $\Psi_R:  S_n^R \rightarrow  UI_R(n)$ & $(2,4,6;1,5,7,8,9;3)$   & $(2,4,6;5,6,7,8,9;9)$ \\ \\
Undoes $\Psi_R|_{S_n^{R\text{-}312}}$ & $\Pi_R:  UG_R(n) \rightarrow S_n^{R\text{-}312}$ & $(2,4,6;4,5,6,7,9;9)$  & $(2,4,6;1,3,5,7,9;8)$ \\ \\
$R$-core & $\Delta_R: U_R(n) \rightarrow UI_R(n)$ & $(7,9,6;5,5,9,8,9;9)$ & $(4,5,6;4,5,7,8,9;9)$ \\ \\
$R$-floor & $\Phi_R: UG_R(n) \rightarrow UFlr_R(n)$ & $(3,4,6;4,5,6,8,9;9)$  & $(3,4,6;6,6,6,8,9;9)$ \\ \\
$R$-ceiling & $\Xi_R: UG_R(n) \rightarrow UCeil_R(n)$ & $(3,4,5;4,5,6,8,9;9)$  & $(5,5,5;6,6,6,9,9;9)$ \\ \\
\end{tabular}\caption*{Table 3.2.  Examples for maps of $R$-tuples for $n = 9$ and $R = \{3, 8 \}$.}
\end{figure}

In Proposition \ref{prop320.2}(ii) it will be seen that the restriction of $\Psi_R$ to $S_n^{R\text{-}312}$ is a bijection to $UG_R(n)$ whose inverse is the following map $\Pi_R$.  Let $\gamma \in UG_R(n)$.  Define an $R$-tuple $\Pi_R(\gamma) =: \pi$ by:  Initialize $\pi_i := \gamma_i$ for $i \in (0,q_1]$.  Let $h \in [r]$.  If $\gamma_{q_h} > \gamma_{q_h+1}$, set $s:= \gamma_{q_h} - \gamma_{q_h+1} + 1$.  Otherwise set $s := 0$.  For $i$ in the right side $(q_h + s, q_{h+1}]$ of the $(h+1)^{st}$ carrel, set $\pi_i := \gamma_i$.  For $i$ in the left side $(q_h, q_h + s]$, set $d := q_h + s - i + 1$ and $\pi_i := rank^d( \hspace{1mm} [\gamma_{q_h}] \hspace{1mm} \backslash \hspace{1mm} \{ \pi_1, ... , \pi_{q_h} \} \hspace{1mm} )$.  (Since $\gamma$ is a gapless $R$-tuple, when $s \geq 1$ we have $\gamma_{q_h + s} = \gamma_{q_h}$.  Since `gapless' includes the upper property, here we have $\gamma_{q_h +s} \geq q_h + s$.  Hence $| \hspace{1mm} [\gamma_{q_h}] \hspace{1mm} \backslash \hspace{1mm} \{ \pi_1, ... , \pi_{q_h} \} \hspace{1mm} | \geq s$, and so there are enough elements available to define these left side $\pi_i$. )  Since $\gamma_{q_h} \leq \gamma_{q_{h+1}}$, it can inductively be seen that $\max\{ \pi_1, ... , \pi_{q_h} \} = \gamma_{q_h}$.

Let $\upsilon \in U_R(n)$.  The information that we need from $\upsilon$ will often be distilled into a skeletal substructure with respect to $R$:  Fix $h \in [r+1]$.  Working within the $h^{th}$ carrel $(q_{h-1}, q_h]$ from the right we recursively find for $u = 1, 2, ...$ :  At $u = 1$ the \emph{rightmost critical pair of $\upsilon$} in the $h^{th}$ carrel is $(q_h, \upsilon_{q_h})$.  Set $x_1 := q_h$.  Recursively attempt to increase $u$ by 1:  If it exists, the \emph{next critical pair to the left} is $(x_u, \upsilon_{x_u})$, where $q_{h-1} < x_u < x_{u-1}$ is maximal such that $\upsilon_{x_{u-1}} - \upsilon_{x_u} > x_{u-1} - x_u$.  Otherwise, let $f_h \geq 1$ be the last value of $u$ attained.  The \emph{set of critical pairs of $\upsilon$ for the $h^{th}$ carrel} is $\{ (x_u, \upsilon_{x_u}) : u \in [f_h] \} =: \mathcal{C}_h$.  Equivalently, here $f_h$ is maximal such that there exists indices $x_1, x_2, ... , x_{f_h}$ such that $q_{h-1} < x_{f_h} < ... < x_1 = q_h$ and $\upsilon_{x_{u-1}} - \upsilon_{x_u} > x_{u-1} - x_u$ for $u \in (1, f_h]$.  The \emph{$R$-critical list for $\upsilon$} is the sequence $(\mathcal{C}_1, ... , \mathcal{C}_{r+1}) =: \mathcal{C}$ of its $r+1$ sets of critical pairs.  The $R$-critical list for the gapless $R$-tuple $\gamma$ of Table 3.1 is $( \{ (1,2), (2,4), (3,6) \}; \{ (7,7), (8,9) \}; \{ (9,9) \} )$.

Without having any $\upsilon$ specified, for $h \in [r+1]$ we define a set $\{ (x_u, y_{x_u}) : u \in [f_h] \} =: \mathcal{C}_h$ of pairs for some $f_h \in [p_h]$ to be a \emph{set of critical pairs} for the $h^{th}$ carrel if:  $x_u \leq y_{x_u}, q_{h-1} < x_{f_h} < ... < x_1 = q_h$, and $y_{x_{u-1}} - y_{x_u} > x_{u-1} - x_u$ for $u \in (1, f_h]$.  A sequence of $r+1$ sets of critical pairs for all of the carrels is an \emph{$R$-critical list}.  The $R$-critical list of a given $\upsilon \in U_R(n)$ is an $R$-critical list.  If $(x, y_x)$ is a critical pair, we call $x$ a \emph{critical index} and $y_x$ a \emph{critical entry}.  We say that an $R$-critical list is a \emph{flag $R$-critical list} if whenever $h \in [r]$ we have $y_{q_h} \leq y_k$, where $k := x_{f_{h+1}}$.  This condition can be restated as requiring that the sequence of all of its critical entries be weakly increasing.  If  $\upsilon \in UF_R(n)$, then its $R$-critical list is a flag $R$-critical list.

We illustrate some recent definitions.  First consider an $R$-increasing upper tuple $\alpha \in UI_R(n)$:  Each carrel subsequence of $\alpha$ is a concatenation of the staircases within the carrel in which the largest entries are the critical entries for the carrel.  Now consider the definition of a gapless $R$-tuple, which begins by considering a $\gamma \in UI_R(n)$:  This definition is equivalent to requiring for all $h \in [r]$ that if $\gamma_{q_h} > \gamma_{q_{h}+1}$, then the leftmost staircase within the $(h+1)^{st}$ carrel must contain an entry $\gamma_{q_h}$.

Here are four kinds of $nn$-tuples that will be seen in Proposition \ref{prop604.6} to arise from extending (flag) $R$-critical lists in various unique ways:

\begin{defn}Let $R \subseteq [n-1]$.

\noindent (i)  We say that $\rho \in U_R(n)$ is an \emph{$R$-shell tuple} if $\rho_i = n$ for every non-critical index $i$ of $\rho$.

\noindent (ii)  We say that $\kappa \in U_R(n)$ is an \emph{$R$-canopy tuple} if it is an $R$-shell tuple whose critical list is a flag critical list.

\noindent (iii)  We say that $\tau \in UF_R(n)$ is an \emph{$R$-floor flag} if the leftmost pair of each non-trivial plateau in $\tau$ has the form $(q_h, \tau_{q_h})$ for some $h \in [r]$.

\noindent (iv)  We say that $\xi \in UF_R(n)$ is an \emph{$R$-ceiling flag} if it is a concatenation of plateaus whose rightmost pairs are the $R$-critical pairs of $\xi$.  \end{defn}

\noindent It will be seen in Corollary \ref{cor604.8} that $R$-increasing upper tuples and $R$-shell tuples bijectively correspond to $R$-critical lists.  Hence the number of $R$-critical lists and of $R$-shell tuples is also ${n  \choose R}$.  Let $UFlr_R(n)$ and $UCeil_R(n)$ respectively denote the sets of $R$-floor flags and of $R$-ceiling flags.

There are various ways in which the skeletal structure specified by a (flag) $R$-critical list will be extended in Proposition \ref{prop604.6} to form an $R$-tuple without changing the $R$-critical list;  this will be done by specifying the entries at the non-critical indices in certain fashions.  Here we describe the most fundamental way of doing this.  We form the $R$-critical list of an upper $R$-tuple and then fill it out in a minimal increasing fashion without changing the $R$-critical list.  Let $\upsilon \in U_R(n)$.  Create an $R$-tuple $\Delta_R(\upsilon) =: \delta$ as follows:  Let $x$ be a critical index for $\upsilon$.  If $x$ is the leftmost critical index set $x^\prime := 0$;  otherwise let $x^\prime$ be the largest critical index that is less than $x$.  For every critical pair $(x, \upsilon_x)$ for $\upsilon$, set $\delta_x := \upsilon_x$.  For $x^\prime < i < x$, set $\delta_i := \upsilon_x - (x-i)$.  This forms a staircase toward the left from each critical index.  Clearly $\delta \in UI_R(n)$.  We call $\Delta_R$ the \emph{R-core map} from $U_R(n)$ to $UI_R(n)$.  At the end of Section 5 it will be noted that the restrictions of $\Delta_R$ to $UFlr_R(n)$ and $UCeil_R(n)$ are bijections to $UG_R(n)$.  Their inverse maps $\Phi_R$ and $\Xi_R$ are introduced there.  A \emph{gapless core $R$-tuple} is an upper $R$-tuple $\eta$ whose $R$-core $\Delta_R(\eta)$ is a gapless $R$-tuple.  Let $UGC_R(n)$ denote the set of gapless core $R$-tuples.  In Section 4 we will see that $UF_R(n) \subseteq UGC_R(n) \subseteq UG_R(n)$.  So $UF_{[n-1]}(n) = UGC_{[n-1]}(n) = UG_{[n-1]}(n)$.

When we restrict our attention to the full $R = [n-1]$ case, we will suppress all prefixes and subscripts of `$R$'.  Above we would have written $UF(n) = UGC(n) = UG(n)$.  It can be seen that $UFlr(n) = UCeil(n) = UF(n)$, and that this is also the set of $[n-1]$-canopy tuples.  The  number of $nn$-tuples in each of these sets is $C_n$.

\section{Cores, shells, gapless tuples, canopies, floors, ceilings}

In this section we use the critical list substructure to relate six kinds of $R$-tuples that can be used as indexes for row bound tableau sets in Section 12.  Over Section 3, this section, and Section 5 we are defining three versions of some of these notions, which have a word such as `floor' in common in their names.  When delineation of these three similar concepts is needed, one should consult the summary paragraph at the end of Section 5.

\begin{fact}\label{fact604.2}Let $\upsilon \in U_R(n)$.  Its $R$-core $\Delta_R(\upsilon) =: \delta$ is an $R$-increasing upper tuple:  $\delta \in UI_R(n)$.  Here $\delta \leq \upsilon$ in $U_R(n)$ and $\delta$ has the same critical list as $\upsilon$.  So $\upsilon^\prime \in U_R(n)$ has the same critical list as $\upsilon$ if and only if $\Delta_R(\upsilon^\prime) = \Delta_R(\upsilon)$.  If $\upsilon \in UI_R(n)$, then $\Delta_R(\upsilon) = \upsilon$.  \end{fact}

\noindent The process used to define the $R$-core map can also be used to bijectively produce the $R$-tuples in $UI_R(n)$ from the set of all $R$-critical lists:  To see surjectivity, note that the staircases within the carrels of a given $\alpha \in UI_R(n)$ can be formed toward the left from the critical pairs of $\alpha$.

We will be defining more maps from sets of upper $R$-tuples to sets of upper $R$-tuples.  We will always require that the $R$-critical list of an upper $R$-tuple be preserved.  At times we will need to have the range contained in $UF_R(n)$.  In those cases, to produce an upper $R$-tuple whose $R$-critical list is a flag $R$-critical list, we must exclude from the domain the $\upsilon \in U_R(n)$ that do not have flag $R$-critical lists.  Sometimes we will already want the domain to be $UF_R(n)$;  this will suffice.  Part (iii) of the next statement characterizes the upper $R$-tuples with flag $R$-critical lists.  Here Part (i) restates part of the fact above to provide contrast for Part (ii).  Part (iv) notes that the relationship of `gapless' to `increasing upper' is analogous to the relationship of `canopy' to `shell'.

\begin{prop}\label{prop604.4}Let $\upsilon \in U_R(n), \eta \in UGC_R(n), \phi \in UF_R(n)$, and $\alpha \in UI_R(n)$.

\noindent (i)  The $R$-core $\Delta_R(\upsilon)$ of $\upsilon$ is an $R$-increasing upper tuple.

\noindent (ii)  The $R$-cores $\Delta_R(\eta)$ and $\Delta_R(\phi)$ of $\eta$ and $\phi$ are gapless $R$-tuples.  We have $UG_R(n) \subseteq UGC_R(n)$ and $UF_R(n) \subseteq UGC_R(n)$.

\noindent (iii)  The $R$-critical list of $\upsilon$ is a flag $R$-critical list if and only if $\upsilon \in UGC_R(n)$.

\noindent (iv)  The $R$-critical list of $\alpha$ is a flag $R$-critical list if and only if $\alpha \in UG_R(n)$.  \end{prop}

\begin{proof}For (ii), recall that $\Delta_R(\eta) \in UG_R(n)$ by definition.  Since $UG_R(n) \subseteq UI_R(n)$, we have $\Delta_R(\gamma) = \gamma$ for $\gamma \in UG_R(n)$.  Hence $UG_R(n) \subseteq UGC_R(n)$.  To show $\Delta_R(\phi) \in UG_R(n)$ and to begin the proof of (iii), let $\upsilon \in U_R(n)$.  Set $\delta := \Delta_R(\upsilon)$.  Fix $h \in [r]$.  Let $k$ be the leftmost critical index of $\upsilon$ in $(q_h, q_{h+1}]$.  Here $\delta_{q_h} = \upsilon_{q_h}$ and $\delta_k = \upsilon_k$.  The index $q_h +1$ is included in the leftmost staircase of $\upsilon$ within $(q_h, q_{h+1}]$.  Here $\delta_{q_h+1} \leq \delta_k$.  Suppose the hypothesis $\delta_{q_h} > \delta_{q_h+1}$ of the definition of `gapless' is satisfied.  Here the entry $\delta_{q_h}$ occurs in the cohort of the leftmost staircase, which is on $[\delta_{q_h+1},\delta_k]$, if and only if $\delta_{q_h} \leq \delta_k$.  To finish (ii), note that this is satisfied since $\upsilon$ is a flag, because $\delta_{q_h} = \upsilon_{q_h} \leq \upsilon_k = \delta_k$.  Part (iii) holds since $\delta_{q_h} \leq \delta_k$ is the same as the flag $R$-critical list defining condition of $\upsilon_{q_h} \leq \upsilon_k$.  Part (iv) follows from Fact \ref{fact604.2} and (iii).  \end{proof}

Most of our kinds of $R$-tuples correspond bijectively to $R$-critical lists or to flag $R$-critical lists.  The following six $R$-tuples $\alpha, \rho, \gamma, \kappa, \tau,$ and $\xi$ will be considered in the proposition below.  Let $\mathcal{C}$ be an $R$-critical list.  For each critical pair $(x, y_x)$ in $\mathcal{C}$, if $x$ is the leftmost critical index set $x^\prime := 0$;  otherwise let $x^\prime$ be the largest critical index that is less than $x$.  Set $\xi_x := \tau_x := \kappa_x := \gamma_x := \rho_x := \alpha_x := y_x$.  Then for $x^\prime < i < x$:  Set $\alpha_i := \alpha_x - (x-i)$.  Set $\rho_i := n$.  Now suppose that $\mathcal{C}$ is a flag $R$-critical list.  Set $\gamma_i := \gamma_x - (x-i)$.  Set $\kappa_i := n$.  If $x$ is the leftmost critical index in the $(h+1)^{st}$ carrel for some $h \in [r]$, then $x^\prime = q_h$ and we set $\tau_i := \max \{ \tau_{q_h}, \tau_x - (x-i) \}$ for $i \in (q_h, x)$.  Otherwise set $\tau_i := \tau_x - (x-i)$ for $i \in (x^\prime, x)$.  Set $\xi_i := \xi_x$.

\begin{prop}\label{prop604.6}Let $\mathcal{C}$ be an $R$-critical list.

\noindent (i)  The $R$-tuples $\alpha$ and $\rho$ above are respectively the unique $R$-increasing upper tuple and the unique $R$-shell tuple whose $R$-critical lists are $\mathcal{C}$.

\noindent (ii)  If $\mathcal{C}$ is a flag $R$-critical list, the $R$-tuples $\gamma, \kappa, \tau,$ and $\xi$ above are respectively the unique gapless $R$-tuple, the unique $R$-canopy tuple, the unique $R$-floor flag, and the unique $R$-ceiling flag whose $R$-critical lists are $\mathcal{C}$.  \end{prop}

\begin{proof}It is clear that the $R$-critical list of each of these six tuples is the given $R$-critical list.  We confirm that the six definitions are satisfied:  Since the $R$-tuples $\alpha$ and $\gamma$ are produced as in the definition of the $R$-core map $\Delta_R$, we see that $\alpha, \gamma \in UI_R(n)$.  Since the $R$-critical list given for $\gamma$ is a flag $R$-critical list, Proposition \ref{prop604.4}(iv) implies $\gamma \in UG_R(n)$.  Clearly $\rho$ is an $R$-shell tuple.  Since $\kappa = \rho$, the flag $R$-critical list hypothesis implies that $\kappa$ is an $R$-canopy tuple.  If $\tau$ has a non-trivial plateau it must occur when $\tau_i$ is set to $\tau_{q_h}$ for some $h \in [r]$ and some consecutive indices $i$ at the beginning of the $(h+1)^{st}$ carrel.  If this $\tau_{q_h}$ is greater than $\tau_{q_h-1}$ then the definition of $R$-floor flag is satisfied.  Otherwise $\tau_{q_h} = \tau_{q_h-1}$, which implies that all entries in the $h^{th}$ carrel have the value $\tau_{q_{h-1}}$.  This plateau will necessarily terminate at the rightmost entry in some earlier carrel, since the entries in the first carrel are strictly increasing.  Clearly $\xi$ is an $R$-ceiling flag.

For the uniqueness of $\alpha$, recall that it was noted earlier that this construction is bijective from $R$-critical lists to $UI_R(n)$.  Restrict this bijection to the flag $R$-critical lists to get uniqueness for $\gamma$.  It is clear from the definitions of $R$-shell tuple, $R$-canopy tuple, and $R$-ceiling flag that for each of these notions any two $R$-tuples with the same $R$-critical list must also have the same non-critical entries.  Let $\tau^\prime$ be any $R$-floor flag with flag $R$-critical list $\mathcal{C}$.  Let $h \in [r]$.  Let $x$ be the leftmost critical index in $(q_h, q_{h+1}]$.  Then for $i \in (q_h, x)$ it can be seen that the critical entries at $q_h$ and $x$ force $\tau_i^\prime = \max \{ \tau_{q_h}^\prime, \tau_x^\prime - (x-i) \}$.  On $(x, q_{h+1}]$ and $(0, q_1]$ the flag $\tau^\prime$ must be increasing.  So on $(x, q_{h+1})$ and $(0, q_1)$ the entries of $\tau^\prime$ are uniquely determined by the critical pairs via staircase decomposition for $\alpha$.  \end{proof}

\noindent Here we say that $\alpha$ and $\rho$ are respectively the \emph{$R$-increasing upper tuple} and the \emph{$R$-shell tuple} \emph{for the $R$-critical list $\mathcal{C}$}.  We also say that $\gamma, \kappa, \tau,$ and $\xi$ are respectively the \emph{gapless $R$-tuple}, the \emph{$R$-canopy tuple}, the \emph{$R$-floor flag}, and the \emph{$R$-ceiling flag} \emph{for the flag $R$-critical list $\mathcal{C}$}.

\begin{cor}\label{cor604.8}The six constructions above specify bijections from the set of $R$-critical lists (flag $R$-critical lists) to the sets of $R$-increasing upper tuples and $R$-shell tuples (gapless $R$-tuples, $R$-canopy tuples, $R$-floor flags, and $R$-ceiling flags).  \end{cor}

\begin{proof}These maps are injective since they preserve the (flag) $R$-critical lists.  To show surjectivity, first find the $R$-critical list of the target $R$-tuple. \end{proof}

In passing we note:

\begin{fact}The subposet $UG_R(n)$ of $U_R(n)$ is a meet sublattice.  Let $\upsilon, \upsilon' \in U_R(n)$.  If $\upsilon \leq \upsilon'$ then $\Delta_R(\upsilon) \leq \Delta_R(\upsilon')$.  This implies that the $R$-core map on $U_R(n)$ preserves meet.  Hence the subposet $UGC_R(n)$ of $U_R(n)$ is a meet sublattice. \end{fact}

\noindent Both $UG_R(n)$ and $UGC_R(n)$ fail to be join sublattices.

\section{Equivalence classes in $\mathbf{\emph{U}}_\mathbf{\emph{R}}\mathbf{\emph{(n)}}$ and $\mathbf{\emph{UF}}_\mathbf{\emph{R}}\mathbf{\emph{(n)}}$, inverses}

Here we present results needed to study the sets of tableaux of shape $\lambda$ with given row bounds in Section 12.  There we reduce that study to the study of the following sets of $R$-increasing tuples, after we determine $R := R_\lambda \subseteq [n-1]$ from $\lambda$:  For $\beta \in U_R(n)$, set $\{ \beta \}_R := \{ \epsilon \in UI_R(n):  \epsilon \leq \beta \}$.  This is not \'{a} priori a principal ideal in $UI_R(n)$, since it is possible that $\beta \notin UI_R(n)$.  But we will see that for any $\beta$ there exists $\alpha \in UI_R(n)$ such that $\{ \beta \}_R$ is the principal ideal $[\alpha]$ in $UI_R(n)$.

Define an equivalence relation $\sim_R$ on $U_R(n)$ as follows:  Let $\upsilon, \upsilon^\prime \in U_R(n)$.  We define $\upsilon \sim_R \upsilon^\prime$ if $\{ \upsilon \}_R = \{ \upsilon^\prime \}_R$.  Sometimes we restrict $\sim_R$ from $U_R(n)$ to $UGC_R(n)$, or further to $UF_R(n)$.  We denote the equivalence classes of $\sim_R$ in these three sets respectively by $\langle \upsilon \rangle_{\sim_R}$, $\langle \eta \rangle_{\sim_R}^G$, and $\langle \phi \rangle_{\sim_R}^F$.  We indicate intervals in $UGC_R(n)$ and $UF_R(n)$ respectively with $[ \cdot , \cdot ]^G$ and $[ \cdot, \cdot ]^F$.

\begin{lem}\label{lemma608.2}Let $\upsilon \in U_R(n), \eta \in UGC_R(n)$, and $\phi \in UF_R(n)$.

\noindent (i)  Here $\{ \upsilon \}_R = [\Delta_R(\upsilon)] \subseteq UI_R(n)$.  So $\upsilon' \sim_R \upsilon$ for some $\upsilon' \in U_R(n)$ if and only if $\Delta_R(\upsilon') = \Delta_R(\upsilon)$ if and only if $\upsilon'$ has the same $R$-critical list as $\upsilon$.

\noindent (ii)  The equivalence classes $\langle \upsilon \rangle_{\sim_R}$, $\langle \eta \rangle_{\sim_R}^G$, and $\langle \phi \rangle_{\sim_R}^F$ are closed respectively in $U_R(n), UGC_R(n)$, and $UF_R(n)$ under the meet and the join operations for $U_R(n)$.  \end{lem}

\begin{proof}First we show that $\{ \upsilon \}_{\sim_R} \subseteq [\Delta_R(\upsilon)]$, which is the most interesting step for (i).  Set $\delta := \Delta_R(\upsilon)$ and let $\alpha \in UI_R(n)$ be such that $\alpha \leq \upsilon$.  Let $x$ be a critical index of $\upsilon$.  So $\delta_x = \upsilon_x$.  Here $\alpha \leq \upsilon$ implies $\alpha_x \leq \delta_x$.  Now let $i$ be a non-critical index of $\upsilon$ and let $x$ be the smallest critical index of $\upsilon$ that is larger than $i$.  Here $\delta_i = \upsilon_x - (x-i)$.  Since $\alpha \in UI_R(n)$ we have $\alpha_i \leq \alpha_x - (x-i)$.  So $\alpha \leq \upsilon$ implies $\alpha_i \leq \delta_i$.  For (ii), note that the critical lists of the join and the meet of two elements of $U_R(n)$ that share a critical list are that mutual critical list.  If the two such elements were in $UGC_R(n)$, it can be seen that their meet and join are in $UGC_R(n)$.  Recall that $U_R(n)$ and $UF_R(n)$ are lattices. \end{proof}

\noindent So by Part (i) we can view these three equivalence classes as consisting of $R$-tuples that share (flag) $R$-critical lists.  And by Part (ii), each of these equivalence classes has a unique minimal and a unique maximal element under the entrywise partial orders.  We denote the minimums of $\langle \upsilon \rangle_{\sim_R} \subseteq U_R(n)$ and of $\langle \eta \rangle_{\sim_R}^G \subseteq UGC_R(n)$ by $\utilde{\upsilon}$ and \d{$\eta$} respectively.  We call the maximums of $\langle \upsilon \rangle_{\sim_R} \subseteq U_R(n)$ and of $\langle \eta \rangle_{\sim_R}^G \subseteq UGC_R(n)$ the \emph{$R$-shell of $\upsilon$} and the \emph{$R$-canopy of $\eta$} and denote them by $\tilde{\upsilon}$ and $\dot{\eta}$ respectively.  For the class $\langle \phi \rangle_{\sim_R}^F \subseteq UF_R(n)$, we call and denote these respectively the \emph{$R$-floor \b{$\phi$} of $\phi$} and the \emph{$R$-ceiling $\bar{\phi}$ of $\phi$}.

These definitions give the containments $\langle \upsilon \rangle_{\sim_R} \subseteq [$$\utilde{\upsilon}$,$\tilde{\upsilon}], \langle \eta \rangle_{\sim_R}^G \subseteq [$\d{$\eta$}$,\dot{\eta}]^G$, and $\langle \phi \rangle_{\sim_R}^F \subseteq [$\b{$\phi$},$\bar{\phi}]^F$ for our next result:

\begin{prop}\label{prop608.4}Let $\upsilon \in U_R(n), \eta \in UGC_R(n)$, and $\phi \in UF_R(n)$.

\noindent (i)  Here $\utilde{\upsilon}$ $ = \Delta_R(\upsilon)$, the $R$-core of $\upsilon$.  In $U_R(n)$ we have $\langle \upsilon \rangle_{\sim_R} = [$$\utilde{\upsilon}$, $\tilde{\upsilon}]$.  The $R$-core $\utilde{\upsilon}$ of $\upsilon$ (respectively $R$-shell $\tilde{\upsilon}$ of $\upsilon$) is the $R$-increasing upper tuple (respectively $R$-shell tuple) for the $R$-critical list of $\upsilon$.

\noindent (ii)  We have $[$$\utilde{\upsilon}$, $\tilde{\upsilon}] \subseteq UGC_R(n)$ or $[$$\utilde{\upsilon}$, $\tilde{\upsilon}] \subseteq U_R(n) \backslash UGC_R(n)$, depending on whether $\upsilon \in UGC_R(n)$ or not.  We also have \d{$\eta$} = $\utilde{\eta}$ and $\dot{\eta} = \tilde{\eta}$.  And $\langle \eta \rangle_{\sim_R}^G =  [$\d{$\eta$}, $\dot{\eta}]^G = [\utilde{\eta}, \tilde{\eta}] = \langle \eta \rangle_{\sim_R}$:  The equivalence classes $UGC_R(n) \supseteq \langle \eta \rangle_{\sim_R}^G$ and $\langle \eta \rangle_{\sim_R} \subseteq U_R(n)$ are the same subset of $U_R(n)$, which is an interval in both contexts.  The $R$-core \d{$\eta$} of $\eta$ (respectively $R$-canopy $\dot{\eta}$ of $\eta$) is the gapless $R$-tuple (respectively $R$-canopy tuple) for the flag $R$-critical list of $\eta$.

\noindent (iii)  In $UF_R(n)$ we have $\langle \phi \rangle_{\sim_R}^F =  [$\b{$\phi$}, $\bar{\phi}]^F$.  The $R$-floor \b{$\phi$} of $\phi$ (respectively $R$-ceiling $\bar{\phi}$ of $\phi$) is the $R$-floor flag (respectively $R$-ceiling flag) for the flag $R$-critical list of $\phi$.  We have $[$\b{$\phi$}, $\bar{\phi}]^F \subseteq $ \\ $[$\d{$\phi$}, $\dot{\phi} ] = \langle \phi \rangle_{\sim_R} \subseteq UGC_R(n)$.  \end{prop}

\noindent So for $\upsilon \in U_R(n)$ the equivalence classes $\langle \upsilon \rangle_{\sim_R}$ are intervals $[ \utilde{\upsilon}, \tilde{\upsilon}]$ that lie entirely in $U_R(n) \backslash UGC_R(n)$ or entirely in $UGC_R(n)$, in which case they coincide with the equivalence classes $\langle \eta \rangle_{\sim_R}^G = [$\d{$\eta$}, $\dot{\eta}]^G$ for $\eta \in UGC_R(n)$ originally defined by restricting $\sim_R$ to $UGC_R(n)$.  However, although for $\phi \in UF_R(n)$ the equivalence class $\langle \phi \rangle_{\sim_R}^F$ is an interval $[$\b{$\phi$}, $\bar{\phi}]^F$ when working within $UF_R(n)$, it can be viewed as consisting of some of the elements of the interval $[$\d{$\phi$}, $\dot{\phi} ]^G$ of $UGC_R(n)$ (or of $U_R(n)$) that is formed by viewing $\phi$ as an element of $UGC_R(n)$.

\begin{proof}The assertions in (i) and (ii) pertaining to $\utilde{\upsilon}$ alone are apparent from Fact \ref{fact604.2}, Proposition \ref{prop604.6}, Lemma \ref{lemma608.2}, and Proposition \ref{prop604.4}.  We know $\langle \eta \rangle_{\sim_R} = [\utilde{\eta},\tilde{\eta}]$.  The first statement in (ii) gives $[\utilde{\eta},\tilde{\eta}] \subseteq UGC_R(n)$.  So $\langle \eta \rangle_{\sim_R} \subseteq UGC_R(n)$.  Hence $\langle \eta \rangle_{\sim_R}^G = \langle \eta \rangle_{\sim_R}$.  Thus \d{$\eta$}$=\utilde{\eta}$ and $\dot{\eta} = \tilde{\eta}$.  To begin working on the five ``critical list for'' claims beyond that for $\Delta_R(\upsilon)$, apply the constructions given in the paragraph preceeding Proposition \ref{prop604.6} to the $R$-critical list of $\upsilon$ and the flag $R$-critical lists of $\eta$ and $\phi$.  By Proposition \ref{prop604.6} this produces the unique $R$-shell tuple $\rho$, the unique gapless $R$-tuple $\gamma$, the unique $R$-canopy tuple $\kappa$,  the unique $R$-floor flag $\tau$, and the unique $R$-ceiling flag $\xi$ for these (flag) $R$-critical lists.  We want to show that $\tilde{\upsilon} = \rho$, \d{$\eta$} $=\gamma, \dot{\eta} = \kappa$, \b{$\phi$} $=\tau$, and $\bar{\phi} = \xi$; the statements about the $R$-core map $\Delta_R$ will then follow from these equalities since all of these pairs of $R$-tuples would have the same (flag) $R$-critical lists.  Here we indicate how to confirm only \b{$\phi$} $= \tau$; the other four confirmations are easier.  By construction $\tau$ had the same critical list as $\phi$, so $\tau \in \langle \phi \rangle_{\sim R}$.  Recall the prescription for the non-critical entries of $\tau$.  It can be seen that decreasing any of these non-critical entries would produce an $R$-tuple that is not a flag or that has a different $R$-critical list.  Thus $\tau$ is the minimum element \b{$\phi$} of $\langle \phi \rangle_{\sim R} \subseteq UF_R(n)$.

Next we show the non-trivial containment only for (iii):  Let $\epsilon \in [$\b{$\phi$},$\bar{\phi}]^F \subseteq UF_R(n)$.  Here \b{$\phi$} $\leq \epsilon \leq \bar{\phi}$ in $UF_R(n)$ implies $\{ $\b{$\phi$}$ \}_R \subseteq \{ \epsilon \}_R \subseteq \{ \bar{\phi} \}_R$ in $UI_R(n)$.  But \b{$\phi$} $\sim_R \phi \sim_R \bar{\phi}$ implies $\{ $\b{$\phi$}$ \}_R = \{ \bar{\phi} \}_R$.  Thus $\{ \epsilon \}_R = \{ \phi \}_R$, and so $\epsilon \sim_R \phi$. \end{proof}

\begin{cor}\label{cor608.6}The equivalence classes of $\sim_R$ can be indexed as follows:

\noindent (i)  In $U_R(n)$, they are precisely indexed by the $R$-increasing upper tuples or the $R$-shell tuples (or by the $R$-critical lists).

\noindent (ii)  In $UGC_R(n)$, they are precisely indexed by the gapless $R$-tuples or the $R$-canopy tuples (or by the flag $R$-critical lists, the $R$-floor flags, or the $R$-ceiling flags).

\noindent (iii)  In $UF_R(n)$, they are precisely indexed by the $R$-floor flags or the $R$-ceiling flags (or by the flag $R$-critical lists, the gapless $R$-tuples, or the $R$-canopy tuples). \end{cor}

If the gapless $R$-tuple label for an equivalence class in $UF_R(n)$ is not a flag, we may want to convert it to the unique $R$-floor (or $R$-ceiling) flag that belongs to the same class.  Let $\gamma \in UG_R(n)$.  Find the $R$-critical list of $\gamma$; by Proposition \ref{prop604.4}(iv) it is a flag $R$-critical list.  As in Section 4, compute the $R$-floor flag $\tau$ and the $R$-ceiling $\xi$ for this flag $R$-critical list.  Define the \emph{$R$-floor map} $\Phi_R : UG_R(n) \longrightarrow UFlr(n)$ and \emph{$R$-ceiling map} $\Xi_R: UG_R(n) \longrightarrow UCeil_R(n)$ by $\Phi_R(\gamma) := \tau$ and $\Xi_R(\gamma) := \xi$.  By Proposition \ref{prop604.6}(ii) and Corollary \ref{cor604.8} these maps are well defined bijections; it can be seen that each has inverse $\Delta_R$.  Part (iii) of the following proposition previews Proposition \ref{prop320.2}(ii).

\begin{prop}\label{prop608.10}The following maps are bijections:

\noindent (i)  $\Phi_R:  UG_R(n) \longrightarrow UFlr_R(n)$ has inverse $\Delta_R$.

\noindent (ii)  $\Xi_R:  UG_R(n) \longrightarrow UCeil_R(n)$ has inverse
$\Delta_R$.

\noindent (iii)  $\Pi_R:  UG_R(n) \longrightarrow S_n^{R\text{-}312}$ has inverse $\Psi_R$.  \end{prop}

To summarize:  In Section 3 the six notions of $R$-increasing upper tuple, $R$-shell tuple, gapless $R$-tuple, $R$-canopy tuple, $R$-floor flag, and $R$-ceiling flag were defined with conditions on the entries of an $R$-tuple.  While introducing the word `for' into these terms, in Section 4 one such $R$-tuple was associated to each (flag) $R$-critical list.  While introducing the word `of' into four of these terms, in this section these kinds of $R$-tuples arose as the extreme elements of equivalence classes.  This began with the classes in $U_R(n)$.  Here these extreme elements were respectively $R$-increasing upper and $R$-shell tuples.  When these classes were restricted to the subset $UGC_R(n)$ of upper $R$-tuples with gapless cores, these extreme elements were respectively gapless $R$-tuples and $R$-canopy tuples.  When these classes were restricted further to the subset $UF_R(n)$ of upper flags, these extreme elements were respectively $R$-floor and $R$-ceiling flags.

\section{Rightmost clump deleting chains}

In this section and the next section we process the $R$-permutations $\pi$ that will index the Demazure tableau sets $\mathcal{D}_\lambda(\pi)$.

Given a set of integers, a \emph{clump} of it is a maximal subset of consecutive integers.  After decomposing a set into its clumps, we index the clumps in the increasing order of their elements.  For example, the set $\{ 2,3,5,6,7,10,13,14 \}$ is the union $L_1 \hspace{.5mm} \cup  \hspace{.5mm} L_2  \hspace{.5mm} \cup \hspace{.5mm}  L_3  \hspace{.5mm} \cup  \hspace{.5mm} L_4$,  where  $L_1 := \{ 2,3 \},  L_2 := \{ 5,6,7 \}, L_3 := \{ 10 \},  L_4 := \{ 13,14 \}$.

For the first part of this section we temporarily work in the context of the full $R = [n-1]$ case.  A chain  $B$  is \textit{rightmost clump deleting} if for $h \in [n-1]$ the element deleted from each $B_{h+1}$ to produce $B_h$ is chosen from the rightmost clump of $B_{h+1}$.  More formally:   It is rightmost clump deleting if for  $h \in [n-1]$  one has $B_{h} = B_{h+1} \backslash \{ b \}$  only when  $[b, m] \subseteq B_{h+1}$,  where  $m := max (B_{h+1})$.  The five rightmost clump deleting chains for  $n = 3$  are shown here:

\begin{figure}[h!]
\hspace{1.5mm}
\setlength\tabcolsep{.1cm}
\begin{tabular}{ccccc}
1& &2& &\cancel{3}\\
 &1& &\cancel{2}& \\
 & &\cancel{1}& & \\
\end{tabular}\hspace{14.5mm}
\begin{tabular}{ccccc}
1& &2& &\cancel{3}\\
 &\cancel{1}& &2& \\
 & &\cancel{2}& & \\
\end{tabular}\hspace{14.5mm}
\begin{tabular}{ccccc}
1& &\cancel{2}& &3\\
 &1& &\cancel{3}& \\
 & &\cancel{1}& & \\
\end{tabular}\hspace{14.5mm}
\begin{tabular}{ccccc}
\cancel{1}& &2& &3\\
 &2& &\cancel{3}& \\
 & &\cancel{2}& & \\
\end{tabular}\hspace{14.5mm}
\begin{tabular}{ccccc}
\cancel{1}& &2& &3\\
 &\cancel{2}& &3& \\
 & &\cancel{3}& & \\
\end{tabular}
\end{figure}

\noindent To form the corresponding $\pi$, record the deleted elements from bottom to top.

After Part (0) restates the definition of this concept, we present four reformulations of it:

\begin{fact}\label{fact320.1}Let $B$ be a chain.  Set $\{ b_{h+1} \} := B_{h+1} \backslash B_h$ for $h \in [n-1]$.  Set $m_h := \max (B_h)$ for $h \in [n]$.  The following conditions are equivalent to this chain being rightmost clump deleting:

\noindent(0)  For $h \in [n-1]$, one has $[b_{h+1}, m_{h+1}] \subseteq B_{h+1}$.

\noindent(i)  For $h \in [n-1]$, one has $[b_{h+1}, m_h] \subseteq B_{h+1}$.

\noindent(ii)  For $h \in [n-1]$, one has $(b_{h+1}, m_h) \subset B_h$.

\noindent(iii)  For $h \in [n-1]$:  If $b_{h+1} < m_h$, then $b_{h+1} = \max([m_h] \backslash B_h)$.

\noindent(iii$^\prime$)  For $h \in [n-1]$, one has $b_{h+1} = \max([m_{h+1}] \backslash B_h)$. \end{fact}

The following characterization is related to Part (ii) of the preceeding fact via the correspondence $\pi \longleftrightarrow B$:

\begin{fact}\label{fact320.2}A permutation $\pi$ is 312-avoiding if and only if for every $h \in [n-1]$ we have \\ $(\pi_{h+1}, \max\{\pi_1, ... , \pi_{h}\}) \subset \{ \pi_1, ... , \pi_{h} \}$. \end{fact}

Since the following result will be generalized by Proposition \ref{prop320.2}, we do not prove it here.

\begin{prop}\label{prop320.1}For the full $R = [n-1]$ case we have:

\noindent (i)  The restriction of the global bijection $\pi \mapsto B$ from $S_n$ to $S_n^{312}$ is a bijection to the set of rightmost clump deleting chains.  Hence there are $C_n$ rightmost clump deleting chains.

\noindent (ii)  The restriction of the rank tuple map $\Psi$ from $S_n$ to $S_n^{312}$ is a bijection to $UF(n)$ whose inverse is $\Pi$.  \end{prop}

\noindent When $R = [n-1]$, the map $\Pi_R =: \Pi : UF(n) \longrightarrow S_n^{312}$ has a simple description.  It was introduced in \cite{PS} for Theorem 14.1.  Given an upper flag $\phi$, recursively construct $\Pi(\phi) =: \pi$ as follows:  Start with $\pi_1 := \phi_1$.  For $i \in [n-1]$, choose $\pi_{i+1}$ to be the maximum element of $[\phi_{i+1}] \backslash \{ \pi_1, ... , \pi_{i} \}$.

We now return to our fixed $R \subseteq [n-1]$.  Let $B$ be an $R$-chain.  More generally, we say $B$ is \textit{$R$-rightmost clump deleting} if this condition holds for each $h \in [r]$:  Let $B_{h+1} =: L_1 \cup L_2 \cup ... \cup L_f$ decompose $B_{h+1}$ into clumps for some $f \geq 1$.  We require $L_e \cup L_{e+1} \cup ... \cup L_f \supseteq B_{h+1} \backslash B_{h} \supseteq L_{e+1} \cup ... \cup L_f$ for some $e \in [f]$.  This condition requires the set $B_{h+1} \backslash B_h$ of new elements that augment the set $B_h$ of old elements to consist of entirely new clumps $L_{e+1}, L_{e+2}, ... , L_f$, plus some further new elements that combine with some old elements to form the next clump $L_e$ in $B_{h+1}$.  Here are some reformulations of the notion of $R$-rightmost clump deleting:

\begin{fact}\label{fact320.3}Let $B$ be an $R$-chain.  For $h \in [r]$, set $b_{h+1} := \min (B_{h+1} \backslash B_{h} )$ and $m_h := \max (B_h)$.  This $R$-chain is $R$-rightmost clump deleting if and only if each of the following holds:

\noindent (i)  For $h \in [r]$, one has $[b_{h+1}, m_{h}] \subseteq B_{h+1}$.

\noindent (ii)  For $h \in [r]$, one has $(b_{h+1}, m_{h}) \subset B_{h+1}$.

\noindent (iii)  For $h \in [r]$, let $s$ be the number of elements of $B_{h+1} \backslash B_{h}$ that are less than $m_{h}$.  These must be the $s$ largest elements of $[m_{h}] \backslash B_{h}$.  \end{fact}

The following characterization is related to Part (ii) of the preceding fact via the correspondence $\pi \longleftrightarrow B$:

\begin{fact}\label{fact320.4}An $R$-permutation $\pi$ is $R$-312-avoiding if and only if for every $h \in [r]$ one has \\ $( \min\{\pi_{q_{h}+1}, ... , \pi_{q_{h+1}} \} , \max \{\pi_1, ... , \pi_{q_{h}}\} ) \subset \{ \pi_1, ... , \pi_{q_{h+1}} \}$.  \end{fact}

Is it possible to characterize the rank $R$-tuple $\Psi_R(\pi) =: \psi$ of an $R$-permutation $\pi$?  An \emph{$R$-flag} is an $R$-increasing upper tuple $\varepsilon$ such that $\varepsilon_{q_{h+1} +1 - u} \geq \varepsilon_{q_{h} +1 - u}$ for $h \in [r]$ and $u \in [\min\{ p_{h+1}, p_{h}\}]$.  It can be seen that $\psi$ is necessarily an $R$-flag.  But the three conditions required so far (upper, $R$-increasing, $R$-flag) are not sufficient:  When $n = 4$ and $R = \{ 1, 3 \}$, the $R$-flag $(3,2,4,4)$ cannot arise as the rank $R$-tuple of an $R$-permutation.  In contrast to the upper flag characterization in the full case, it might not be possible to develop a simply stated sufficient condition for an $R$-tuple to be the rank $R$-tuple  $\Psi_R(\pi)$ of a general $R$-permutation $\pi$.  But it can be seen that the rank $R$-tuple $\psi$ of an $R$-312-avoiding permutation $\pi$ is necessarily a gapless $R$-tuple, since a failure of `gapless' for $\psi$ leads to the containment of an $R$-312 pattern.  Building upon the observation that $UG(n) = UF(n)$ in the full case, this seems to indicate that the notion of ``gapless $R$-tuple'' is the correct generalization of the notion of ``flag'' from $[n-1]$-tuples to $R$-tuples.  (It can be seen directly that a gapless $R$-tuple is necessarily an $R$-flag.)

\begin{prop}\label{prop320.2}For general $R \subseteq [n-1]$ we have:

\noindent (i)  The restriction of the global bijection $\pi \mapsto B$ from $S_n$ to $S_n^{R\text{-}312}$ is a bijection to the set of $R$-rightmost clump deleting chains.

\noindent (ii)  The restriction of the rank $R$-tuple map $\Psi_R$ from $S_n$ to $S_n^{R\text{-}312}$ is a bijection to $UG_R(n)$ whose inverse is $\Pi_R$.  \end{prop}

\begin{proof}Setting $b_h = \min\{\pi_{q_{h}+1}, ... , \pi_{q_{h+1}} \}$ and $m_{h} = \max \{\pi_1, ... , \pi_{q_{h}}\}$, use Fact \ref{fact320.4}, the $\pi \mapsto B$ bijection, and Fact \ref{fact320.3}(ii) to confirm (i).

As noted above, the restriction of $\Psi_R$ to $S_n^{R\text{-}312}$ gives a map to $UG_R(n)$.  Let $\gamma \in UG_R(n)$ and construct $\Pi_R(\gamma) =: \pi$.  Let $h \in [r]$.  Recall that $\max\{ \pi_1, ... , \pi_{q_h} \} = \gamma_{q_h}$.  Since $\gamma$ is $R$-increasing it can be seen that the $\pi_i$ are distinct.  So $\pi$ is an $R$-permutation.  Let $s \geq 0$ be the number of entries of $\{ \pi_{q_{h}+1} , ... , \pi_{q_{h+1}} \}$ that are less than $\gamma_{q_{h}}$.  These are the $s$ largest elements of $[\gamma_{q_{h}}] \backslash \{ \pi_1, ... , \pi_{q_{h}} \}$.  If in the hypothesis of Fact \ref{fact320.3} we take $B_h := \{\pi_1, ... , \pi_{q_h} \}$, we have $m_h = \gamma_{q_h}$.  So the chain $B$ corresponding to $\pi$ satisfies Fact \ref{fact320.3}(iii).  Since Fact \ref{fact320.3}(ii) is the same as the characterization of an $R$-312-avoiding permutation in Fact \ref{fact320.4}, we see that $\pi$ is $R$-312-avoiding.  It can be seen that $\Psi_R[\Pi_R(\gamma)] = \gamma$, and so $\Psi_R$ is surjective from $S_n^{R\text{-}312}$ to $UG_R(n)$.  For the injectivity of $\Psi_R$, now let $\pi$ denote an arbitrary $R$-312-avoiding permutation.  Form $\Psi_R(\pi)$, which is a gapless $R$-tuple.  Using Facts \ref{fact320.4} and \ref{fact320.3}, it can be seen that $\Pi_R[\Psi_R(\pi)] = \pi$.  Hence $\Psi_R$ is injective.  \end{proof}

\section{Projecting and lifting the notion of 312-avoiding}

In Propositions \ref{prop824.2} and \ref{prop824.4} we use the six maps $\Psi, \Pi, \Psi_R, \Pi_R, \Delta_R$, and $\Phi_R$ that we developed for other purposes to relate the notion of $R$-312-avoiding to that of 312-avoiding.  Some of the applications of these maps ``sort'' the entries of the $R$-tuples within their carrels.

If $\sigma \in S_n$ is 312-avoiding, it is easy to see that its $R$-projection $\bar{\sigma} \in S_n^R$ is $R$-312-avoiding.  Let $\pi \in S^R_n$ be $R$-312-avoiding.  Is it the $R$-projection $\bar{\sigma}$ of some 312-avoiding permutation $\sigma \in S_n$?  The following procedure for constructing an answer to this question can be naively developed, keeping in mind Fact \ref{fact320.3}(iii):  Form the $R$-rightmost clump deleting chain $B$ associated to $\pi$.  Set $\sigma_i := \pi_i$ on the first carrel $(0, q_1]$.  Let $h \in [r]$.  Let $s \geq 0$ be the number of elements of $B_{h+1} \backslash B_{h}$ that are less than $\max(B_{h}) =: m$.  List these elements in decreasing order to fill the left side $(q_h, q_h+s]$ of the $(h+1)^{st}$ carrel $(q_h, q_{h+1}]$ of $\sigma$.  Fill the right side $(q_h+s, q_{h+1}]$ of this carrel of $\sigma$ by listing the other $t := p_{h+1} - s$ elements of $B_{h+1} \backslash B_{h}$ in increasing order.  Part (ii) of the following result refers to the ``length'' of a permutation in the sense of Proposition 1.5.2 of \cite{BB}.

\begin{prop}\label{prop824.1}Suppose $\pi \in S^R_n$ is $R$-312-avoiding.

\noindent (i) The permutation $\sigma \in S_n$ constructed here is 312-avoiding and $\bar{\sigma} = \pi$.

\noindent (ii) This $\sigma$ is the unique minimum length 312-avoiding lift of $\pi$.  \end{prop}

\begin{proof}The construction of $\sigma$ re-orders the cohorts of $\pi$ within their carrels, and so $\bar{\sigma} = \pi$.  Such re-orderings cannot create a violation of 312-avoiding that involves three cohorts.  Let $h \in [r]$ and consider the $(h+1)^{st}$ carrel.  The first $s$ entries here are decreasing, the last $t$ entries are increasing, and the first $s$ entries are smaller than the last $t$ entries.  So there is no 312-violation entirely within this cohort.  Consider a `3' entry in an earlier cohort being in a potential violation.  Since the last $t$ entries here are all greater than that entry, a violation with the `12' entries being here would have to involve two of the first $s$ entries.  But these are decreasing.  The `31' entries cannot occur on $(0, q_1]$.  Consider having the `31' entries in this $(h+1)^{st}$ cohort.  To decrease, both would have to come from the first $s$ entries.  The `3' entry would be less than $m$.  But then the fact rules out having the `2' entry occur in a later cohort.

Let $\sigma^\prime$ be any 312-avoiding lift of $\pi$.  Any other ordering of the entries on $(0,q_1]$ would make $\sigma^\prime$ longer than $\sigma$.  If the $t$ largest entries of the $(h+1)^{st}$ cohort did not appear in the rightmost positions or if they were not listed in ascending order, then $\sigma^\prime$ would be longer than $\sigma$.  The $s$ smallest entries here are all smaller than $m$.  If these entries do not appear in descending order, then $m$ could serve as the `3' entry for a violation in which the `12' entries would be drawn from these first $s$ entries. \end{proof}

This lifting process can also be described using three existing maps.  To pass from the ``degenerate'' $R$-world to the full $R = [n-1]$ world of ordinary permutations, for the second equality below we use the map $\Pi$.  This produces a final output of a permutation from the given $R$-permutation input.  We will use the following result to derive a weaker version of our Theorem \ref{theorem737.1}(ii) from Theorem 14.1 of \cite{PS}:

\begin{prop}\label{prop824.2}Suppose $\pi \in S_n^R$ is $R$-312-avoiding.  Let $\sigma \in S_n$ be the minimum length 312-avoiding lift of $\pi$.  Then $\Delta_R[\Psi(\sigma)] = \Psi_R(\pi)$ and so $\sigma = \Pi[\Phi_R(\Psi_R(\pi))]$.  \end{prop}

\begin{proof}Let $\phi$ denote the upper flag $\Psi(\sigma)$ and let $\gamma$ denote the gapless $R$-tuple $\Psi_R(\pi)$.  Let $h \in [r]$ and consider the $(h+1)^{st}$ carrel $(q_h, q_{h+1}]$.  On the right side $(q_h+s, q_{h+1}]$ of $(q_h, q_{h+1}]$ we defined $\sigma_i := \pi_i$.  The first entry $\pi_{q_h+s+1} =: \sigma_{q_h+s+1}$ here was larger than all earlier entries of $\pi$, and so $\sigma_{q_h+s+1}$ is also larger than all earlier entries of $\sigma$.  Hence applying $\Psi$ (respectively $\Psi_R$) does nothing on $(q_h+s, q_{h+1}]$ to $\sigma$ (respectively $\pi$) since its entries there are increasing.  So $\phi_i = \gamma_i$ for $i \in (q_h+s, q_{h+1}]$.  As $\Psi_R$ ranks the $p_h$ largest elements of $B_{h+1}$ onto $(q_h, q_{h+1}]$ from the right, when it arrives at the index $q_h+s$ the next largest element of $B_{h+1}$ available is $m$.  From the fact it can be seen that the $s-1$ next largest elements available are $m-1, m-2, ... , m-s+1$.  Hence $\gamma_i = m-(q_h+s-i)$ for $i \in (q_h, q_h+s]$.  Clearly $\phi_i = m$ for $i \in (q_h, q_h+s]$.  Now regard the $[n-1]$ tuple $\phi$ as an $R$-tuple and set $\delta := \Delta_R(\phi)$.  We find the $\delta_i$ on $(q_h, q_{h+1}]$ from the right as we apply $\Delta_R$ to $\phi$:  Nothing happens on $(q_h+s-1,q_{h+1}]$ since $\phi$ is increasing there, starting with $\phi_{q_h+s} = m$.  So $\delta_i = \phi_i = \gamma_i$ for $i \in (q_h+s, q_{h+1}]$.  Since one also has $\phi_i = m$ on $(q_h, q_h+s-1]$, we get $\delta_i = m - (q_h+s-i)$ for $i \in (q_h, q_h+s]$.  So $\delta_i = \gamma_i$ on $(q_h, q_h +s]$.  Take $s := 0$ above to see that $\delta_i = \gamma_i$ on $(0,q_1]$.  For the second equality, note that $\phi$ is an $R$-floor flag and that $\Phi_R[\Psi_R(\pi)]$ is an upper flag.  Apply (i) and then Proposition \ref{prop608.10}(iii).  \end{proof}

We further consider an $R$-312-avoiding permutation $\pi$ and its associated $R$-rightmost clump deleting chain, keeping in mind the picture provided by Fact \ref{fact320.3}(iii).  We want to describe all 312-avoiding lifts $\sigma^\prime$ of $\pi$.  Let $h \in [r]$.  As in Section 6, let $B_{h+1} =: L_1 \cup L_2 \cup ... \cup L_f$ decompose $B_{h+1}$ into clumps for some $f \geq 1$.  Restating the clump deleting condition in Section 6, we take $e \in [f]$ to be maximal such that $L_e \cap B_h \neq \emptyset$ and $B_{h+1} \backslash B_h \supseteq L_{e+1} \cup ... \cup L_f$.  The $s$ elements of $B_{h+1} \backslash B_h$ that are smaller than $m$ are in the clump $L_e$.  It is possible that some elements $m+1, m+2, ...$ from $B_{h+1} \backslash B_h$ are also in $L_e$.  Set $m^\prime := \max(L_e)$ and $s^\prime := | L_e \backslash B_h |$.  Since $\pi$ is $R$-increasing, when $s^\prime > s$ we have $\pi_{q_h + s + 1} = m+1, ... , \pi_{q_h+s^\prime} = m^\prime$ with $m^\prime - m = s^\prime - s$.  So then $\pi$ contains this staircase within the subinterval $(q_h+s, q_h + s^\prime]$ of $(q_h, q_{h+1}]$.  In any case we refer to the cohort $L_e \backslash B_h$ on $(q_h, q_h + s^\prime]$ as the (possibly empty) \emph{subclump} $L_e^\prime$ of $L_e$.

\begin{fact}With respect to the entities introduced above for $h \in [r]$:  Corresponding to the clumps $L_{e+1}, ... , L_f$ of $B_{h+1}$ there are respective staircases of $\pi$ within $(q_h, q_{h+1}]$.  When $s^\prime > s$ there is also a staircase of $\pi$ within $(q_h+s, q_h+s^\prime]$.  The supports of these staircases ``pave'' $(q_h+s,q_{h+1}]$.  An analogous statement with no subclump holds for $\pi$ on the first carrel $(0,q_1]$.  \end{fact}

\begin{prop}\label{prop824.3}Suppose $\pi \in S_n^R$ is $R$-312-avoiding.  Let $\sigma^\prime$ be a 312-avoiding lift of $\pi$.  In terms of the entities above, this lift $\sigma^\prime$ may be obtained from the minimum length 312-avoiding lift $\sigma$ of $\pi$ as follows:  Let $h \in [r]$.  For each of the clumps $L_{e+1}, ... , L_f$ of $B_{h+1}$, its entries in $\sigma$ may be locally rearranged on its support in any 312-avoiding fashion when forming $\sigma^\prime$.  The entries for the subclump $L_e^\prime$ may be locally rearranged on $(q_h, q_h + s^\prime]$ in any 312-avoiding fashion provided that its entries less than $m$ remain in decreasing order.  The entries for each of the clumps of $B_1$ may be locally rearranged as for $L_{e+1}, ... , L_f$.  Conversely, any such rearrangement of the entries of $\sigma$ produces a 312-avoiding lift of $\pi$.  \end{prop}

\begin{proof}The $p_1!p_2!\cdots p_{r+1}!$ lifts of $\pi$ can be obtained from $\sigma$ by forming all rearrangements of its $r+1$ cohorts within their carrels.  Let $\sigma^\prime$ be a 312-avoiding lift of $\pi$.  Let $h \in [r]$.  We continue to refer to the entities above, noting that the analysis of the chain and the clumps for $\pi$ can be used when working with $\sigma$.  According to these clumps, we split the $(h+1)^{st}$ cohort of $\sigma$ into subcohorts, whose supports $(q_h, q_h+s^\prime]$, $(q_h+s^\prime, \cdot ], ... , ( \cdot ,q_{h+1}]$ ``paved'' the $(h+1)^{st}$ carrel $(q_h, q_{h+1}]$.  The entries in each of these subcohorts are smaller than the entries in later subcohorts.  Since these subcohorts correspond to clumps (or a subclump) of the set $B_{h+1}$ for $\pi$, there exist ``gap'' entries of $\pi$ in its later carrels when $h+1 < r+1$.  Now we attempt to create a 312-avoiding permutation $\sigma^\prime$ from $\sigma$:  Intermingling entries among these subcohorts would produce a 312-violation in which the `2' entry would be one of these gap entries.  (Such intermingling is not possible when $h+1 = r+1$ because then there is just one non-empty subcohort for this last carrel.)  So each subcohort must stay on its original support.  If the decreasing order in which the $s$ entries of $\sigma$ that are less than $m$ appeared on $(q_h, q_h + s]$ is changed among themselves as they are rearranged on $(q_h, q_h + s^\prime]$, then a 312-violation would arise in which the entry $m$ from an earlier carrel would be the `3'.  Creating a local 312-violation within one of our subintervals obviously would create a 312-violation for $\sigma^\prime$ as a whole.  These considerations also apply to the first carrel $(0, q_1]$ if one takes $s^\prime := 0$.  We have ruled out all of the rearrangements not permitted by the statement.  Conversely, suppose $\sigma^\prime$ is produced from the 312-avoiding $\sigma$ with one of the permitted rearrangements.  As noted in the proof of Proposition \ref{prop824.1}(i) for $\sigma$, a 312-violation cannot involve three carrels.  Having only the `3' entry come from an earlier carrel can be ruled out as before.  Can a violation with the `31' entries coming from the carrel at hand arise for $\sigma^\prime$?  Since the entries in any new `31' pair created by a permitted rearrangement must come from the same clump, there will not exist a later entry that can serve as the `2'.  And 312-violations within one subinterval are not permitted.  \end{proof}

We will use the following result to derive Theorem 14.1 of \cite{PS} from our Theorem \ref{theorem737.1}(ii):

\begin{prop}\label{prop824.4}Suppose $\pi \in S_n^R$ is $R$-312-avoiding.  Let $\sigma^\prime \in S_n$ be a 312-avoiding lift of $\pi$.  Then $\Delta_R[\Psi(\sigma^\prime)] = \Psi_R(\pi)$ and so $\pi = \Pi_R[\Delta_R[\Psi(\sigma^\prime)]]$.  \end{prop}

\begin{proof}We return to the proofs of Propositions \ref{prop824.2} and \ref{prop824.3}, now knowing by Proposition \ref{prop824.3} how $\sigma^\prime$ can be formed from the minimum length 312-avoiding lift $\sigma$.  Since $\Delta_R[\Psi(\sigma)] = \Psi_R(\pi)$, by Proposition \ref{prop824.2} we only need to show $\Delta_R[\Psi(\sigma^\prime)] = \Delta_R[\Psi(\sigma)]$.  Let $h \in [r]$.  We prepare to compute $\Delta_R[\Psi(\sigma^\prime)]$ within the $(h+1)^{st}$ carrel $(q_h, q_{h+1}]$ by splitting $(q_h, q_{h+1}]$ into the subintervals created for Proposition \ref{prop824.3}.  We then work from the right one subinterval at a time.  For now ignore the subinterval $(q_h, q_h+s^\prime]$ for the subclump $L_e^\prime$.  On the other subintervals, the entries of $\sigma$ formed staircases.  On each of these subintervals the local rearranging for the new entries of $\sigma^\prime$ followed by the application of $\Psi$ produces entries that are index-wise no smaller than the original staircase entries of $\sigma$.  Therefore it can be seen that the subsequent application of $\Delta_R$ to these subintervals one at a time reproduces those staircases of $\sigma$.  Since the application of $\Delta_R \circ \Psi$ did nothing to these entries of $\sigma$ in the proof of Proposition \ref{prop824.2}, on these subintervals we have obtained the desired equality.  This argument also works on the subintervals of $(0,q_1]$.  Returning to the $(h+1)^{st}$ carrel, this argument still works on the right portion $(q_h+s, q_h+s^\prime]$ of the leftmost subinterval $(q_h, q_h+s^\prime]$.  If $s^\prime \geq 1$, it produces an entry of $m+1$ for $\Delta_R[\Psi(\sigma^\prime)]$ at the index $q_h +s+1$.  On the left portion $(q_h, q_h+s]$ of this subinterval, note that following the application of $\Psi$ every entry will be no less than $m$.  So the subsequent application of $\Delta_R$ on this left portion will reproduce the staircase on $(q_h, q_h+s]$ that $\Delta_R[\Psi(\sigma)]$ had in the proof of Proposition \ref{prop824.2}.  Use Proposition \ref{prop320.2}(ii) to produce the second equality.  \end{proof}

\section{Shapes, tableaux, connections to Lie theory}

A \emph{partition} is an $n$-tuple $\lambda \in \mathbb{Z}^n$ such that $\lambda_1 \geq \ldots \geq \lambda_n \geq 0$.  Let $\Lambda_n^+$ denote the set of partitions.  Fix such a $\lambda$ for the rest of the paper.  We say it is \textit{strict} if $\lambda_1 > \ldots > \lambda_n$.  The \textit{shape} of $\lambda$, also denoted $\lambda$, consists of $n$ left justified rows with $\lambda_1, \ldots, \lambda_n$ boxes.  We denote its column lengths by $\zeta_1 \geq \ldots \geq \zeta_{\lambda_1}$.  The column length $n$ is called the \emph{trivial} column length.  Since the columns are more important than the rows, the boxes of $\lambda$ are transpose-indexed by pairs $(j,i)$ such that $1 \leq j \leq \lambda_1$ and $1 \leq i \leq \zeta_j$.  Sometimes for boundary purposes we refer to a $0^{th}$ \emph{latent column} of boxes, which is a prepended $0^{th}$ column of trivial length.  If $\lambda = 0$, its shape is the \textit{empty shape} $\emptyset$.  Define $R_\lambda \subseteq [n-1]$ to be the set of distinct non-trivial column lengths of $\lambda$.  Note that $\lambda$ is strict if and only if $R_\lambda = [n-1]$, i.e. $R$ is full.  Set $|\lambda| := \lambda_1 + \ldots + \lambda_n$.

A \textit{(semistandard) tableau of shape $\lambda$} is a filling of $\lambda$ with values from $[n]$ that strictly increase from north to south and weakly increase from west to east.  Let $\mathcal{T}_\lambda$ denote the set of tableaux of shape $\lambda$.  Under entrywise comparison $\leq$, this set $\mathcal{T}_\lambda$ becomes a poset that is the distributive lattice $L(\lambda, n)$ introduced by Stanley.  The principal ideals in $\mathcal{T}_\lambda$ are clearly convex polytopes in $\mathbb{Z}^{|\lambda|}$.  Fix $T \in \mathcal{T}_\lambda$.  For $j \in [\lambda_1]$, we denote the one column ``subtableau'' on the boxes in the $j^{th}$ column by $T_j$.  Here for $i \in [\zeta_j]$ the tableau value in the $i^{th}$ row is denoted $T_j(i)$.  The set of values in $T_j$ is denoted $B(T_j)$.  Columns $T_j$ of trivial length must be \emph{inert}, that is $B(T_j) = [n]$.  The $0^{th}$ \textit{latent column} $T_0$ is an inert column that is sometimes implicitly prepended to the tableau $T$ at hand:  We ask readers to refer to its values as needed to fulfill definitions or to finish constructions.  We say $T$ is a $\lambda$-\textit{key} if $B(T_l) \supseteq B(T_j)$ for $1 \leq l \leq j \leq \lambda_1$.  To define the \emph{content $\Theta(T) := \theta$ of $T$}, for $i \in [n]$ take $\theta_i$ to be the number of values in $T$ equal to $i$.  The empty shape has one tableau on it, the \textit{null tableau}.  Fix a set $Q \subseteq [n]$ with $|Q| =: q \geq 0$.  The \textit{column} $Y(Q)$ is the tableau on the shape for the partition $(1^q, 0^{n-q})$ whose values form the set $Q$.  Then for $d \in [q]$, the value in the $(q+1-d)^{th}$ row of $Y(Q)$ is $rank^d(Q)$.

Fix a partition $\lambda \in \Lambda_n^+$ and determine the set $R_\lambda$.  For us, the most important values in a tableau of shape $\lambda$ occur at the ends of its rows.  Using the latent column when needed, these $n$ values from $[n]$ are gathered into an $R_\lambda$-tuple as follows:  We group the boxes at the ends of the rows of $\lambda$ into ``cliffs''.  Note that for $h \in [r+1]$ one has $\lambda_i = \lambda_{i^\prime}$ for $i, i^\prime \in (q_{h-1}, q_{h} ]$.  For $h \in [r+1]$ the coordinates of the $p_h$ boxes in the $h^{th}$ \emph{cliff} form the set $\{ (\lambda_i, i) : i \in (q_{h-1}, q_{h} ] \}$.  Let $T \in \mathcal{T}_\lambda$.  The \textit{$\lambda$-row end list} $\Omega_\lambda(T) =: \omega$ of $T$ is the $R_\lambda$-tuple defined by $\omega_i := T_{\lambda_i}(i)$ for $i \in [n]$.  Here for $h \in [r+1]$ the $h^{th}$ cohort of $\omega$ is the set of the values of $T$ that increase down the boxes of the $h^{th}$ cliff.  So $\omega \in UI_{R_\lambda}(n)$.

Let $\pi$ be an $R_\lambda$-permutation and form the corresponding $R_\lambda$-chain $B$.  The $\lambda$-\textit{key} $Y_\lambda(\pi)$ \textit{of} $\pi$ is the tableau of shape $\lambda$ formed by juxtaposing from left to right $\lambda_n$ inert columns and $\lambda_{q_h}-\lambda_{q_{h+1}}$ copies of $Y(B_h)$ for $r \geq h \geq 1$.  The map $\pi \mapsto Y_\lambda(\pi) =: Y$ is a bijection from $R_\lambda$-permutations to $\lambda$-keys that is denoted $\pi \leftrightarrow Y$.  The bijection from $R_\lambda$-chains to $\lambda_{R_\lambda}$-keys is denoted $B \leftrightarrow Y$.  It is easy to see that the $\lambda$-row end list $\Omega_\lambda[Y_\lambda(\pi)]$ of the $\lambda$-key of $\pi$ is the rank $R_\lambda$-tuple $\Psi_{R_\lambda}(\pi) =: \psi$ of $\pi$:  Here $\psi_i = Y_{\lambda_i}(i)$ for $i \in [n]$.

Let $\alpha \in UI_{R_\lambda}(n)$.  Define $\mathcal{Z}_\lambda(\alpha)$ to be the subset of tableaux $T \in \mathcal{T}_\lambda$ such that $\Omega_\lambda(T) = \alpha$.  To see that $\mathcal{Z}_\lambda(\alpha) \neq \emptyset$, for $i \in [n]$ take $T_j(i) := i$ for $j \in [1, \lambda_i)$ and $T_{\lambda_i}(i) := \alpha_i$.  This subset is closed under the join operation for the lattice $\mathcal{T}_\lambda$.  We define the \emph{$\lambda$-row end max tableau $M_\lambda(\alpha)$ for $\alpha$} to be the unique maximal element of $\mathcal{Z}_\lambda(\alpha)$.  The definition of $Q_\lambda(\beta)$, a close relative to $M_\lambda(\alpha)$, can be found in Section 12.

When we are considering tableaux of shape $\lambda$, much of the data used will be in the form of $R_\lambda$-tuples.  Many of the notions used will be definitions from Section 3 that are being applied with $R := R_\lambda$.  The structure of each proof will depend only upon $R_\lambda$ and not upon any other aspect of $\lambda$:  If $\lambda^\prime, \lambda^{\prime\prime} \in \Lambda_n^+$ are such that $R_{\lambda^\prime} = R_{\lambda^{\prime\prime}}$, then the development for $\lambda^{\prime\prime}$ will in essence be the same as for $\lambda^\prime$.  To emphasize the original independent entity $\lambda$ and to reduce clutter, from now on rather than writing `$R$' or `$R_\lambda$' we will replace `$R$' by `$\lambda$' in subscripts and in prefixes.  Above we would have written $\omega \in UI_\lambda(n)$ instead of having written $\omega \in UI_{R_\lambda}(n)$ (and instead of having written $\omega \in UI_R(n)$ after setting $R := R_\lambda$).  When $\lambda$ is a strict partition, we omit the `$\lambda$-' prefixes and the subscripts.

To connect to Lie theory, fix $R \subseteq [n-1]$ and set $J := [n-1] \backslash R$.  The $R$-permutations are the one-rowed forms of the inverses of the minimum length representatives collected in $W^J$ for the cosets in $W /W_J$, where $W$ is the Weyl group of type $A_{n-1}$ and $W_J$ is its parabolic subgroup $\langle s_i: i \in J \rangle$.  Fix a partition $\lambda$.  It is strict exactly when the weight it depicts for $GL(n)$ is strongly dominant.  If we take the set $R$ above to be $R_\lambda$, then the restriction of the partial order $\leq$ on $\mathcal{T}_\lambda$ to the $\lambda$-keys depicts the Bruhat order on that $W^J$.  Consult the second and third paragraphs of Section 14 for the Demazure and flag Schur polynomials.  Further details appear in Sections 2 and 3 and the appendix of \cite{PW1}.

\section{312-Avoiding (gapless) keys, row end max tableaux}

Here we re-express the $R$-permutations with tableaux.

Let $\alpha \in UI_\lambda(n)$.  The values of the $\lambda$-row end max tableau $M_{\lambda}(\alpha) =: M$ can be determined as follows:  For $h \in [r]$ and $j \in (\lambda_{q_{h+1}}, \lambda_{q_h}]$, first set $M_j(i) = \alpha_i$ for $i \in (q_{h-1}, q_h]$.  When $h > 1$, from east to west among columns and south to north within a column, also set $M_j(i) := \min\{ M_j(i+1)-1, \\ M_{j+1}(i) \}$ for $i \in (0, q_{h-1}]$.  Finally, set $M_j(i) := i$ for $j \in (0, \lambda_n]$ and $i \in (0,n]$.  (When $\zeta_j = \zeta_{j+1}$, this process yields $M_j = M_{j+1}$.)

\begin{lem}\label{lemma340.1}Let $\gamma$ be a gapless $\lambda$-tuple.  The $\lambda$-row end max tableau $M_{\lambda}(\gamma) =: M$ is a key.  For $h \in [r]$ and $j := \lambda_{q_{h+1}}$, the $s \geq 0$ elements in $B(M_{j}) \backslash B(M_{j+1})$ that are less than $M_{j+1}(q_{h}) = \gamma_{q_{h}}$ are the $s$ largest elements of $[\gamma_{q_{h}}] \backslash B(M_{j+1})$. \end{lem}

\begin{proof} Let $h \in [r]$ and set $j := \lambda_{q_{h+1}}$.  We claim $B(M_{j+1}) \subseteq B(M_j)$.  If $M_j(q_{h} + 1) = \gamma_{q_{h}+1}  >  \gamma_{q_h} = M_{j+1}(q_{h})$, then $M_j(i) = M_{j+1}(i)$ for $i \in (0, q_{h}]$ and the claim holds.  Otherwise $\gamma_{q_{h} + 1} \leq \gamma_{q_{h}}$.  The gapless condition on $\gamma$ implies that if we start at $(j, q_{h}+1)$ and move south, the successive values in $M_j$ increment by 1 until some lower box has the value $\gamma_{q_{h}}$.  Let $i \in (q_{h}, q_{h+1}]$ be the index such that $M_j(i) = \gamma_{q_{h}}$.  Now moving north from $(j,i)$, the values in $M_j$ decrement by 1 either all of the way to the top of the column, or until there is a row index $k \in (0, q_{h})$ such that $M_{j+1}(k) < M_j(k+1)-1$.  In the former case set $k := 0$.  If $k > 0$ we have $M_j(x) = M_{j+1}(x)$ for $x \in (0,k]$.  Now use  $M_j(k+1) \leq M_{j+1}( k+1)$ to see that the values $M_{j + 1}(k+1), M_{j+1}( k+2), ... , M_{j+1}( q_{h})$ each appear in the interval of values $[ M_j(k+1), M_j(i) ]$.  Thus $B(M_{j+1}) \subseteq B(M_j)$.  Using the parenthetical remark above, we see that $M$ is a key.  There are $q_{h+1} - i$ elements in $B(M_j) \backslash B(M_{j+1})$ that are larger than $M_{j+1}( q_{h}) = \gamma_{q_{h}}$.  So $s := (q_{h+1} - q_{h}) - (q_{h+1} - i) \geq 0$ is the number of values in $B(M_j) \backslash B(M_{j+1})$ that are less than $\gamma_{q_{h}}$.  These $s$ values are the complement in $[ M_j(k+1), M_j(i) ]$ of the set $\{ \hspace{1mm} M_{j+1}(x) : x \in [k+1, q_{h}] \hspace{1mm} \}$, where $M_j(i) = M_{j+1}(q_{h}) = \gamma_{q_{h}}$.  \end{proof}

A $\lambda$-key $Y$ is \textit{gapless} if the condition below is satisfied for $h \in [r-1]$:  Let $b$ be the smallest value in a column of length $q_{h+1}$ that does not appear in a column of length $q_{h}$.  For $j \in (\lambda_{q_{h +2}}, \lambda_{q_{h+1}}]$, let $i \in (0, q_{h+1}]$ be the shared row index for the occurrences of $b = Y_j(i)$.  Let $m$ be the bottom (largest) value in the columns of length $q_{h}$.  If $b > m$ there are no requirements.  Otherwise:  For $j \in (\lambda_{q_{h +2}}, \lambda_{q_{h+1}}]$, let $k \in (i, q_{h+1}]$ be the shared row index for the occurrences of $m = Y_j(k)$.  For $j \in (\lambda_{q_{h + 2}}, \lambda_{q_{h+1}}]$ one must have $Y_j(i+1) = b+1, Y_j(i+2) = b+2, ... , Y_j(k-1) = m-1$ holding between $Y_j(i) = b$ and $Y_j(k) = m$.  (Hence necessarily $m - b = k - i$.)

The bijections $\pi \mapsto B$ and $\Psi_R$ of Proposition \ref{prop320.2} are respectively implicitly present here, from $\mathcal{A}_R$ to $\mathcal{B}_R$ and from $\mathcal{A}_R$ to $\mathcal{C}_R$:

\begin{thm}\label{theorem340}Let $\lambda \in \Lambda_n^+$ and set $R := R_\lambda$.  Consider the following three pairs of sets:

\noindent (a)   The set  $\mathcal{A}_R$  of $R$-312-avoiding permutations and the set $\mathcal{P}_\lambda$  of their $\lambda$-keys.

\noindent (b)   The set  $\mathcal{B}_R$  of $R$-rightmost clump deleting chains and the set  $\mathcal{Q}_\lambda$ of gapless $\lambda$-keys.

\noindent (c)   The set  $\mathcal{C}_R$  of gapless $R$-tuples and the set  $\mathcal{R}_\lambda$  of their $\lambda$-row end max tableaux.

\noindent (i) The process of tableau portrayal is a bijection from $\mathcal{B}_R$ to $\mathcal{Q}_\lambda$ and the process of constructing the $\lambda$-row end max tableau is a bijection from $\mathcal{C}_R$ to $\mathcal{R}_\lambda$.

\noindent (ii) We have $\mathcal{P}_\lambda = \mathcal{Q}_\lambda$.  The restriction of the global bijection $\pi \mapsto B$ to $\mathcal{A}_R$ induces a map from $\mathcal{P}_\lambda$ to $\mathcal{Q}_\lambda$ that is the identity.   So an $R$-permutation is  $R$-312-avoiding if and only if its $\lambda$-key is gapless.

\noindent (iii) If an $R$-permutation is $R$-312-avoiding, then the $\lambda$-row end max tableau of its rank $R$-tuple is its $\lambda$-key.  We have $\mathcal{P}_\lambda = \mathcal{R}_\lambda$.  The map $\Psi_R$ from $\mathcal{A}_R$ to $\mathcal{C}_R$ induces a map from $\mathcal{P}_\lambda$ to $\mathcal{R}_\lambda$ that is the identity.  \end{thm}

\noindent In the full case when $\lambda$ is strict and $R = [n-1]$, the converse of the first statement of Part (iii) holds:  If the row end max tableau of the rank tuple of a permutation is the key of the permutation, then the permutation is 312-avoiding.  A counterexample to this converse for general $\lambda$ appears in Section 15.  The bijection from $\mathcal{C}_R$ to $\mathcal{R}_\lambda$ and the equality $\mathcal{Q}_\lambda = \mathcal{R}_\lambda$ imply that an $R$-tuple is $R$-gapless if and only if it arises as the $\lambda$-row end list of a gapless $\lambda$-key.

\begin{proof}

\noindent For the first part of (i), use the $B \mapsto Y$ bijection to relate Fact \ref{fact320.3}(i) to the definition of gapless $\lambda$-key.  The map in the second part is surjective by definition and is also obviously injective.  Use the construction of the bijection $\pi \mapsto B$ from $\mathcal{A}_R$ to $\mathcal{B}_R$ and the first part of (i) to confirm (ii).

Let $\pi \in S_n^{R\text{-}312}$.  Create the $R$-chain $B$ corresponding to $\pi$ and then its $\lambda$-key $Y := Y_\lambda(\pi)$.  Set $\gamma := \Psi_R(\pi)$ and then $M := M_\lambda(\gamma)$.  Clearly $B(Y_{\lambda_v}) = B_1 = \{ \gamma_1, ... , \gamma_v \} = B(M_{\lambda_v})$ for $v := q_1$.  Proceed by induction on $h \in [r]$:  For $v := q_h$ assume $B(Y_{\lambda_v}) = B(M_{\lambda_v})$.  Note that $Y_{\lambda_v}(v) = \max[B(Y_{\lambda_v})]$ and $\gamma_v = M_{\lambda_v}(v)$.  Proposition \ref{prop320.2}(ii) says that $\gamma$ is $R$-gapless;  this implies $\max[B(Y_{\lambda_v})] = \gamma_v$.  Set $v^\prime := q_{h+1}$.  Let $s_B$ be the number of values in $B(Y_{\lambda_{v^\prime}}) \backslash B(Y_{\lambda_v})$ that are less than $\gamma_v$.  Since $\gamma_v \in B(Y_{\lambda_v})$, the number of values in $B(Y_{\lambda_{v^\prime}}) \backslash B(Y_{\lambda_v})$ that exceed $\gamma_v$ is $p_{h+1} - s_B$.  These values are the entries in $\{ \pi_{v+1} , ... , \pi_{v^\prime} \}$ that exceed $\gamma_v$.  So from $\gamma := \Psi_R(\pi)$ and the description of $M_\lambda(\gamma)$ it can be seen that these values are exactly the values in $B(M_{\lambda_{v^\prime}}) \backslash B(M_{\lambda_v})$ that exceed $\gamma_v$.  Since $M$ is a key by Lemma \ref{lemma340.1} and $\gamma_v \in B(M_{\lambda_v})$, the number $s_M$ of values in $B(M_{\lambda_{v^\prime}}) \backslash B(M_{\lambda_v})$ that are less than $\gamma_v$ is $p_{h+1} - (p_{h+1}-s_B) = s_B =: s$.  From Proposition \ref{prop320.2}(i) we know that $B$ is $R$-rightmost clump deleting.  By Fact \ref{fact320.3}(iii) applied to $B$ and Lemma \ref{lemma340.1} applied to $\gamma$, we see that for both $Y$ and for $M$ the ``new'' values that are less than $\gamma_v$ are the $s$ largest elements of $[\gamma_v] \backslash B(Y_{\lambda_v}) = [\gamma_v] \backslash B(M_{\lambda_v})$.  Hence $Y_{\lambda_{v^\prime}} = M_{\lambda_{v^\prime}}$.  Since we only need to consider the rightmost columns of each length when showing that two $\lambda$-keys are equal, we have $Y = M$.  The rest of (iii) is evident.  \end{proof}

\begin{cor}When $\lambda$ is strict, there are $C_n$ gapless $\lambda$-keys. \end{cor}

\section{Sufficient condition for Demazure convexity}

Fix a $\lambda$-permutation $\pi$.  We define the set $\mathcal{D}_\lambda(\pi)$ of Demazure tableaux.  Then we show that if $\pi$ is $\lambda$-312-avoiding one has $\mathcal{D}_\lambda(\pi) = [Y_\lambda(\pi)]$.

First we need to specify how to find the \emph{scanning tableau} $S(T)$ for a given $T \in \mathcal{T}_\lambda$.   See page 394 of \cite{Wi2} for an example of this method.  Given a sequence $x_1, x_2, ...$, its \emph{earliest weakly increasing subsequence (EWIS)} is $x_{i_1}, x_{i_2}, ...$, where $i_1 = 1$ and for $u > 1$ the index $i_u$ is the smallest index satisfying $x_{i_u} \geq x_{i_{u-1}}$.  Let $T \in \mathcal{T}_\lambda$.  Draw the shape $\lambda$ and fill its boxes as follows to produce $S(T)$:  Form the sequence of the bottom values of the columns of $T$ from left to right.  Find the EWIS of this sequence, and mark each box that contributes its value to this EWIS.  The sequence of locations of the marked boxes for a given EWIS is its \emph{scanning path}.  Place the final value of this EWIS in the lowest available location in the leftmost available column of $S(T)$.  This procedure can be repeated as if the marked boxes are no longer part of $T$, since it can be seen that the unmarked locations form the shape of some $n$-partition.  Ignoring the marked boxes, repeat this procedure to fill in the next value of $S(T)$.  Once all of the scanning paths originating in the first column have been found, every location in $T$ has been marked and the first column of $S(T)$ has been created.  For $j> 1$, to fill in the $j^{th}$ column of $S(T)$:  Ignore the leftmost $(j-1)$ columns of $T$, remove all of the earlier marks from the other columns, and repeat the above procedure.  The scanning path originating at a location $(l,k) \in \lambda$ is denoted $\mathcal{P}(T;l,k)$.  It was shown in \cite{Wi2} that $S(T)$ is the ``right key'' of Lascoux and Sch\"{u}tzenberger for $T$, which was denoted $R(T)$ there.

As in \cite{PW1}, we now use the $\lambda$-key $Y_\lambda(\pi)$ of $\pi$ to define the set of \emph{Demazure tableaux}:  $\mathcal{D}_\lambda(\pi) :=$ \\ $\{ T \in \mathcal{T}_\lambda : S(T) \leq Y_\lambda(\pi) \}$.  We list some basic facts concerning keys, scanning tableaux, and sets of Demazure tableaux.  Part (i) is elementary.  Parts (ii) and (iii) either appear in \cite{Wi2}, \cite{PW1}, and/or \cite{Wi3}, or can be deduced from results therein using $S(T) = R(T)$.  The remaining parts follow in succession from Part (iii).

\begin{fact}\label{fact420}Let $T \in \mathcal{T}_\lambda$ and let $Y \in \mathcal{T}_\lambda$ be a key.

\noindent (i)  If $\Theta(Y) = \Theta(U)$ for some $U \in \mathcal{T}_\lambda$, then $U = Y$.

\noindent (ii)  $S(T)$ is a key.

\noindent (iii)  $T \leq S(T)$ and $S(Y) = Y$.

\noindent (iv)  $Y_\lambda(\pi) \in \mathcal{D}_\lambda(\pi)$ and $\mathcal{D}_\lambda(\pi) \subseteq [Y_\lambda(\pi)]$.

\noindent (v)  The unique maximal element of $\mathcal{D}_\lambda(\pi)$ is $Y_\lambda(\pi)$.

\noindent (vi)  The Demazure sets $\mathcal{D}_\lambda(\sigma)$ of tableaux are nonempty subsets of $\mathcal{T}_\lambda$ that are precisely indexed by the $\sigma \in S_n^\lambda$.  \end{fact}

For $U \in \mathcal{T}_\lambda$, define $m(U)$ to be the maximum value in $U$.  (Define $m(U) := 1$ if $U$ is the null tableau.)  Let $(l,k) \in \lambda$.  As in Section 4 of \cite{PW1}, define $U^{(l,k)}$ to be the tableau formed from $T$ by finding and removing the scanning paths that begin at $(l,\zeta_l)$ through $(l, k+1)$, and then removing the $1^{st}$ through $l^{th}$ columns of $T$.  (If $l = \lambda_1$, then $U^{(l,k)}$ is the null tableau for any $k \in [\zeta_{\lambda_1}]$.)  Set $S := S(T)$.  Lemma 4.1 of \cite{PW1} states that $S_l(k) = \text{max} \{ T_l(k), m(U^{(l,k)}) \}$.

To reduce clutter in the proofs we write $Y_\lambda(\pi) =: Y$.

\begin{prop}\label{prop420.1}Let $\pi \in S^\lambda_n$ and $T \in \mathcal{T}_\lambda$ be such that $T \leq Y_\lambda(\pi)$.  If there exists $(l,k) \in \lambda$ such that $Y_l(k) < m(U^{(l,k)})$, then $\pi$ is $\lambda$-312-containing.  \end{prop}

\begin{proof}Reading the columns from right to left and then each column from bottom to top, let $(l,k)$ be the first location in $\lambda$ such that $m(U^{(l,k)}) > Y_l(k)$.  In the rightmost column we have $m(U^{(\lambda_1,i)}) = 1$ for all $i \in [\zeta_{\lambda_1}]$.  Thus $m(U^{(\lambda_1,i)}) \leq Y_{\lambda_1}(i)$ for all $i \in [\zeta_{\lambda_1}]$.  So we must have $l \in [1, \lambda_1)$.

There exists $j > l$ and $i \leq k$ such that $m(U^{(l,k)}) = T_j(i)$.  Since $T \leq Y$, so far we have $Y_l(k) < T_j(i) \leq Y_j(i)$.  Note that since $Y$ is a key we have $k < \zeta_l$.  Then for $k < f \leq \zeta_l$ we have $m(U^{(l,f)}) \leq Y_l(f)$.  So $T \leq Y$ implies that $S_l(f) \leq Y_l(f)$ for $k < f \leq \zeta_l$.

Assume for the sake of contradiction that $\pi$ is $\lambda$-312-avoiding.  Theorem \ref{theorem340}(ii) says that its $\lambda$-key $Y$ is gapless.  If the value $Y_l(k)$ does not appear in $Y_j$, then the columns that contain $Y_l(k)$ must also contain $[Y_l(k), Y_j(i)]$:  Otherwise, the rightmost column that contains $Y_l(k)$ has index $\lambda_{q_{h+1}}$ for some $h \in [r-1]$ and there exists some $u \in [Y_l(k), Y_j(i)]$ such that $u \notin Y_{\lambda_{q_{h+1}}}$.  Then $Y$ would not satisfy the definition of gapless $\lambda$-key, since for this $h+1$ in that definition one has $b \leq u$ and $u \leq m$.  If the value $Y_l(k)$ does appear in $Y_j$, it appears to the north of $Y_j(i)$ there.  Then $i \leq k$ implies that some value $Y_l(f) < Y_j(i)$ with $f < k$ does not appear in $Y_j$.  As above, the columns that contain the value $Y_l(f) < Y_l(k)$ must also contain $[Y_l(f), Y_j(i)]$.  In either case $Y_l$ must contain $[Y_l(k), Y_j(i)]$.  This includes $T_j(i)$.

Now let $f > k$ be such that $Y_l(f) = T_j(i)$.  Then we have $S_l(f) > S_l(k) = \text{max} \{T_l(k),m(U^{(l,k)}) \}$ $\geq T_j(i) = Y_l(f)$.  This is our desired contradiction.  \end{proof}

As in Section 5 of \cite{PW1}: When $m(U^{(l,k)}) > Y_l(k)$, define the set $A_\lambda(T,\pi;l,k) := \emptyset$.  Otherwise, define $A_\lambda(T,\pi;l,k) := [ k , \text{min} \{ Y_l(k), T_l(k+1) -1, T_{l+1}(k) \} ] $.  (Refer to fictitious bounding values $T_l(\zeta_l + 1) := n+1$ and $T_{\lambda_l + 1}(l) := n$.)

\begin{thm}\label{theorem420}Let $\lambda$ be a partition and $\pi$ be a $\lambda$-permutation.  If $\pi$ is $\lambda$-312-avoiding, then $\mathcal{D}_\lambda(\pi) = [Y_\lambda(\pi)]$.  \end{thm}

\begin{proof}Let $T \leq Y$ and $(l,k) \in \lambda$.  The contrapositive of the proposition gives $A_\lambda(T,\pi;l,k) =$ \\ $[ k , \text{min} \{ Y_l(k), T_l(k+1) - 1, T_{l+1}(k) \} ]$.  Since $T \leq Y$, we see that $T_l(k) \in A_\lambda(T,\pi;l,k)$ for all $(l,k) \in \lambda$.  Theorem 5.1 of \cite{PW1} says that $T \in \mathcal{D}_\lambda(\pi)$.  \end{proof}

Since principal ideals in $\mathcal{T}_\lambda$ are convex polytopes in $\mathbb{Z}^{|\lambda|}$, we immediately obtain:

\begin{cor}\label{cor420}If $\pi$ is $\lambda$-312-avoiding, then $\mathcal{D}_\lambda(\pi)$ is a convex polytope in $\mathbb{Z}^{|\lambda|}$.  \end{cor}

\section{Necessary condition for Demazure convexity}

Continue to fix a $\lambda$-permutation $\pi$.  We show that $\pi$ must be $\lambda$-312-avoiding for the set of Demazure tableaux $\mathcal{D}_\lambda(\pi)$ to be a convex polytope in $\mathbb{Z}^{|\lambda|}$.  We do so by showing that if $\pi$ is $\lambda$-312-containing, then $\mathcal{D}_\lambda(\pi)$ does not contain a particular semistandard tableau that lies on the line segment defined by two particular keys that are in $\mathcal{D}_\lambda(\pi)$.

\begin{thm}\label{theorem520}Let $\lambda$ be a partition and let $\pi$ be a $\lambda$-permutation.  If $\mathcal{D}_\lambda(\pi)$ is the principal ideal $[Y_\lambda(\pi)]$ in $\mathcal{T}_\lambda$, then $\pi$ is $\lambda$-312-avoiding.  More generally:  If $\mathcal{D}_\lambda(\pi)$ is convex in $\mathbb{Z}^{|\lambda|}$, then $\pi$ is $\lambda$-312-avoiding.  \end{thm}

\begin{proof}For the contrapositive, assume that $\pi$ is $\lambda$-312-containing.  Here $|R_\lambda| =: r \geq 2$.  There exists $1 \leq g < h \leq r$ and some $a \leq q_g < b \leq q_h < c$ such that $\pi_b < \pi_c < \pi_a$.  Among such patterns, we specify one that is optimal for our purposes.  Figure 11.1 charts the following choices for $\pi = (4,8;9;2,3;1,5;6,7)$ in the first quadrant.  Choose $h$ to be minimal.  So $b \in (q_{h-1}, q_h]$.  Then choose $b$ so that $\pi_b$ is maximal.  Then choose $a$ so that $\pi_a$ is minimal.  Then choose $g$ to be minimal.  So $a \in (q_{g-1} , q_g]$.  Then choose any $c$ so that $\pi_c$ completes the $\lambda$-312-containing condition.

These choices have led to the following two prohibitions; see the rectangular regions in Figure 11.1:

\noindent (i) By the minimality of $h$ and the maximality of $\pi_b$, there does not exist $e \in (q_g, q_h]$ such that $\pi_b < \pi_e < \pi_c$.

\noindent (ii) By the minimality of $\pi_a$, there does not exist $e \in [q_{h-1}]$ such that $\pi_c < \pi_e < \pi_a$.

\noindent If there exists $e \in [q_g]$ such that $\pi_b < \pi_e < \pi_c$, choose $d \in [q_g]$ such that $\pi_d$ is maximal with respect to this condition; otherwise set $d = b$.  So $\pi_b \leq \pi_d$ with $d \leq b$.  We have also ruled out:

\noindent (iii) By the maximality of $\pi_d$, there does not exist $e \in [q_g]$ such that $\pi_d < \pi_e < \pi_c$.

\vspace{1.5pc}\centerline{\includegraphics[scale=1]{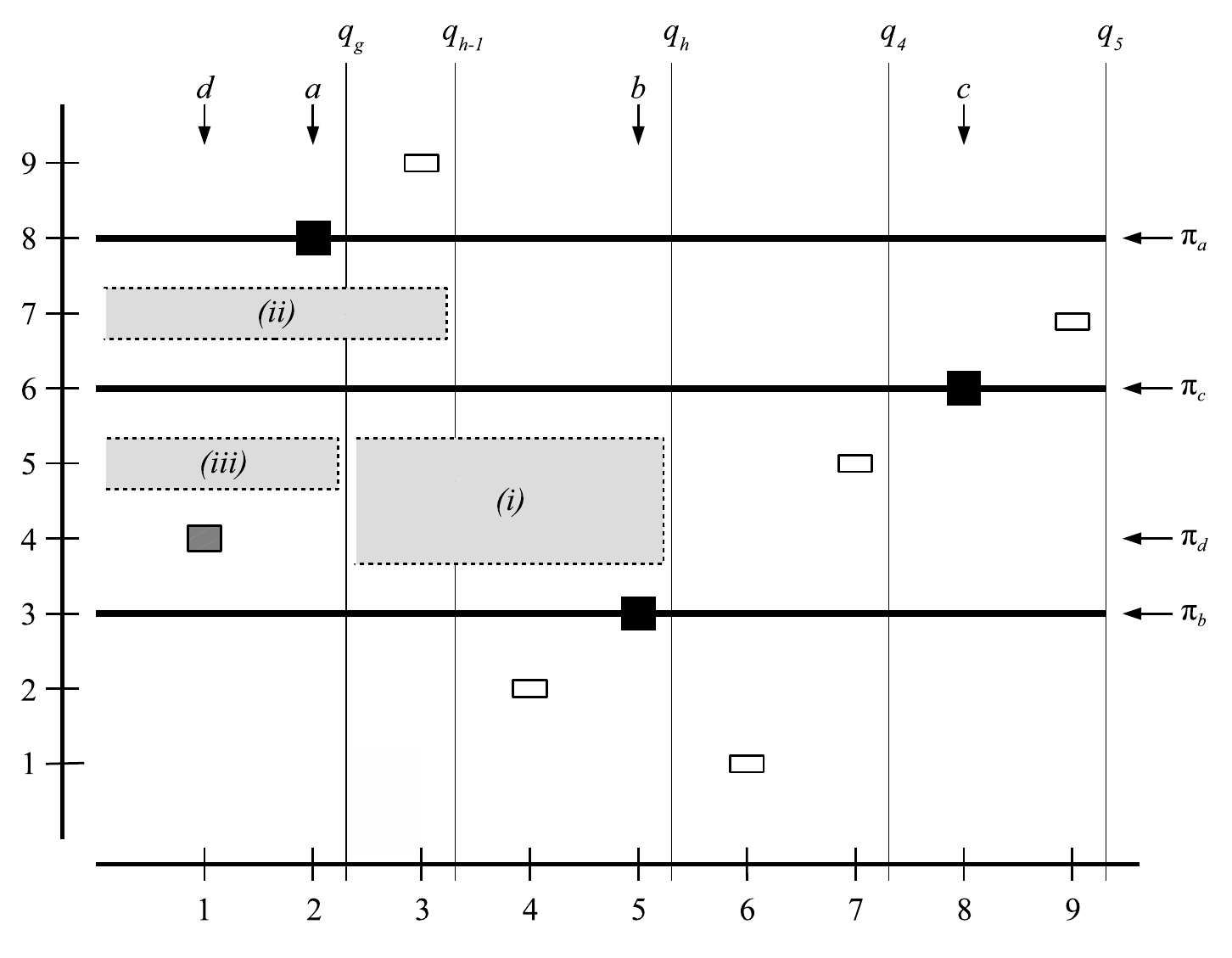}}

\vspace{0pc}\centerline{Figure 11.1.  Prohibited regions (i), (ii), and (iii) for $\pi = (4,8;9;2,3;1,5;6,7)$.}

\vspace{2pc}Set $Y := Y_\lambda(\pi)$.  Now let $\chi$ be the permutation resulting from swapping the entry $\pi_b$ with the entry $\pi_d$ in $\pi$; so $\chi_b := \pi_d, \chi_d := \pi_b$, and $\chi_e := \pi_e$ when $e \notin \{ b, d \}$.  (If $d = b$, then $\chi = \pi$ with $\chi_b = \pi_b = \chi_d = \pi_d$.)  Let $\bar{\chi}$ be the $\lambda$-permutation produced from $\chi$ by sorting each cohort into increasing order.  Set $X := Y_\lambda(\bar{\chi})$.  Let $j$ denote the column index of the rightmost column with length $q_h$; so the value $\chi_b = \pi_d$ appears precisely in the $1^{st}$ through $j^{th}$ columns of $X$.  Let $f \leq h$ be such that $d \in (q_{f-1}, q_f]$, and let $k \geq j$ denote the column index of the rightmost column with length $q_f$.  The swap producing $\chi$ from $\pi$ replaces $\pi_d = \chi_b$ in the $(j+1)^{st}$ through $k^{th}$ columns of $Y$ with $\chi_d = \pi_b$ to produce $X$.  (The values in these columns may need to be re-sorted to meet the semistandard criteria.)  So $\chi_d \leq \pi_d$ implies $X \leq Y$ via a column-wise argument.

Forming the union of the prohibited rectangles for (i), (ii), and (iii), we see that there does not exist $e \in [q_{h-1}]$ such that $\pi_d = \chi_b < \pi_e < \pi_a$.  Thus we obtain:

\noindent (iv) For $l > j$, the $l^{th}$ column of $X$ does not contain any values from $[\chi_b, \pi_a)$.

\noindent Let $(j,i)$ denote the location of the $\chi_b$ in the $j^{th}$ column of $X$ (and hence $Y$).  So $Y_j(i) = \pi_d$.  By (iv) and the semistandard conditions, we have $X_{j+1}(u) = \pi_a$ for some $u \leq i$.  By (i) and (iii) we can see that $X_j(i+1) > \pi_c$.

Let $m$ denote the column index of the rightmost column of $\lambda$ with length $q_g$.  This is the rightmost column of $X$ that contains $\pi_a$.  Let $\mu \subseteq \lambda$ be the set of locations of the $\pi_a$'s in the $(j+1)^{st}$ through $m^{th}$ columns of $X$; note that $(j+1, u) \in \mu$.  Let $\omega$ be the permutation obtained by swapping $\chi_a = \pi_a$ with $\chi_b = \pi_d$ in $\chi$;  so $\omega_a := \chi_b = \pi_d$, $\omega_b := \chi_a = \pi_a$, $\omega_d := \chi_d = \pi_b$, and $\omega_e := \pi_e$ when $e \notin \{ d, a, b \}$.  Let $\bar{\omega}$ be the $\lambda$-permutation produced from $\omega$ by sorting each cohort into increasing order.  Set $W := Y_\lambda(\bar{\omega})$.  By (iv), obtaining $\omega$ from $\chi$ is equivalent to replacing the $\pi_a$ at each location of $\mu$ in $X$ with $\chi_b$ (and leaving the rest of $X$ unchanged) to obtain $W$.  So $\chi_b < \pi_a$ implies $W < X$.

\vspace{1.5pc}\centerline{\includegraphics[scale=1]{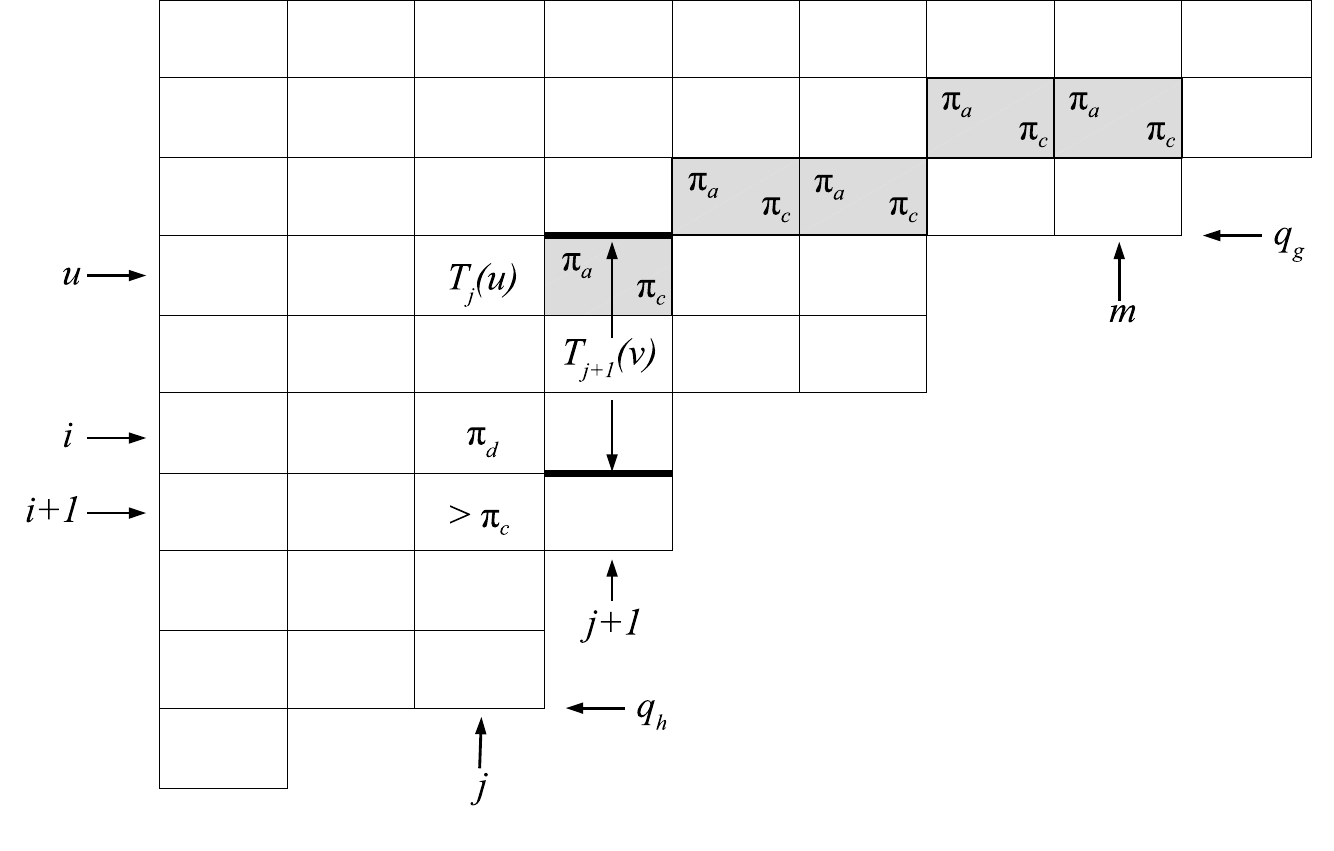}}

\vspace{-.5pc}\centerline{Figure 11.2.  Values of $X$ (respectively $T$) are in upper left (lower right) corners.}

\vspace{2pc}Let $T$ be the result of replacing the $\pi_a$ at each location of $\mu$ in $X$ with $\pi_c$ (and leaving the rest unchanged).  So $T < X \leq Y$.  See the conceptual Figure 11.2 for $X$ and $T$; the shaded boxes form $\mu$.  In particular $T_{j+1}(u) = \pi_c$.  This $T$ is not necessarily a key; we need to confirm that it is semistandard.  For every $(q,p) \notin \mu$ we have $W_q(p) = T_q(p) = X_q(p)$.  By (iv), there are no values in $X$ in any column to the right of the $j^{th}$ column from $[\pi_c, \pi_a)$.  The region $\mu$ is contained in these columns.  Hence we only need to check semistandardness when moving from the $j^{th}$ column to $\mu$ in the $(j+1)^{st}$ column.  Here $u \leq i$ implies $T_j(u) \leq T_j(i) = \pi_d < \pi_c = T_{j+1}(u)$.  So $T \in \mathcal{T}_\lambda$.

Now we consider the scanning tableau $S(T) =: S$ of $T$:  Since $(j, i+1) \notin \mu$, we have $T_j(i+1) = X_j(i+1)$.  Since $X_j(i+1) > \pi_c = T_{j+1}(u)$, the location $(j+1,u)$ is not in a scanning path $\mathcal{P}(T;j,i^\prime)$ for any $i^\prime > i$.  Since $T_j(i) = \chi_b = \pi_d < \pi_c$, the location $(j+1,v)$ is in $\mathcal{P}(T;j,i)$ for some $v \in [u,i]$.  By the semistandard column condition one has $T_{j+1}(v) \geq T_{j+1}(u) = \pi_c$.  Thus $S_j(i) \geq \pi_c > \chi_b = \pi_d = Y_j(i)$.  Hence $S(T) \nleq Y$, and so $T \notin \mathcal{D}_\lambda(\pi)$.  Since $T \in [Y]$, we have $\mathcal{D}_\lambda(\pi) \neq [Y]$.

In $\mathbb{R}^{|\lambda|}$, consider the line segment $U(t) = W + t(X-W)$, where $0 \leq t \leq 1$.  Here $U(0) = W$ and $U(1) = X$.  The value of $t$ only affects the values at the locations in $\mu$.  Let $x := \frac{\pi_c - \chi_b}{\pi_a - \chi_b}$.  Since $\chi_b < \pi_c < \pi_a$, we have $0 < x < 1$.  The values in $\mu$ in $U(x)$ are $\chi_b + \frac{\pi_c - \chi_b}{\pi_a-\chi_b}(\pi_a-\chi_b) = \pi_c$.  Hence $U(x) = T$.  Since $X$ and $W$ are keys, we have $S(X) = X$ and $S(W) = W$.  Then $W < X \leq Y$ implies $W \in \mathcal{D}_\lambda(\pi)$ and $X \in \mathcal{D}_\lambda(\pi)$.  Thus $U(0), U(1) \in \mathcal{D}_\lambda(\pi)$ but $U(x) \notin \mathcal{D}_\lambda(\pi)$.  If a set $\mathcal{E}$ is a convex polytope in $\mathbb{Z}^N$ and $U(t)$ is a line segment with $U(0), U(1) \in \mathcal{E}$, then $U(t) \in \mathcal{E}$ for any $0 < t < 1$ such that $U(t) \in \mathbb{Z}^N$.  Since $0 < x < 1$ and $U(x) = T \in \mathbb{Z}^{|\lambda|}$ with $U(x) \notin \mathcal{D}_\lambda(\pi)$, we see that $\mathcal{D}_\lambda(\pi)$ is not a convex polytope in $\mathbb{Z}^{|\lambda|}$.  \end{proof}

When one first encounters the notion of a Demazure polynomial $d_\lambda(\pi;x)$, given Facts \ref{fact420}(iv)(v) it is natural to hope that $d_\lambda(\pi;x)$ is simply the sum of $x^{\Theta(T)}$ over all $T \in [Y_\lambda(\pi)]$.  Combining Theorems \ref{theorem420} and \ref{theorem520}, we can now say:

\begin{cor}\label{cor520}Let $\pi \in S_n^\lambda$. The set $\mathcal{D}_\lambda(\pi)$ of Demazure tableaux of shape $\lambda$ is a convex polytope in $\mathbb{Z}^{|\lambda|}$ if and only if $\pi$ is $\lambda$-312-avoiding if and only if $\mathcal{D}_\lambda(\pi) = [Y_\lambda(\pi)]$.  \end{cor}

\noindent When $\lambda$ is the strict partition $(n,n-1, ..., 2,1)$, this convexity result appeared as Theorem 3.9.1 in \cite{Wi1}.

\section{Sets of tableaux specified by row bounds}

We use some of the $R$-tuples studied in Section 4 and 5 to develop precise indexing schemes for row bound tableau sets.

Determine the subset $R_\lambda \subseteq [n-1]$ for our fixed partition $\lambda$.  We must temporarily suspend our notation shortcuts regarding `$R_\lambda$'.  Let $\beta$ be an $R_\lambda$-tuple.  We define the \emph{row bound set of tableaux} to be $\mathcal{S}_\lambda(\beta) := \{ T \in \mathcal{T}_\lambda : T_j(i) \leq \beta_i \text{ for } j \in [0, \lambda_1] \text{ and } i \in [\zeta_j] \}$.  In Section 8 for $\alpha \in UI_\lambda(n)$ we set $\mathcal{Z}_\lambda(\alpha) := \{ T \in \mathcal{T}_\lambda : T_{\lambda_i}(i) = \alpha_i \text{ for } i \in [n] \}$.  Set $\delta := \Delta_{R_\lambda}(\beta)$ and note that $\delta \leq \beta$ is upper if and only if $\beta$ is upper.  The set $\mathcal{S}_\lambda(\beta)$ is empty if and only if $\beta$ fails to be upper:  This condition is clearly necessary for $\mathcal{S}_\lambda(\beta) \neq \emptyset$;  for sufficiency since $\delta \in UI_{R_\lambda}(n)$ we can re-use the $T \in \mathcal{T}_\lambda$ given in Section 8 to see that $\emptyset \neq \mathcal{Z}_\lambda(\delta) \subseteq \mathcal{S}_\lambda(\beta)$.  Henceforth we assume that $\beta$ is upper:  $\beta \in U_{R_\lambda}(n)$.  To interface with the literature for flagged Schur functions, we give special attention to the \emph{flag bound sets} $\mathcal{S}_\lambda(\varphi)$ for upper flags $\varphi \in UF_{R_\lambda}(n)$.  We also want to name the row bound sets $\mathcal{S}_\lambda(\eta)$ for $\eta \in UGC_{R_\lambda}(n)$; we call these the \emph{gapless core bound sets}.

We can focus on the row ends of the tableaux at hand because $\mathcal{S}_\lambda(\beta) = \{ T \in \mathcal{T}_\lambda : T_{\lambda_i}(i) \leq \beta_i \text{ for } i \in [n] \}$.  Let $\alpha \in U_{R_\lambda}(n)$.  In Section 8 we noted that $\mathcal{Z}_\lambda(\alpha) \neq \emptyset$ if and only if $\alpha \in UI_{R_\lambda}(n)$.  These $\mathcal{Z}_\lambda(\alpha)$ are disjoint for distinct $\alpha \in UI_{R_\lambda}(n)$.  Clearly $\mathcal{S}_\lambda(\beta) = \bigcup \mathcal{Z}_\lambda(\alpha)$, taking the union over the $\alpha$ in the subset $\{ \beta \}_{R_\lambda} \subseteq UI_{R_\lambda}(n)$ defined in Section 5.  This observation and Lemma \ref{lemma608.2} allow us to study the three kinds of row bound sets $\mathcal{S}_\lambda(\beta)$ by considering the principal ideals $[\Delta_{R_\lambda}(\beta)] = \{ \beta \}_{R_\lambda}$ of $UI_{R_\lambda}(n)$ for $\beta \in U_{R_\lambda}(n)$ or $\beta \in UGC_{R_\lambda}(n)$ or $\beta \in UF_{R_\lambda}(n)$.  The results of Sections 4 and 5 can be used to show:

\begin{prop}\label{prop623.1}Let $\beta \in U_{R_\lambda}(n)$.  The row bound set $\mathcal{S}_\lambda(\beta)$ is a flag bound tableau set if and only if $\beta \in UGC_{R_\lambda}(n)$.  For the ``if'' statement use $\mathcal{S}_\lambda(\beta) = \mathcal{S}_\lambda(\varphi)$ for $\varphi := \Phi_{R_\lambda}[\Delta_{R_\lambda}(\beta)]$. \end{prop}

\noindent At times for indexing reasons we will prefer the ``gapless core'' viewpoint.

When $\lambda$ is not strict, it is possible to have $\mathcal{S}_\lambda(\beta) = \mathcal{S}_\lambda(\beta^\prime)$ for distinct $\beta, \beta^\prime \in U_{R_\lambda}(n)$.  We want to study how much such labelling ambiguity is present for $\mathcal{S}_\lambda(.)$, and we want to develop unique labelling systems for the tableau sets $\mathcal{S}_\lambda(\beta), \mathcal{S}_\lambda(\eta)$, and $\mathcal{S}_\lambda(\varphi)$.  For $\beta, \beta^\prime \in U_{R_\lambda}(n)$, define $\beta \approx_{\lambda} \beta^\prime$ if $\mathcal{S}_\lambda(\beta) = \mathcal{S}_\lambda(\beta^\prime)$.  Sometimes we restrict $\approx_\lambda$ to $UGC_{R_\lambda}(n)$ or further to $UF_{R_\lambda}(n)$.  We denote the equivalence class of $\beta \in U_{R_\lambda}(n), \eta \in UGC_{R_\lambda}(n)$, and $\varphi \in UF_{R_\lambda}(n)$ by $\langle \beta \rangle_\lambda, \langle \eta \rangle_\lambda^G$, and $\langle \varphi \rangle_\lambda^F$.  By the $\mathcal{S}_\lambda(\beta) = \bigcup \mathcal{Z}_\lambda(\alpha)$ observation and the fact that the $\mathcal{Z}_\lambda(\alpha)$ are non-empty and disjoint, it can be seen that this relation $\approx_\lambda$ on $U_{R_\lambda}(n)$ or $UGC_{R_\lambda}(n)$ or $UF_{R_\lambda}(n)$ is the same as the relation $\sim_{R_\lambda}$ defined on these sets in Section 5:  Each relation can be expressed in terms of principal ideals of $UI_{R_\lambda}(n)$.

\begin{prop}\label{prop623.2}On the sets $U_{R_\lambda}(n), UGC_{R_\lambda}(n)$, and $UF_{R_\lambda}(n)$, the relation $\approx_\lambda$ coincides with the relation $\sim_{R_\lambda}$.  \end{prop}

\noindent Now that we know that these relations coincide, we can safely return to replacing `$R_\lambda$' with `$\lambda$' in subscripts and prefixes.  Definitions and results from Sections 4 and 5 will be used by always taking $R := R_\lambda$.  In particular, the three (five) unique labelling systems listed in Corollary \ref{cor608.6} for the equivalence classes of $\sim_{\lambda}$ can now be used for the equivalence classes of $\approx_\lambda$.  Henceforth we more simply write `$\sim$' for `$\approx_\lambda$'.

We state some applications of the work in Sections 4 and 5 to the current context:

\begin{prop}\label{prop623.4}Below we take $\beta, \beta^\prime \in U_\lambda(n)$ and $\eta, \eta' \in UGC_\lambda(n)$ and $\varphi, \varphi^\prime \in UF_\lambda(n)$:

\noindent (i)  The row bound sets $\mathcal{S}_\lambda(\beta)$ are precisely indexed by the $\lambda$-increasing upper tuples $\alpha \in UI_\lambda(n)$, which are the minimal representatives in $U_\lambda(n)$ for the equivalence classes $\langle \beta \rangle_\lambda$.  One has $\mathcal{S}_\lambda(\beta) = \mathcal{S}_\lambda(\beta^\prime)$ if and only if $\beta \sim \beta^\prime$ if and only if $\Delta_\lambda(\beta) = \Delta_\lambda(\beta^\prime)$.

\noindent (ii)  The gapless core bound sets $\mathcal{S}_\lambda(\eta)$ are precisely indexed by the gapless $\lambda$-tuples $\gamma \in UG_\lambda(n)$, which are the minimal representatives in $UGC_\lambda(n)$ for the equivalence classes $\langle \eta \rangle_\lambda^G = \langle \eta \rangle_\lambda$.  One has $\mathcal{S}_\lambda(\eta) = \mathcal{S}_\lambda(\eta^\prime)$ if and only if $\eta \sim \eta^\prime$ if and only if $\Delta_\lambda(\eta) = \Delta_\lambda(\eta^\prime)$.

\noindent (iii)  The flag bound sets $\mathcal{S}_\lambda(\varphi)$ are precisely indexed by the $\lambda$-floor flags $\tau \in UFlr_\lambda(n)$, which are the minimal representatives in $UF_\lambda(n)$ for the equivalence classes $\langle \varphi \rangle_\lambda^F$.  One has $\mathcal{S}_\lambda(\varphi) = \mathcal{S}_\lambda(\varphi^\prime)$ if and only if $\varphi \sim \varphi^\prime$ if and only if  $\Phi_\lambda[\Delta_\lambda(\varphi)] = \Phi_\lambda[\Delta_\lambda(\varphi^\prime)]$.  The flag bound sets $\mathcal{S}_\lambda(\varphi)$ can also be faithfully depicted as the sets $\mathcal{S}_\lambda(\gamma)$ for $\gamma \in UG_\lambda(n)$ by taking $\gamma := \Delta_\lambda(\varphi)$.  \end{prop}

Let $\beta \in U_\lambda(n)$.  Following Theorem 23 of \cite{RS}, we define the \emph{$\lambda$-row bound max tableau} $Q_\lambda(\beta)$ to be the least upper bound in $\mathcal{T}_\lambda$ of the tableaux in $\mathcal{S}_\lambda(\beta)$.  It can be seen that $Q_\lambda(\beta) \in  \mathcal{S}_\lambda(\beta)$.  Let $\alpha \in UI_\lambda(n)$.  Recall that the $\lambda$-row end max tableau $M_\lambda(\alpha)$ is the least upper bound in $\mathcal{T}_\lambda$ of the tableaux in $\mathcal{Z}_\lambda(\alpha)$.

\begin{prop}\label{prop623.8}Let $\beta, \beta^\prime \in U_\lambda(n)$ and set $\Delta_\lambda(\beta) =: \delta \in UI_\lambda(n)$.

\noindent (i)  Here $\mathcal{S}_\lambda(\beta) = [Q_\lambda(\beta)]$ and so $\mathcal{S}_\lambda(\beta) = \mathcal{S}_\lambda(\beta^\prime)$ if and only if $Q_\lambda(\beta) = Q_\lambda(\beta^\prime)$.

\noindent (ii)  Here $M_\lambda(\delta) = Q_\lambda(\beta)$ and so $\mathcal{S}_\lambda(\beta) = [M_\lambda(\delta)]$.  \end{prop}

\begin{proof}Only the first claim in (ii) is not already evident:  Recall that $\mathcal{S}_\lambda(\beta) = \bigcup \mathcal{Z}_\lambda(\alpha)$, union over $\alpha \in \{ \beta \}_\lambda \subseteq UI_\lambda(n)$.  By Proposition \ref{prop623.2} and Lemma \ref{lemma608.2}(i) we have $\{ \beta \}_\lambda = [\delta]$.  So $\mathcal{S}_\lambda(\beta) = \bigcup \mathcal{Z}_\lambda(\alpha) = \mathcal{S}_\lambda(\delta)$, union over $\alpha \leq \delta$ in $UI_\lambda(n)$.  Here $\delta \in UI_\lambda(n)$ implies $\mathcal{Z}_\lambda(\delta) \neq \emptyset$.  Let $U \in \mathcal{Z}_\lambda(\delta)$ and $T \in \mathcal{Z}_\lambda(\alpha)$ for some $\alpha \in UI_\lambda(n)$.  Here $\alpha < \delta$ would imply $T \ngtr U$.  Hence $Q_\lambda(\beta) \in \mathcal{Z}_\lambda(\delta)$.  So both $Q_\lambda(\beta)$ and $M_\lambda(\delta)$ are the maximum tableau of $\mathcal{Z}_\lambda(\delta)$.  \end{proof}

\section{Coincidences of row or flag bound sets with \\ Demazure tableau sets}

When can one set of tableaux arise both as a row bound set $\mathcal{S}_\lambda(\beta)$ for some upper $\lambda$-tuple $\beta$ and as a Demazure set $\mathcal{D}_\lambda(\pi)$ for some $\lambda$-permutation $\pi$?  Since we will seek coincidences between flag Schur polynomials $s_\lambda(\varphi; x)$ and Demazure polynomials $d_\lambda(\pi; x)$, we should also pose this question for the flag bound set $\mathcal{S}_\lambda(\varphi)$ for some upper flag $\varphi$.  Our deepest result, Theorem \ref{theorem520}, gave a necessary condition for a Demazure tableau set to be convex.  Initially we refer to it here for guiding motivation.  Then we use it to prove the hardest part, Part (iii) for necessity, of the theorem below.  Out next deepest result, Theorem \ref{theorem420}, implied a sufficient condition (Corollary \ref{cor420}) for a Demazure tableau set to be convex.  We use it here to prove Part (ii), for sufficiency, of the theorem below.

For motivation, note that any set $\mathcal{S}_\lambda(\beta)$ is convex in $\mathbb{Z}^{|\lambda|}$: Proposition \ref{prop623.8}(i) says that it is the principal ideal $[Q_\lambda(\beta)]$ of $\mathcal{T}_\lambda$, where $Q_\lambda(\beta)$ is the $\lambda$-row bound max tableau for $\beta$.  And Theorem \ref{theorem520} says that $\mathcal{D}_\lambda(\pi)$ is convex only if the $\lambda$-permutation $\pi$ is $\lambda$-312-avoiding.  So, to begin the proof of Part (ii) below, fix $\pi \in S_n^{\lambda\text{-}312}$.  Find the key $Y_\lambda(\pi)$ of $\pi$.  Theorem \ref{theorem420} says that $\mathcal{D}_\lambda(\pi) = [Y_\lambda(\pi)]$.  (Since $[Y_\lambda(\pi)]$ is convex, we now know that the sets $\mathcal{D}_\lambda(\pi)$ for such $\pi$ are exactly the convex candidates to arise in the form $\mathcal{S}_\lambda(\beta)$.)  Form the rank $\lambda$-tuple $\Psi_\lambda(\pi) =: \gamma$ of $\pi$.  By Proposition \ref{prop320.2}(ii) we know that $\gamma$ is a gapless $\lambda$-tuple:  $\gamma \in UG_\lambda(n)$.  Theorem \ref{theorem340}(iii) says that $Y_\lambda(\pi) = M_\lambda(\gamma)$, the $\lambda$-row end max tableau for $\gamma$.  Proposition \ref{prop623.8}(ii) gives $M_\lambda(\gamma) = Q_\lambda(\gamma)$, since $UG_\lambda(n) \subseteq UI_\lambda(n)$ by definition and $\Delta_\lambda(\gamma) = \gamma$ by Fact \ref{fact604.2}.  So $\mathcal{D}_\lambda(\pi) = [Q_\lambda(\gamma)]$.  Hence by Proposition \ref{prop623.8}(i) it arises as $\mathcal{S}_\lambda(\gamma) = [Q_\lambda(\gamma)]$.

Parts (i) and (ii) of the following theorem give sufficient conditions for a coincidence from two perspectives, Part (iii) gives necessary conditions for a coincidence, and Part (iv) presents a neutral precise indexing.  But the theorem statement begins with a less technical summary:

\begin{thm}\label{theorem721}Let $\lambda$ be a partition.  A row bound set $\mathcal{S}_\lambda(\beta)$ of tableaux for an upper $\lambda$-tuple $\beta$ arises as a Demazure set if and only if the $\lambda$-core $\Delta_\lambda(\beta)$ of $\beta$ is a gapless $\lambda$-tuple.  Therefore every flag bound set of tableaux arises as a Demazure set, and a row bound set arises as a Demazure set if and only if it arises as a flag bound set.  A Demazure set $\mathcal{D}_\lambda(\pi)$ of tableaux for a $\lambda$-permutation $\pi$ arises as a row bound set if and only if $\pi$ is $\lambda$-312-avoiding.  Specifically:

\noindent (i)  Let $\beta \in U_\lambda(n)$.  If $\beta \in UGC_\lambda(n)$, set $\pi := \Pi_\lambda[\Delta_\lambda(\beta)]$. Then $\mathcal{S}_\lambda(\beta) = \mathcal{D}_\lambda(\pi)$, and $\pi$ is the unique $\lambda$-permutation for which this is true. Here $\pi \in S_n^{\lambda\text{-}312}$.

\noindent (ii) Let $\pi \in S_n^\lambda$.  If $\pi \in S_n^{\lambda\text{-}312}$, set $\gamma := \Psi_\lambda(\pi)$. Then $\mathcal{D}_\lambda(\pi) = \mathcal{S}_\lambda(\gamma)$, and $\mathcal{D}_\lambda(\pi) = \mathcal{S}_\lambda(\beta)$ for some $\beta \in U_\lambda(n)$ implies $\Delta_\lambda(\beta) = \gamma$. Here $\gamma \in UG_\lambda(n)$ and so $\beta \in UGC_\lambda(n)$.

\noindent (iii)    Suppose some $\beta \in U_\lambda(n)$ and some $\pi \in S_n^\lambda$ exist such that $\mathcal{S}_\lambda(\beta) = \mathcal{D}_\lambda(\pi)$. Then one has $Q_\lambda(\beta) = Y_\lambda(\pi)$ and $\Delta_\lambda(\beta) = \Psi_\lambda(\pi)$.  Here $\beta \in UGC_\lambda(n)$ and $\pi \in S_n^{\lambda\text{-}312}$.

\noindent (iv) The collection of the sets $\mathcal{S}_\lambda(\varphi)$ for $\varphi \in UF_\lambda(n)$ is the same as the collection of sets $\mathcal{D}_\lambda(\pi)$ for $\pi \in S_n^{\lambda\text{-}312}$. These collections can be simultaneously precisely indexed by $\gamma \in UG_\lambda(n)$ as follows: Given such a $\gamma$, produce $\Phi_\lambda(\gamma) =: \varphi \in UFlr_\lambda(n)$ and $\Pi_\lambda(\gamma) =: \pi \in S_n^{\lambda\text{-}312}$.  \end{thm}

\begin{proof}First confirm (i) - (iv):  The first and last two claims in (ii) were deduced above.  For (i), use Proposition \ref{prop320.2}(ii) to see $\pi \in S_n^{\lambda\text{-}312}$ and to express $\Delta_\lambda(\beta)$ as $\Psi_\lambda(\pi)$.  The first claim in (ii) then tells us that $\mathcal{S}_\lambda[\Delta_\lambda(\beta)] = \mathcal{D}_\lambda(\pi)$.  But Proposition \ref{prop623.4}(i) gives $\mathcal{S}_\lambda[\Delta_\lambda(\beta)] = \mathcal{S}_\lambda(\beta)$.  We return to the second claims in (i) and (ii) after we confirm (iii).  So suppose we have $\mathcal{S}_\lambda(\beta) = \mathcal{D}_\lambda(\pi)$.  Since $\mathcal{S}_\lambda(\beta)$ is a principal ideal in $\mathcal{T}_\lambda$, Theorem \ref{theorem520} tells us that we must have $\pi \in S_n^{\lambda\text{-}312}$.  The unique maximal elements of $\mathcal{S}_\lambda(\beta)$ and of $\mathcal{D}_\lambda(\pi)$ (see Proposition \ref{prop623.8}(i) and Fact \ref{fact420}(v)) must coincide: $Q_\lambda(\beta) = Y_\lambda(\pi)$.  Via consideration of $M_\lambda[\Delta_\lambda(\beta)]$, Proposition \ref{prop623.8}(ii) implies $\Omega_\lambda[Q_\lambda(\beta)] = \Delta_\lambda(\beta)$.  Section 8 noted that $\Omega_\lambda[Y_\lambda(\pi)] = \Psi_\lambda(\pi)$.  Hence $\Delta_\lambda(\beta) = \Psi_\lambda(\pi) \in UG_\lambda(n)$.  The uniqueness in (i) is obtained by applying the inverse $\Pi_\lambda$ of $\Psi_\lambda$ to the requirements in (iii) that $\pi \in S_n^{\lambda\text{-}312}$ and that $\Psi_\lambda(\pi) = \Delta_\lambda(\beta)$.  The uniqueness-up-to-$\Delta_\lambda$-equivalence in (ii) restates the second claim of (iii).  For (iv): Proposition \ref{prop623.4}(iii) says that the collection of the sets $\mathcal{S}_\lambda(\varphi)$ is precisely indexed by the $\lambda$-floor flags.  By restricting Fact \ref{fact420}(vi), these sets $\mathcal{D}_\lambda(\pi)$ are already precisely indexed by their $\lambda$-312-avoiding permutations.  By Proposition \ref{prop608.10}(i) and Proposition \ref{prop320.2}(ii), apply the bijections $\Delta_\lambda$ and $\Psi_\lambda$ to re-index these collections with gapless $\lambda$-tuples.  Use (ii) and Proposition \ref{prop623.4}(ii) to see that the same set $\mathcal{S}_\lambda(\varphi) = \mathcal{D}_\lambda(\pi)$ will arise from a given gapless $\lambda$-tuple $\gamma$ when these re-indexings are undone via $\varphi := \Phi_\lambda(\gamma)$ and $\pi := \Pi_\lambda(\gamma)$.  Three of the four initial summary statements of this theorem should now be apparent.  The third statement follows from Proposition \ref{prop623.1}. \end{proof}

\section{Flagged Schur functions and key polynomials}

We use Theorem \ref{theorem721} to improve upon the results in \cite{RS} and \cite{PS} concerning coincidences between flag Schur polynomials and Demazure polynomials.

Let $x_1,\ldots, x_n$ be indeterminants.  Let $T \in \mathcal{T}_\lambda$.  The \textit{monomial} $x^{\Theta(T)}$ of $T$ is $x_1^{\theta_1}\ldots x_n^{\theta_n}$, where $\theta$ is the content $\Theta(T)$.

Let $\beta$ be an upper $\lambda$-tuple:  $\beta \in U_\lambda(n)$.  We introduce the \textit{row bound sum} $s_\lambda(\beta; x) := \sum x^{\Theta(T)}$, sum over $T \in \mathcal{S}_\lambda(\beta)$.  In particular, to relate to the literature \cite{RS} \cite{PS}, at times we restrict our attention to flag row bounds.  Here for $\varphi \in UF_\lambda(n)$ we define the \textit{flag Schur polynomial} to be $s_\lambda(\varphi; x)$.  (Often `upper' is not required at the outset; if $\varphi$ is not upper then the empty sum would yield $0$ for the polynomial.  Following Stanley we write `flag' instead of `flagged' \cite{St2}, and following Postnikov and Stanley we write `polynomial' for `function' \cite{PS}.)  More generally, for $\eta \in UGC_\lambda(n)$ we define the \emph{gapless core Schur polynomial} to be $s_\lambda(\eta;x)$.

Let $\pi$ be a $\lambda$-permutation:  $\pi \in S_n^\lambda$.  Here we define the \textit{Demazure polynomial} $d_\lambda(\pi; x)$ to be $\sum x^{\Theta(T)}$, sum over $T \in \mathcal{D}_\lambda(\pi)$.  For Lie theorists, we make two remarks:  Using the Appendix and Sections 2 and 3 of \cite{PW1}, via the right key scanning method and the divided difference recursion these polynomials can be identified as the Demazure characters for $GL(n)$ and as the specializations of the key polynomials $\kappa_\alpha$ of \cite{RS} to a finite number of variables.  In the axis basis, the highest and lowest weights for the corresponding Demazure module are $\lambda$ and $\Theta[Y_\lambda(\pi)]$.  (Postnikov and Stanley chose `Demazure character' over `key polynomial' \cite{PS}.  By using `Demazure polynomial' for the $GL(n)$ case, which should be recognizable to Lie theorists, we leave `Demazure character' available for general Lie type.)

We say that two polynomials that are defined as sums of monomials over sets of tableaux are \textit{identical as generating functions} if the two tableau sets coincide.  So we write $s_\lambda(\beta; x) \equiv s_{\lambda'}(\beta'; x)$ if and only if $\lambda = \lambda'$ and then $\beta \sim \beta'$.  It is conceivable that the polynomial equality $s_\lambda(\beta; x) = s_{\lambda'}(\beta'; x)$ could ``accidentally'' hold between two non-identical row bound sums, that is when $\lambda \neq \lambda^\prime$ and/or $\beta \nsim \beta'$.

It is likely that Part (ii) of the following preliminary result can be deduced from the linear independence aspect of the sophisticated Corollary 7 of \cite{RS}:  One would need to show that specializing $x_{n+1} = x_{n+2} = ... = 0$ there does not create problematic linear dependences.

\begin{prop}\label{prop737}Let $\lambda, \lambda' \in \Lambda_n^+$.

\noindent (i) Let $\beta \in U_\lambda(n)$ and $\beta' \in U_{\lambda'}(n)$.  If $s_\lambda(\beta; x) = s_{\lambda'}(\beta'; x)$, then $\lambda = \lambda'$.

\noindent (ii) Let $\pi \in S_n^\lambda$ and let $\pi' \in S_n^{\lambda'}$.  If $d_\lambda(\pi; x) = d_{\lambda'}(\pi'; x)$, then $\lambda = \lambda'$ and $\pi = \pi'$. Hence $d_\lambda(\pi; x) \equiv d_{\lambda'}(\pi'; x)$.  \end{prop}

\begin{proof}  Let $T^0_\lambda$ denote the unique minimal element of $\mathcal{T}_\lambda$.  Note that $\Theta(T_\lambda^0) = \lambda$.  Clearly $T^0_\lambda \in \mathcal{S}_\lambda(\beta)$ and $T^0_\lambda \leq Y_\lambda(\pi)$.  Note that if $T, T' \in \mathcal{T}_\lambda$ are such that $T < T'$, then when $P(n)$ is ordered lexicographically from the left we have $\Theta(T) > \Theta(T')$.  So when $T^0_\lambda$ is in a subset of $T_\lambda$, it is the unique tableau in that subset that attains the lexicographic maximum of the contents in $P(n)$ for that subset.  Since $T^0_\lambda$ is a key, we have $S(T^0_\lambda) = T^0_\lambda$.  So $S(T^0_\lambda) \leq Y_\lambda(\pi)$, and we have $T^0_\lambda \in \mathcal{D}_\lambda(\pi)$.  We can now see that $s_\lambda(\beta; x) = s_{\lambda'}(\beta'; x)$ and $d_\lambda(\pi; x) = d_{\lambda'}(\pi'; x)$ each imply that $\lambda = \lambda'$.  By Fact \ref{fact420}(v), we know that $Y_\lambda(\pi)$ is the unique maximal element of $\mathcal{D}_\lambda(\pi)$.  So $Y_\lambda(\pi)$ is the unique tableau in $\mathcal{D}_\lambda(\pi)$ that attains the lexicographic minimum of the contents for $\mathcal{D}_\lambda(\pi)$.  By Fact \ref{fact420}(i), since $Y_\lambda(\pi)$ is a key it is the unique tableau in $\mathcal{T}_\lambda$ with its content.  So $d_\lambda(\pi; x) = d_\lambda(\pi'; x)$ implies $Y_\lambda(\pi) \in \mathcal{D}_\lambda(\pi^\prime)$ and $Y_\lambda(\pi^\prime) \in \mathcal{D}_\lambda(\pi)$.  Hence $Y_\lambda(\pi) = Y_\lambda(\pi')$, which implies $\pi = \pi^\prime$.  \end{proof}

We now compare row bound sums to Demazure polynomials.  The first two parts of the following ``sufficient'' theorem quickly restate most of Parts (i) and (ii) of Theorem \ref{theorem721} in the current context, and the third similarly recasts Part (iv).

\begin{thm}\label{theorem737.1}Let $\lambda$ be a partition.

\noindent (i)  If $\eta \in UGC_\lambda(n)$, then $\Pi_\lambda[\Delta_\lambda(\eta)] =: \pi \in S_n^{\lambda\text{-}312}$ and $s_\lambda(\eta;x) \equiv d_\lambda(\pi;x)$.

\noindent (ii)  If $\pi \in S_n^{\lambda\text{-}312}$, then $\Psi_\lambda(\pi) =: \gamma \in UG_\lambda(n)$ and $d_\lambda(\pi;x)  \equiv s_\lambda(\gamma;x)$.

\noindent (iii) Every flag Schur polynomial is identical to a uniquely determined Demazure polynomial and every $\lambda$-312-avoiding Demazure polynomial is identical to a uniquely determined flag Schur polynomial.  \end{thm}

Next we obtain \emph{necessary} conditions for having equality between a row bound sum and a Demazure polynomial: we see that working merely with polynomials does not lead to any new coincidences.

\begin{thm}\label{theorem737.2}Let $\lambda$ and $\lambda'$ be partitions.  Let $\beta$ be an upper $\lambda$-tuple and let $\pi$ be a $\lambda'$-permutation.  Suppose $s_\lambda(\beta; x) = d_{\lambda'}(\pi; x)$.  Then $Q_\lambda(\beta) = Y_{\lambda'}(\pi)$.  Hence $\lambda = \lambda'$ and $\Delta_\lambda(\beta) = \Psi_\lambda(\pi)$.  Here $\pi$ is $\lambda$-312-avoiding and $\Delta_\lambda(\beta)$ is a gapless $\lambda$-tuple (and so $\beta \in UGC_\lambda(n)$).  Hence the only row bound sums that can arise as Demazure polynomials are the flag Schur polynomials.  We have $s_\lambda(\beta; x) \equiv d_{\lambda'}(\pi; x)$.  The row bound sum $s_\lambda(\beta; x)$ is identical to the flag Schur polynomial $s_\lambda(\Phi_\lambda[\Delta_\lambda(\beta)];x)$.  \end{thm}

\begin{proof}\noindent Reasoning as in the first part of the proof of Proposition \ref{prop737} implies $\lambda = \lambda'$.  Since $\mathcal{S}_\lambda(\beta) = [Q_\lambda(\beta)]$ by Proposition \ref{prop623.8}(i), the tableau $Q_\lambda(\beta)$ is the unique tableau in $\mathcal{S}_\lambda(\beta)$ that attains the lexicographic minimum of the contents for $\mathcal{S}_\lambda(\beta)$.  Since the analogous remark was made in the proof of Proposition \ref{prop737} for $Y_\lambda(\pi) \in \mathcal{D}_\lambda(\pi)$, we must have $\Theta[Q_\lambda(\beta)] = \Theta[Y_\lambda(\pi)]$.  But $Y_\lambda(\pi)$ is the unique tableau in $\mathcal{T}_\lambda$ with its content.  So we must have $Q_\lambda(\beta) = Y_\lambda(\pi)$.  As for Theorem \ref{theorem721}, this implies $\Delta_\lambda(\beta) = \Psi_\lambda(\pi)$.  Since $\mathcal{D}_\lambda(\pi) \subseteq [Y_\lambda(\pi)]$, we have $\mathcal{D}_\lambda(\pi) \subseteq [Q_\lambda(\beta)] = \mathcal{S}_\lambda(\beta)$.  Suppose that $\pi$ is $\lambda$-312-containing.  Then Corollary \ref{cor520} says $\mathcal{D}_\lambda(\pi) \neq [Q_\lambda(\beta)]$.  So $\mathcal{D}_\lambda(\pi) \subset \mathcal{S}_\lambda(\beta)$.  This implies that $d_\lambda(\pi; x) \neq s_\lambda(\beta; x)$, a contradiction.  So $\pi$ must be $\lambda$-312-avoiding.  Therefore $\Psi_\lambda(\pi) =: \gamma \in UG_\lambda(n)$.  Use Proposition \ref{prop623.1} for the ``only row bound sums that can arise'' statement.  Theorem \ref{theorem721}(ii) now says that $\mathcal{D}_\lambda(\pi) = \mathcal{S}_\lambda(\gamma)$.  And $\gamma = \Delta_\lambda(\beta)$ gives $\mathcal{S}_\lambda(\gamma) = \mathcal{S}_\lambda(\beta)$.  Since $\gamma \in UG_\lambda(n)$, we can form the $\lambda$-floor flag $\Phi_\lambda[\Delta_\lambda(\beta)] \sim \beta$. \end{proof}

By using the relating of row bound sums to Demazure polynomials in this theorem, we can extend what was said in Proposition \ref{prop737}(i) concerning accidental equalities between row bound sums.  There we learned that $s_\lambda(\beta;x) = s_\lambda(\beta';x)$ forced $\lambda = \lambda'$.  So here we need consider only one $\lambda$:

\begin{cor}\label{newcor737}Let $\lambda$ be a partition.

\noindent (i)  Let $\beta \in U_\lambda(n)$ and $\eta \in UGC_\lambda(n)$.  If $s_\lambda(\beta;x) = s_\lambda(\eta;x)$ then $\beta \sim \eta$.  Hence $\beta \in UGC_\lambda(n)$ and $s_\lambda(\beta;x) \equiv s_\lambda(\eta;x)$.

\noindent (ii)  The partitionings of $UGC_\lambda(n)$ into the equivalence class intervals in Proposition \ref{prop608.4}(ii) give a complete description of the indexing ambiguity and of non-equality for gapless core Schur polynomials.

\noindent (iii) More specifically, the analogous statement for $UF_\lambda(n)$ and flag Schur polynomials follows from Proposition \ref{prop608.4}(iii).  \end{cor}

\noindent Parts (ii) and (iii) could have been derived from Theorem \ref{theorem737.1}.

\begin{proof}Create $\Pi_{\lambda}[\Delta_{\lambda}(\eta)] =: \pi \in S_n^{\lambda\text{-}312}$ from $\Delta_{\lambda}(\eta) \in UG_\lambda(n)$.  Apply Theorem \ref{theorem737.1}(ii) to obtain $d_{\lambda}(\pi; x) \equiv s_\lambda(\Delta_\lambda(\eta);x) \equiv s_\lambda(\eta;x)$.  Then apply Theorem \ref{theorem737.2} to $s_\lambda(\beta; x) = d_{\lambda}(\pi; x)$ to obtain $s_\lambda(\beta; x) \equiv d_\lambda(\pi; x)$.  So $s_\lambda(\beta; x) \equiv s_{\lambda}(\eta; x)$, which implies $\beta \sim \eta$ and $\beta \in UGC_\lambda(n)$.  \end{proof}

We do not know if it is possible to rule out accidental coincidences between all pairs of row bound sums:

\begin{prob}\label{prob14.5}Find $n \geq 1$, $\lambda \in \Lambda_n^+$, and $\beta, \beta^\prime \in U_\lambda(n) \backslash UGC_\lambda(n)$ such that $s_\lambda(\beta;x) = s_\lambda(\beta^\prime;x)$ but $\Delta_\lambda(\beta) \neq \Delta_\lambda(\beta^\prime)$. \end{prob}

Improving upon Equation 13.1 and Corollary 14.6 of \cite{PS}, in \cite{PW2} we will give a ``maximum efficiency'' determinant expression for the Demazure polynomials $d_\lambda(\pi;x)$ with $\pi \in S_n^{\lambda\text{-}312}$.

\section{Connecting to earlier work}

\noindent Working with an infinite number of variables $x_1, x_2, ...$, Reiner and Shimozono studied \cite{RS} coincidences between ``skew'' flag Schur polynomials and Demazure polynomials in their Theorems 23 and 25.  To indicate how those statements are related to our results, we consider only their ``non-skew'' flag Schur polynomials and specialize those results to having just $n$ variables $x_1, ... , x_n$.  Then their key polynomials $\kappa_\alpha(x)$ are indexed by ``(weak) compositions $\alpha$ (into $n$ parts)''.  The bijection from our pairs $(\lambda, \pi)$ with $\lambda \in \Lambda_n^+$ and $\pi \in S_n^\lambda$ to their compositions $\alpha$ that was noted in Section 3 of \cite{PW1} is indicated in the sixth paragraph of the Appendix to that paper:  Let $\pi \in S_n^\lambda$.  After creating $\alpha$ via $\alpha_{\pi_i} := \lambda_i$ for $i \in [n]$, here we write $\pi.\lambda := \alpha$.  Under this bijection the Demazure polynomial $d_\lambda(\pi;x)$ of \cite{PW1} and their key polynomial $\kappa_\alpha(x)$ are defined by the same recursion.  Reiner and Shimozono characterized the coincidences between the $s_\lambda(\phi;x)$ for $\phi \in UF_\lambda(n)$ and the $d_{\lambda'}(\pi;x)$ for $\pi \in S_n^{\lambda'}$ from the perspectives of both the flag Schur polynomials and the Demazure polynomials.  To relate the index $\phi$ to the index $\pi \in S_n^{\lambda'}$, their theorems refer to the tableau we denote $Q_\lambda(\phi)$.  Part (i) of the following fact extends the sixth paragraph of the Appendix of \cite{PW1}.  Part (ii) can be confirmed with Proposition \ref{prop623.8}(ii), Proposition \ref{prop604.4}(ii) and Lemma \ref{lemma340.1}.

\begin{fact}\label{fact824B}Let $\pi \in S_n^\lambda$.  Let $\phi \in UF_\lambda(n)$.

\noindent (i)  The content $\Theta[Y_\lambda(\pi)]$ of the key of $\pi$ is the composition that has the unique decomposition $\pi.\lambda$.

\noindent (ii)  The tableau $Q_\lambda(\phi)$ is a $\lambda$-key $Y_\lambda(\sigma)$ for a uniquely determined $\sigma \in S_n^\lambda$. \end{fact}

From the perspective of flag Schur polynomials, in our language their Theorem 23 first said that every $s_\lambda(\phi;x)$ arises as a $d_{\lambda'}(\pi;x)$ for at least one pair $(\lambda', \pi)$ with $\lambda' \in \Lambda_n^+$ and $\pi \in S_n^{\lambda'}$.  Second, that $d_{\lambda'}(\pi;x)$ must be the Demazure polynomial for which $\pi.\lambda' = \Theta[Q_\lambda(\phi)]$.  Their first statement appears here as a weaker form of the first part of Theorem \ref{theorem737.1}(iii).  The fact above can be used to show that their second (uniqueness) claim is equivalent to the first (and central) ``necessary'' claim $Q_\lambda(\phi) = Y_{\lambda'}(\pi)$ in our Theorem \ref{theorem737.2} that is produced by taking $\beta := \phi$.

From the other perspective, their Theorem 25 put forward a characterization for a Demazure polynomial $d_\lambda(\pi;x)$ that arises as a flag Schur polynomial $s_{\lambda'}(\phi;x)$ for some $\phi \in UF_{\lambda'}(n)$.  This characterization is stated in terms of a flag $\phi(\alpha)$ that is specified by a recipe to be applied to a composition $\alpha$;  this is given before the statement of Theorem 25.  Let $(\lambda, \pi)$ be the pair corresponding to $\alpha$.  It appears that the recipe for $\phi(\alpha)$ should have ended with `having size $\lambda_i$' instead of `having size $\alpha_i$'; we take this fix for granted for the remainder of the discussion.  But no recipe of this form can be completely useful for general partitions $\lambda$ since any $\lambda$-tuple $\phi(\alpha)$ produced will be constant on the carrels of $[n]$ determined by $\lambda$.  So first we consider only strict $\lambda$.  Here it can be seen that their $\phi(\alpha)$ becomes our $\Psi(\pi) =: \psi$.  Thus their $T_{\lambda(\alpha),\phi(\alpha)}$ is our $Q(\psi)$, and so the condition $Key(\alpha) = T_{\lambda(\alpha),\phi(\alpha)}$ translates to $Y(\pi) = Q(\psi)$.  Following the statement of Theorem \ref{theorem340}, we noted that the converse of the first part of its Part (iii) held when $\lambda$ is strict.  Using Proposition \ref{prop623.8}(ii), Theorem \ref{theorem340}(iii), and that fact we see that this $Y(\pi) = Q(\psi)$ condition is equivalent to requiring $\pi \in S_n^{312}$.  So when $\lambda$ is strict the two directions of Theorem 25 appear in this paper as parts of Theorem \ref{theorem737.1}(ii) and Theorem \ref{theorem737.2}.

Now consider Theorem 25 for general $\lambda$.  Its hypothesis $\kappa(\alpha) = S_{\lambda / \mu}(X_\phi)$ translates to $d_\lambda(\pi;x) = s_{\lambda'}(\phi;x)$.  In the necessary direction a counterexample to their condition $Key(\alpha) = T_{\lambda(\alpha),\phi(\alpha)}$, which translates to $Y_\lambda(\pi) = Q_\lambda(\phi(\alpha))$, is given by $\alpha = (1,2,0,1)$.  Turning to the sufficient direction:  From looking at $\Omega_\lambda[Y_\lambda(\pi)]$ it can be seen that the $nn$-tuple $\phi(\alpha) =: \phi$ is in $UF_\lambda(n) \subseteq UGC_\lambda(n)$ as well as being constant on the carrels of $\lambda$.  Suppose that their condition $Y_\lambda(\pi) = Q_\lambda(\phi)$ is satisfied, and set $\Delta_\lambda(\phi) =: \gamma \in UG_\lambda(n)$.  Then $Y_\lambda(\pi) = Q_\lambda(\gamma)$, and Theorem \ref{theorem340} gives $\pi \in S_n^{\lambda\text{-}312}$.  Then Theorem \ref{theorem737.1}(ii) implies that $d_\lambda(\pi;x) = s_\lambda(\phi;x)$, which confirms this part of Theorem 25.  However, the set of cases $(\lambda,\pi)$ that are produced by this sufficient condition is smaller than that produced by the $\lambda$-312-avoiding sufficient condition:  It can be seen that each index $\gamma$ produced above has only a single critical entry in each carrel of $\lambda$, while the general indexes $\gamma'$ that can arise for such coincidences range over all of the larger set $UG_\lambda(n)$.

Discussing Theorem 25 for general $\lambda$ further, the necessary condition can be completely ``loosened up'' by replacing `$T_{\lambda(\alpha),\phi(\alpha)}$' with `$T_{\lambda',\phi}$', which translates to $Q_{\lambda'}(\phi)$.  This repaired version now gives the necessary condition $Y_\lambda(\pi) = Q_{\lambda'}(\phi)$, which is the central claim of Theorem \ref{theorem737.2}.  To include more cases, one might attempt to extend our view of the  sufficient part of Theorem 25 for strict $\lambda$ to general $\lambda$ as follows:  Let $\pi \in S_n^\lambda$ and set $\psi := \Psi_\lambda(\pi)$.  Since $\psi$ is not constant on the carrels of $\lambda$, it is hoped that all cases will now be included.  Does having $Y_\lambda(\pi) = Q_\lambda(\psi)$ imply that the Demazure polynomial $d_\lambda(\pi;x)$ is equal to the row bound sum $s_\lambda(\psi;x)$?  If this were true, then Theorem \ref{theorem737.2} tells us that $\pi \in S_n^{\lambda\text{-}312}$ and $\psi \in UGC_\lambda(n)$.  So to provide a counterexample, we do not need to compute polynomials.  It will suffice to specify an example of $Y_\lambda(\pi) = Q_\lambda(\psi)$ with either $\pi \in S_n^\lambda \backslash S_n^{\lambda\text{-}312}$ or with $\psi \in U_\lambda(n) \backslash UGC_\lambda(n)$.  We do the former, since at the same time it will also provide a counterexample to the converse of the first part of Theorem \ref{theorem340}(iii).  Choose $n = 4, \lambda = (2,1,1,0)$ and $\pi = (4;1,2;3)$.  Then $Y_\lambda(\pi) = Q_\lambda(\psi)$ with $\pi \notin S_n^{\lambda\text{-}312}$.  So this proposed condition is ``too loose''.  (The proof of the sufficient direction of Theorem 25 cites the converse of the second part of Theorem 23, not the implication itself.)

We prepare to discuss a related result \cite{PS} of Postnikov and Stanley.  Let $\pi \in S_n^\lambda$.  Our definitions of the $\lambda$-chain $B$ and the $\lambda$-key $Y_\lambda(\pi)$ can be extended to all of $S_n$ so that $Y_\lambda(\sigma') = Y_\lambda(\pi)$ exactly for the $\sigma' \in S_n$ such that $\bar{\sigma'} = \pi$.  Then our definition of Demazure polynomial can be extended from $S_n^\lambda$ to $S_n$ such that $d_\lambda(\sigma';x) = d_\lambda(\pi;x)$ for exactly the same $\sigma'$.  Their paper used this ``looser'' indexing for the Demazure polynomials.

In their Theorem 14.1, Postnikov and Stanley stated a sufficient condition for a coincidence to occur from the perspective of Demazure polynomials:  If $\pi \in S_n$ is 312-avoiding, then $d_\lambda(\pi;x) = s_\lambda(\phi;x)$ for a certain $\phi \in UF_\lambda(n)$.  After noting that this theorem followed from Theorem 20 of \cite{RS}, they provided their own proof of it.  Their bijective recipe for forming $\phi$ from $\pi$ was complicated.  Their inverse for this bijection is our inverse map $\Pi$ of Proposition \ref{prop320.1}(ii), which takes upper flags to 312-avoiding permutations.  Since the inverse of the inverse of a bijection must be the bijection, from that proposition it follows that their recipe for $\phi$ must be the restriction of our $\Psi$ to $S_n^{312}$.  The following result uses the machinery provided by our maps of $n$-tuples in Propositions \ref{prop824.2} and \ref{prop824.4} to prove that their theorem is equivalent to a weaker version of one of ours:

\begin{thm}\label{theorem824.5}Theorem 14.1 of \cite{PS} is equivalent to our Theorem \ref{theorem737.1}(ii), once `$\equiv$' in the latter result has been replaced by `$=$'.  \end{thm}

\begin{proof}Let $\sigma^\prime$ be a 312-avoiding permutation.  Set $\pi := \bar{\sigma^\prime}$ and $\gamma := \Psi_\lambda(\pi)$.  Then $\pi$ is $\lambda$-312-avoiding and $d_\lambda(\sigma^\prime; x) = d_\lambda(\pi; x) \equiv s_\lambda(\gamma; x)$ by Theorem \ref{theorem737.1}(ii).  Recall the remark above that noted that our map $\Psi$ is the map $b(\cdot)$ of \cite{PS}.  Set $\varphi := \Psi(\sigma^\prime)$; this is the upper flag used in Theorem 14.1 of \cite{PS}.  Then $\varphi \sim \Psi_\lambda(\pi)$ by Proposition \ref{prop824.4}.  So $s_\lambda(\gamma; x) = s_\lambda(\varphi; x)$ gives us the result $d_\lambda(\sigma^\prime;x) = s_\lambda(\varphi;x)$ of \cite{PS}.  Conversely, let $\pi$ be a $\lambda$-312-avoiding $\lambda$-permutation.  Let $\sigma$ be the minimum length lift of $\pi$.  Again $\varphi := \Psi(\sigma)$ is the upper flag used in Theorem 14.1 of \cite{PS}.  So that result gives us $d_\lambda(\pi; x) = d_\lambda(\sigma; x) = s_\lambda(\varphi; x)$.  And Proposition \ref{prop824.2} implies $\varphi \sim \Psi_\lambda(\pi) =: \gamma$.  So $s_\lambda(\varphi;x) = s_\lambda(\gamma;x)$ gives us the $d_\lambda(\pi;x) = s_\lambda(\gamma;x)$ consequence of Theorem \ref{theorem737.1}(ii).  \end{proof}

\noindent To convert their index $\sigma' \in S_n^{312}$ for a Demazure polynomial to an index for a flag Schur polynomial, Postnikov and Stanley set $\varphi' := \Psi(\sigma')$.  For one such $\sigma'$, our unique corresponding element of $S_n^{\lambda\text{-}312}$ is $\pi := \bar{\sigma'}$.  Let $\sigma$ be the minimum length $\lambda$-312-avoiding lift of $\pi$, and let $\sigma''$ be any other such lift.  We work with $\Psi_\lambda(\pi) =: \gamma \in UG_\lambda(n)$ and $s_\lambda(\gamma;x)$.  As they apply $\Psi$ to various $\sigma''$, they produce various upper flags $\varphi''$.  By Proposition \ref{prop824.4} we see $\varphi'' \sim \gamma$.  By Proposition \ref{prop824.2} it can be seen that our ``favored'' $\Psi(\sigma)$ is the $\lambda$-floor flag $\Phi_\lambda[\gamma] =: \tau$ that is the unique minimal upper flag such that $s_\lambda(\tau;x) = s_\lambda(\varphi'';x) = d_\lambda(\sigma'';x)$.

\section{Distinctness of polynomials}

Table 16.1 summarizes our results concerning the generating functions for our tableau sets; there $\lambda, \lambda' \in \Lambda_n^+$.  The three bi-implications for the row bound sums in the ``identical'' column restate Proposition \ref{prop623.4}; the two for the Demazure generating functions restate (or specialize from) Fact \ref{fact420}(vi).  The four sufficient-for-polynomial-equality implications in Rows 3,4,6, and 7 in the ``equal'' column follow immediately.  In those rows the necessary implications for equality for the two row bound sums are in Corollary \ref{newcor737}(ii)(iii); those for the two Demazure polynomials restate (or specialize from) Proposition \ref{prop737}(ii).  In Row 5, the bi-implications respectively come from Theorems \ref{theorem721} and \ref{theorem737.2}.  For the non-identicality in Row 1, refer to the definition of `$\equiv$' or note $\Delta_\lambda(\eta) \neq \Delta_\lambda(\beta)$ and use Proposition \ref{prop623.4}(i).  The non-equality follows from Proposition \ref{prop737}(i) and Corollary \ref{newcor737}(i).

The count notations ${n \choose R}$ and $C_n^R$ were defined in Section 3.  For the count in Row 2, recall that the equivalence classes $\langle \beta \rangle_{\approx_\lambda}$ of tableau sets $\mathcal{S}_\lambda(\beta)$ can be indexed by the elements of $UI_\lambda(n)$ according to Proposition \ref{prop623.4}(i).  We know that $| UI_\lambda(n) | = {n \choose R_\lambda} =: {n \choose \lambda}$.  The Demazure tableau sets $\mathcal{D}_\lambda(\pi)$ are indexed by $\pi \in S_n^\lambda$ by Fact \ref{fact420}(vi), and we know $| S_n^\lambda | = {n \choose R_\lambda}$.  The counts of $C_n^\lambda := C_n^{R_\lambda}$ appearing in the table will be justified in the proof of Theorem \ref{theorem18.1}.  It is mysterious that ${n \choose \lambda} - C_n^\lambda$ counts both the number of row bound sums that cannot arise as Demazure polynomials as well as the number of Demazure polynomials that cannot arise as row bound sums. Can this be explained by an underlying phenomenon?

\begin{prob}\label{prob16.1}Let $\lambda$ be a partition.  Set $J := [n-1] \backslash R_\lambda$.  Is there a non-$T$-equivariant deformation of the Schubert varieties in the $GL(n)$ flag manifold $G/P_J$ that bijectively moves the torus characters $d_\lambda(\pi;x)$ to the $s_\lambda(\alpha;x)$ for $\pi \in S_n^\lambda$ and $\alpha \in UI_\lambda(n)$ with exactly $C_n^R$ fixed points, namely $d_\lambda(\sigma;x) = s_\lambda(\gamma;x)$ for $\sigma \in S_n^{\lambda\text{-}312}, \gamma \in UG_\lambda(n)$, and $\Psi_\lambda(\sigma) = \gamma$?  \end{prob}

\noindent Finding such a bijection would rule out accidental equalities among all row bound sums.  To get started, first compute the dimensions $|\mathcal{D}_\lambda(\pi)|$ and $|\mathcal{S}_\lambda(\alpha)|$ for all $\pi \in S_n^\lambda$ and $\alpha \in UI_\lambda(n)$ for some small $\lambda \in \Lambda_n^+$.  Use this data to propose a guiding bijection from $UI_\lambda(n)$ to $S_n^\lambda$ that extends our bijection $\Pi_\lambda : UG_\lambda(n) \rightarrow S_n^{\lambda\text{-}312}$.

\vspace{.25in}

\begin{table}[h!]

{
\begin{center}

\begin{tabular}{ccccc}
 & Identical as & & Equal as & \\
 & \underline{generating functions?} & \underline{Count} & \underline{polynomials?} & \underline{Count} \\ \\
 (1) & $\eta \in UGC_\lambda(n)$ and & & $\eta \in UGC_\lambda(n)$ and & \\
 & $\lambda \neq \lambda'$ or $\beta \notin UGC_\lambda(n)$ & $-$ & $\lambda \neq \lambda'$ or $\beta \notin UGC_\lambda(n)$ & $-$ \\
 & $\Rightarrow s_\lambda(\eta;x) \hspace{1mm} \cancel{\equiv} \hspace{1mm} s_{\lambda'}(\beta;x)$ & & $\Rightarrow s_\lambda(\eta;x) \neq s_{\lambda'}(\beta;x)$ & \\ \\
 (2) & $\beta \in U_\lambda(n), \beta' \in U_{\lambda'}(n)$: & & $\beta \in U_\lambda(n), \beta' \in U_{\lambda'}(n)$: & \\
 & $s_\lambda(\beta;x) \equiv s_{\lambda'}(\beta';x) \Leftrightarrow$ & \Large ${n \choose \lambda}$ & \normalsize $s_\lambda(\beta;x) = s_{\lambda'}(\beta';x) \Leftrightarrow$ ? & $-$ \\
 & $\lambda = \lambda', \beta \sim \beta'$ & & (Problem 14.5) \\ \\
 (3) & $\eta \in UGC_\lambda(n), \eta' \in UGC_{\lambda'}(n)$: & & $\eta \in UGC_\lambda(n), \eta' \in UGC_{\lambda'}(n)$: & \\
 & $s_\lambda(\eta;x) \equiv s_{\lambda'}(\eta';x) \Leftrightarrow$ & $C_n^\lambda$ & $s_\lambda(\eta;x) = s_{\lambda'}(\eta';x) \Leftrightarrow$ & $C_n^\lambda$ \\
 & $\lambda = \lambda', \eta \sim \eta'$ & & $\lambda = \lambda', \eta \sim \eta'$. \\  \\
 (4) & $\varphi \in UFlr_\lambda(n), \varphi' \in UFlr_{\lambda'}(n)$: & & $\varphi \in UFlr_\lambda(n), \varphi' \in UFlr_{\lambda'}(n)$: & \\
 & $s_\lambda(\varphi;x) \equiv s_{\lambda'}(\varphi';x) \Leftrightarrow$ & $C_n^\lambda$ & $s_\lambda(\varphi;x) = s_{\lambda'}(\varphi';x) \Leftrightarrow$ & $C_n^\lambda$ \\
 & $\lambda = \lambda', \varphi \sim \varphi'$ & & $\lambda = \lambda', \varphi \sim \varphi'$ \\  \\
 (5) & $\beta \in U_\lambda(n), \pi \in S_n^{\lambda'}$: & & $\beta \in U_\lambda(n), \pi \in S_n^{\lambda'}$: & \\
 & $s_\lambda(\beta;x) \equiv d_{\lambda'}(\pi;x) \Leftrightarrow$ & $C_n^\lambda$ & $s_\lambda(\beta;x) = d_{\lambda'}(\pi;x) \Leftrightarrow$ & $C_n^\lambda$ \\
 & $\lambda = \lambda', \Delta_\lambda(\beta) = \Psi_\lambda(\pi)$, & & $\lambda = \lambda', \Delta_\lambda(\beta) = \Psi_\lambda(\pi)$, & \\
 & $\beta \in UGC_\lambda(n), \pi \in S_n^{\lambda \text{-}312}$ & & $\beta \in UGC_\lambda(n), \pi \in S_n^{\lambda \text{-}312}$ \\ \\
 (6) & $\sigma \in S_n^{\lambda \text{-}312}, \sigma' \in S_n^{\lambda'\text{-}312}$: & & $\sigma \in S_n^{\lambda \text{-}312}, \sigma' \in S_n^{\lambda'\text{-}312}$: & \\
 & $d_\lambda(\sigma;x) \equiv d_{\lambda'}(\sigma';x) \Leftrightarrow$ & $C_n^\lambda$ & $d_\lambda(\sigma;x) = d_{\lambda'}(\sigma';x) \Leftrightarrow$ & $C_n^\lambda$  \\
 & $\lambda = \lambda', \sigma = \sigma'$ & & $\lambda=\lambda', \sigma=\sigma$ \\ \\
 (7) & $\pi \in S_n^\lambda, \pi' \in S_n^{\lambda'}$: & & $\pi \in S_n^\lambda, \pi' \in S_n^{\lambda'}$: & \\
 & $d_\lambda(\pi;x) \equiv d_{\lambda'}(\pi';x) \Leftrightarrow$ & \Large ${n \choose \lambda}$ & \normalsize $d_\lambda(\pi;x) = d_{\lambda'}(\pi';x) \Leftrightarrow$ & \Large ${n \choose \lambda}$ \\
 & \normalsize $\lambda = \lambda', \pi = \pi'$ & & $\lambda = \lambda', \pi = \pi'$

\end{tabular}\caption*{Table 16.1}
\end{center}
}

\end{table}

\section{Characterization of Gessel-Viennot determinant inputs}

In Theorem 2.7.1 of \cite{St1}, Stanley used the Gessel-Viennot technique to give a determinant expression for a generating function for certain sets of $n$-tuples of non-intersecting lattice paths.  Then in his proof of Theorem 7.16.1 of \cite{St2}, he recast that generating function for some cases by viewing such $n$-tuples of lattice paths as tableaux.  After restricting to non-skew shapes and to a finite number of variables, his generating function becomes our row bound sum $s_\lambda(\beta;x)$ for certain $\beta \in U_\lambda(n)$.  Theorem 2.7.1 required that the pair $(\lambda, \beta)$ satisfies the requirement that Gessel and Viennot call \cite{GV} ``nonpermutable''.  In \cite{PW2} we will present the following combination of his Theorems 2.7.1 and 7.16.1:

\begin{prop}\label{prop17.1}Let $\beta \in U_\lambda(n)$.  If the pair $(\lambda, \beta)$ is nonpermutable, then the row bound sum $s_\lambda(\beta;x)$ is given by the $n \times n$ determinant $| h_{\lambda_j-j+i}(i,\beta_j;x) |$ \end{prop}

In Theorem 2.7.1 Stanley noted that $(\lambda, \phi)$ is nonpermutable for every $\phi \in UF_\lambda(n)$;  this implicitly posed the problem of characterizing all $\beta \in U_\lambda(n)$ for which $(\lambda, \beta)$ is nonpermutable.  The ceiling map $\Xi_\lambda: UG_\lambda(n) \longrightarrow UF_\lambda(n)$ defined in Sections 5 and 4 can be extended to all of $U_\lambda(n)$ by ignoring the requirement in Section 4 that the critical list at hand satisfy the flag condition.  We will refer to this extension as the \emph{platform map} $\Xi_\lambda: U_\lambda(n) \longrightarrow U_\lambda(n)$.  For a given $\lambda \in \Lambda_n^+$, the main result of \cite{PW2} will characterize the $\beta \in U_\lambda(n)$ for which $(\lambda, \beta)$ is nonpermutable:

\begin{thm}Let $\lambda$ be a partition.  Let $\beta \in U_\lambda(n)$.  The pair $(\lambda, \beta)$ is nonpermutable if and only if $\beta \in UGC_\lambda(n)$ and $\beta \leq \Xi_\lambda(\beta)$. \end{thm}

Hence we will again see that restricting consideration from all upper $\lambda$-tuples $\beta \in U_\lambda(n)$ down to at least the gapless core $\lambda$-tuples $\eta \in UGC_\lambda(n)$ enables saying something nice about the row bound sums $s_\lambda(\eta;x)$.  By Corollary \ref{newcor737}, we know that $s_\lambda(\eta;x) = s_\lambda(\beta;x)$ for $\eta \in UGC_\lambda(n)$ and $\beta \in U_\lambda(n)$ if and only if $\beta \sim \eta$.  Then $\beta \in UGC_\lambda(n)$.  So to compute $s_\lambda(\eta;x)$ for a given $\eta \in UGC_\lambda(n)$, the possible inputs for the Gessel-Viennot determinant are the $\eta' \in UGC_\lambda(n)$ such that $\eta' \sim \eta$ and $\eta' \leq \Xi_\lambda(\eta')$.  We will say that a particular such $\lambda$-tuple \emph{attains maximum efficiency} if the corresponding determinant has fewer total monomials among its entries than does the determinant for any other application of Proposition \ref{prop17.1} to a $\beta \in U_\lambda(n)$ that produces $s_\lambda(\eta;x)$.

\begin{cor}\label{cor17.3}Let $\eta \in UGC_\lambda(n)$.  The gapless $\lambda$-tuple $\Delta_\lambda(\eta)$ attains maximum efficiency.\end{cor}

\section{Parabolic Catalan counts}

The section (or paper) cited at the beginning of each item in the following statement points to the definition of the concept:

\begin{thm}\label{theorem18.1}Let $R \subseteq [n-1]$.  Write the elements of $R$ as $q_1 < q_2 < ... < q_r$.  Set $q_0 := 0$ and $q_{r+1} := n$.  Let $\lambda$ be a partition $\lambda_1 \geq \lambda_2 \geq ... \geq \lambda_n \geq 0$ whose shape has the distinct column lengths $q_r, q_{r-1}, ... , q_1$.  Set $p_h := q_h - q_{h-1}$ for $1 \leq h \leq r+1$.  The number $C_n^R =: C_n^\lambda$ of $R$-312-avoiding permutations is equal to the number of:

\noindent (i) \cite{GGHP}:  ordered partitions of $[n]$ into blocks of sizes $p_h$ for $1 \leq h \leq r+1$ that avoid the pattern 312, and $R$-$\sigma$-avoiding permutations for $\sigma \in \{ 123, 132, 213, 231, 321 \}$.

\noindent (ii)  Section 3:  gapless $R$-tuples $\gamma \in UG_R(n)$, $R$-canopy tuples $\kappa$, $R$-floor flags $\tau \in UFlr_R(n)$, $R$-ceiling flags $\xi \in UCeil_R(n)$.

\noindent (iii)  Section 3:  flag $R$-critical lists and (here only) $r$-tuples $(\mu^{(1)}, ... , \mu^{(r)})$ of shapes such that $\mu^{(h)}$ is contained in a $p_h \times (n-q_h)$ rectangle for $1 \leq h \leq r$ and for $1 \leq h \leq r-1$ the length of the first row in $\mu^{(h)}$ does not exceed the length of the $p_{h+1}^{st}$ (last) row of $\mu^{(h+1)}$ plus the number of times that (possibly zero) last row length occurs in $\mu^{(h+1)}$.

\noindent (iv)  Sections 5 and 12:  the four collections of equivalence classes in $UGC_R(n) \supseteq UF_R(n)$ and $UGC_\lambda(n) \supseteq UF_\lambda(n)$ that are defined by the equivalence relations $\sim_R$ and $\approx_\lambda$ respectively.

\noindent (v)  Sections 6 and 9:  $R$-rightmost clump deleting chains and gapless $\lambda$-keys.

\noindent (vi)  Section 10:  sets of Demazure tableaux of shape $\lambda$ that are convex in $\mathbb{Z}^{|\lambda|}$.

\noindent (vii)  Sections 10 and 12:  distinct sets $\mathcal{D}_\lambda(\pi)$ of Demazure tableaux of shape $\lambda$ indexed by $\pi \in S_n^{\lambda\text{-}312}$, and distinct sets $\mathcal{S}_\lambda(\eta)$ (or $\mathcal{S}_\lambda(\phi)$) of gapless core (or flag) bound tableaux of shape $\lambda$ indexed by $\eta \in UGC_\lambda(n)$ (or $\phi \in UF_\lambda(n)$).

\noindent (viii)  Sections 12 and 10:  coincident pairs ($\mathcal{S}_\lambda(\beta), \mathcal{D}_\lambda(\pi)$) of sets of tableaux of shape $\lambda$ for upper $\lambda$-tuples $\beta \in U_\lambda(n)$ and $\lambda$-permutations $\pi \in S_n^\lambda$.

\noindent (ix)  Section 14:  Demazure polynomials $d_\lambda(\pi;x)$ indexed by $\pi \in S_n^{\lambda\text{-}312}$ that are distinct as polynomials, and gapless core Schur polynomials $s_\lambda(\eta;x)$ (or flag Schur polynomials $s_\lambda(\phi;x)$) indexed by $\eta \in UGC_\lambda(n)$ (or $\phi \in UF_\lambda(n)$) that are distinct as polynomials.

\noindent (x)  Section 14:  coincident pairs ($s_\lambda(\beta;x), d_\lambda(\pi;x)$) of polynomials indexed by upper $\lambda$-tuples $\beta \in U_\lambda(n)$ and $\lambda$-permutations $\pi \in S_n^\lambda$.

\noindent (xi)  Section 17:  valid upper $\lambda$-tuple inputs to the Gessel-Viennot determinant expressions for flag Schur polynomials on the shape $\lambda$ that attain maximum efficiency.  \end{thm}

\noindent As in Table 16.1, Item (vii) can be restated as:  generating functions $d_\lambda(\pi;x)$ for $\pi \in S_n^{\lambda\text{-}312}$ and $s_\lambda(\eta;x)$ for $\eta \in UGC_\lambda(n)$ (or $s_\lambda(\phi;x)$ for $\phi \in UF_\lambda(n)$) that are distinct within their respective collections in the sense of not being identical as generating functions.  Item (viii) can be similarly restated using the notion of the pairs of associated generating functions not being identical.

\begin{proof}Part (i) restates our $C_n^R$ definition with the terminology of \cite{GGHP}; for the second claim see the discussion below.  Use Proposition \ref{prop320.2}(ii), Corollary \ref{cor604.8}, Corollary \ref{cor608.6}, and Proposition \ref{prop623.2} to confirm (ii), the first part of (iii), and (iv).  For the second part of (iii), destrictify gapless $R$-tuples.  Use Proposition \ref{prop320.2}(i) and Theorem \ref{theorem340}(i) to confirm (v);  Part (vi) follows from Corollary \ref{cor520}.  Use the specialization of Fact \ref{fact420}(vi), Proposition \ref{prop623.4}, and Theorem \ref{theorem721}(iii) to confirm (vii) and (viii).  Use Proposition \ref{prop737}(ii), Corollary \ref{newcor737}, and Theorem \ref{theorem737.2} to confirm (ix) and (x).  Part (xi) is confirmed with Proposition \ref{prop623.4}(iii) and Corollary \ref{cor17.3}.  \end{proof}

To use the Online Encyclopedia of Integer Sequences \cite{Slo} to determine if the counts $C_n^R$ had been studied, we had to form sequences.  Define the \emph{total parabolic Catalan number $C_n^\Sigma$} to be $\sum C_n^R$, sum over $R \subseteq [n-1]$.  We also computed $C_n^\Sigma$ for small $n$ and found N.J.A. Sloane's 2013 contribution A226316.  These ``hits'' led us to the papers \cite{GGHP} and \cite{CDZ}.

Let $R$ be as in the theorem.  Let $2 \leq t \leq r+1$.  Fix a permutation $\sigma \in S_t$.  Apparently for the sake of generalization in and of itself with new enumeration results as a goal, Godbole, Goyt, Herdan and Pudwell defined \cite{GGHP} the notion of an ordered partition of $[n]$ with block sizes $b_1, b_2, ... , b_{r+1}$ that avoids the pattern $\sigma$.  It appears that that paper was the first paper to consider a notion of pattern avoidance for ordered partitions that can be used to produce our notion of $R$-312-avoiding permutations:  Take $b_1 := q_1$, $b_2 := q_2 - q_1$, ... , $b_{r+1} := n - q_r$, $t := 3$, and $\sigma := (3;1;2)$.  Their Theorem 4.1 implies that the number of such ordered partitions that avoid $\sigma$ is equal to the number of such ordered partitions that avoid each of the other five permutations for $t = 3$.  This can be used to confirm that the $C_{2m}^R$ sequence defined above is indeed Sequence A220097 of the OEIS (which is for avoiding the pattern 123).  Chen, Dai, and Zhou gave generating functions \cite{CDZ} in Theorem 3.1 and Corollary 2.3 for the $C_{2m}^R$ for $R = \{ 2, 4, 6, ... , 2m-2 \}$ for $m \geq 1$ and for the $C_n^\Sigma$ for $n \geq 1$.

How can the $C_n^\Sigma$ total counts be modeled?  Gathering the $R$-312-avoiding permutations or the $n$-tuples from Theorem \ref{theorem18.1}(ii) for this purpose would require retaining the ``semicolon dividers'' in those $R$-tuples.  Some other objects retain the information concerning $R$ more elegantly.  We omit definitions for some of the concepts in the next statement.  We also suspend our convention of omitting the prefix `$[n-1]$-':  Before, a `rightmost clump deleting' chain deleted one element at each stage.  Now this unadorned term describes a chain that deletes any number of elements in any number of stages, provided that they constitute entire clumps of the largest elements still present plus possibly a subset from the rightmost of the other clumps.  When $n = 3$ one has $C_n^\Sigma = 12$.  Five of these chains were displayed in Section 6.  A sixth is \cancel{1} \cancel{2} \cancel{3}.  Here are the other six, plus one such chain for $n = 17$:

\vspace{1pc}

\begin{figure}[h!]
\begin{center}
\setlength\tabcolsep{.1cm}
\begin{tabular}{ccccc}
1& &2& &\cancel{3}\\
 &\cancel{1}& &\cancel{2}
\end{tabular}\hspace{10mm}
\begin{tabular}{ccccc}
1& &\cancel{2}& &3\\
 &\cancel{1}& &\cancel{3}
\end{tabular}\hspace{10mm}
\begin{tabular}{ccccc}
\cancel{1}& &2& &3\\
 &\cancel{2}& &\cancel{3}
\end{tabular}\hspace{10mm}
\begin{tabular}{ccccc}
1& &\cancel{2}& &\cancel{3}\\
 & & \cancel{1}& &
\end{tabular}\hspace{10mm}
\begin{tabular}{ccccc}
\cancel{1}& &2& &\cancel{3}\\
 & & \cancel{2}
\end{tabular}\hspace{10mm}
\begin{tabular}{ccccc}
\cancel{1}& &\cancel{2}& &{3}\\
 & &\cancel{3}
\end{tabular}
\end{center}
\end{figure}

\vspace{-1pc}

\begin{figure}[h!]
\begin{center}
\setlength\tabcolsep{.3cm}
\begin{tabular}{ccccccccccccccccc}
1 & 2 & \cancel{3} & 4 & 5 & \cancel{6} & 7 & 8 & 9 & 10 & 11 & \cancel{12} & 13 & 14 & \cancel{15} & 16 & 17 \\
  &   & 1 & 2 & 4 & 5 & 7 & \cancel{8} & 9 & \cancel{10} & 11 & \cancel{13} & \cancel{14} & \cancel{16} & \cancel{17} & & \\
  &   &   &   &   & 1 & 2 & \cancel{4} & 5 & \cancel{7} & \cancel{9} & \cancel{11} & & & & & \\
  &   &   &   &   &   &   & \cancel{1} & \cancel{2} & \cancel{5} & & & & & & &
\end{tabular}
\end{center}
\end{figure}

\vspace{-1.5pc}

\begin{cor}\label{cor18.2}  The total parabolic Catalan number $C_n^\Sigma$ is the number of:

\noindent (i)  ordered partitions of $\{1, 2, ... , n \}$ that avoid the pattern 312.

\noindent (ii)  rightmost clump deleting chains for $[n]$, and gapless keys whose columns have distinct lengths less than $n$.

\noindent (iii)  for each $m \geq 1$, the flag Schur polynomials in $n$ variables on shapes with at most $n-1$ rows in which there are $m$ columns of each column length that is present.

\noindent (iv)  Schubert varieties in all of the flag manifolds $SL(n) / P_J$ for $J \subseteq [n-1]$ such that their ``associated'' Demazure tableaux form convex sets as in Section 11.  \end{cor}

\noindent Part (iv) highlights the fact that the convexity result of Corollary \ref{cor520} depends only upon the information from the indexing $R$-permutation for the Schubert variety, and not upon any further information from the partition $\lambda \in \Lambda_n^+$.  In addition to their count $op_n[(3;1;2)] = C_n^\Sigma$, the authors of \cite{GGHP} and \cite{CDZ} also considered the number $op_{n,k}(\sigma)$ of such $\sigma$-avoiding ordered partitions with $k$ blocks.  The models above can be adapted to require the presence of exactly $k$ blocks, albeit of unspecified sizes.

\vspace{1pc}\noindent \textbf{Added Note.}  We learned of the paper \cite{MW} after posting this paper on the arXiv.  As at the end of Section 8, let $R$ and $J$ be such that $R \cup J = [n-1]$ and $R \cap J = \emptyset$.  It could be interesting to compare the definition for what we would call an `$R$-231-avoiding' $R$-permutation (as in \cite{GGHP}) to M{\"u}hle's and Williams' definition of a `$J$-231-avoiding' $R$-permutation in Definition 5 of \cite{MW}.  There they impose an additional condition $w_i = w_k + 1$ upon the pattern to be avoided.  For their Theorems 21 and 24, this condition enables them to extend the notions of ``non-crossing partition'' and of ``non-nesting partition'' to the parabolic quotient $S_n / W_J$ context of $R$-permutations to produce sets of objects that are equinumerous with their $J$-231-avoiding $R$-permutations.  Their Theorem 7 states that this extra condition is superfluous when $J = \emptyset$.  In this case their notions of $J$-non-crossing partition and of $J$-non-nesting partition specialize to the set partition Catalan number models that appeared as Exercises 6.19(pp) and 6.19(uu) of \cite{St2} (or as Exercises 159 and 164 of \cite{St3}).  So if it is agreed that their reasonably stated generalizations of the notions of non-crossing and non-nesting partitions are the most appropriate generalizations that can be formulated for the $S_n / W_J$ context, then the mutual cardinality of their three sets of objects indexed by $J$ and $n$ becomes a competitor to our $C_n^R$ count for the name ``$R$-parabolic Catalan number''.  This development has made the obvious metaproblem more interesting:  Now not only must one determine whether each of the 214 Catalan models compiled in \cite{St3} is ``close enough'' to a pattern avoiding permutation interpretation to lead to a successful $R$-parabolic generalization, one must also determine which parabolic generalization fits the model at hand.

\vspace{1pc}\noindent \textbf{Acknowledgments.}  We thank Keith Schneider, Joe Seaborn, and David Raps for some helpful conversations, and we are also indebted to David Raps for some help with preparing this paper.  Feedback from Vic Reiner encouraged us to complete our analysis of the related results in \cite{RS} and of their relationship to our results.

\end{spacing}

\centerline{\includegraphics[scale=.99]{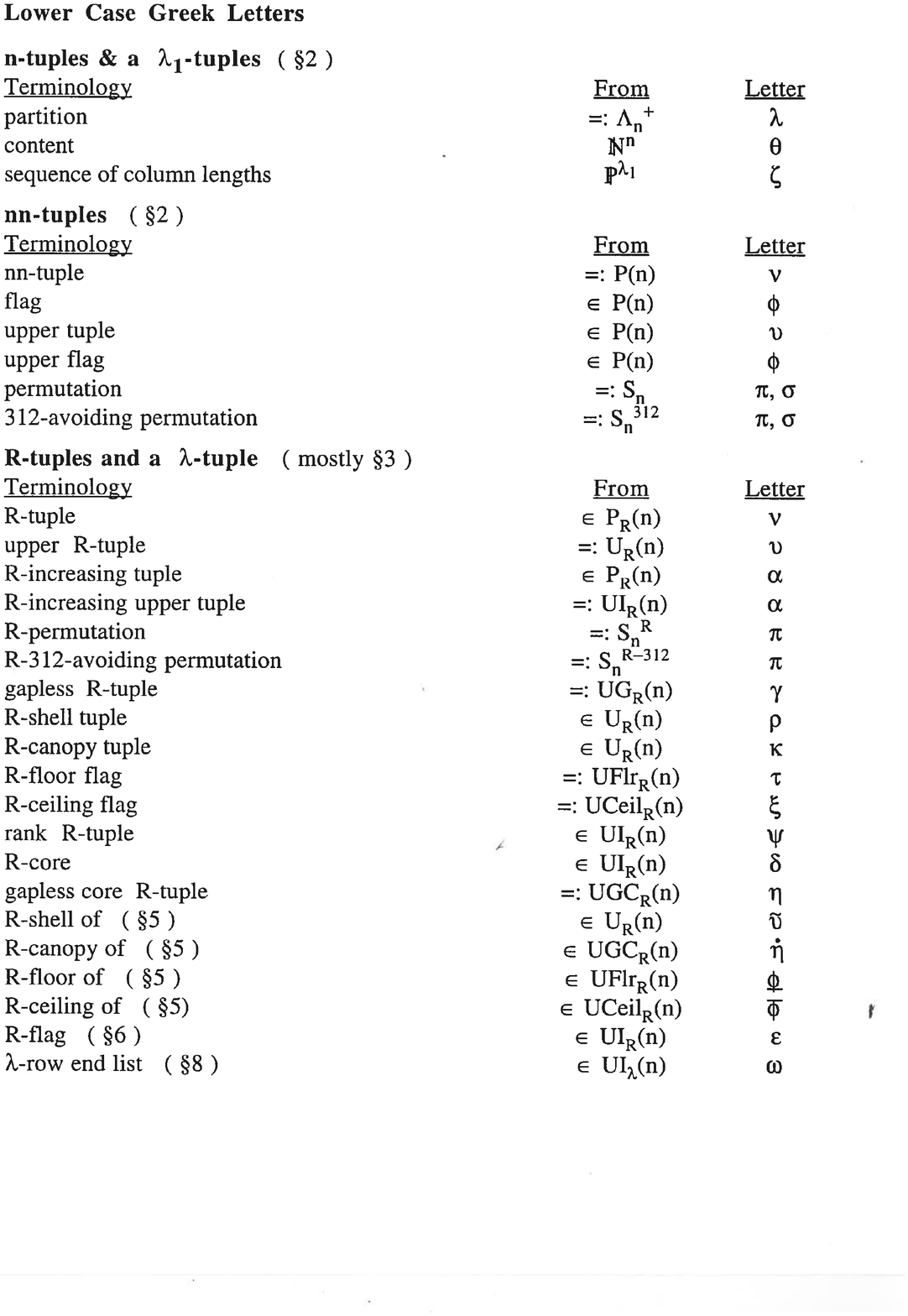}}

\centerline{\includegraphics[scale=.99]{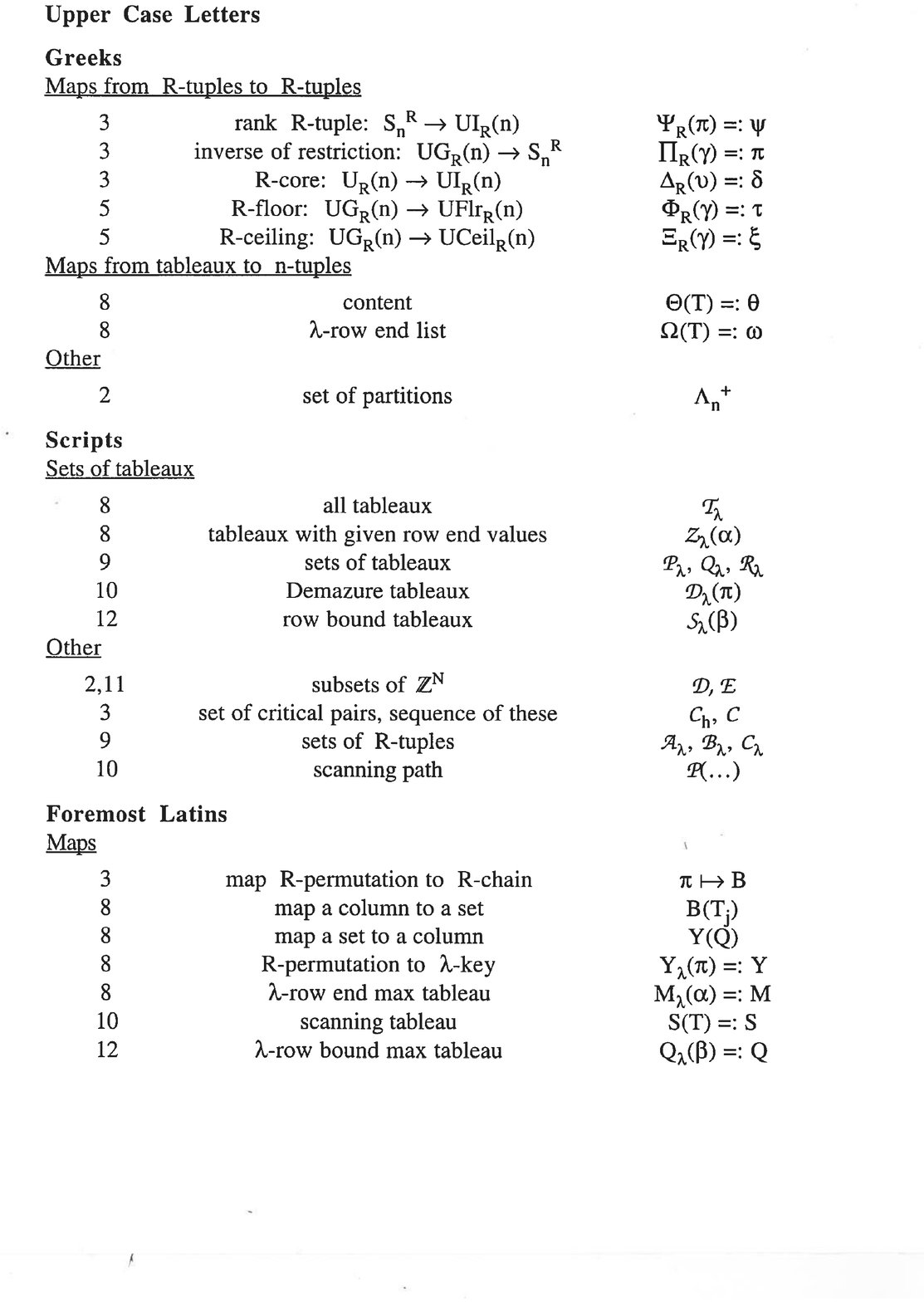}}

\end{document}